\documentclass[12pt, twoside]{article}
\renewcommand{\baselinestretch}{1.1}%
\usepackage{times}
\usepackage{bm}
\usepackage{graphics}
\usepackage{theorem}
\usepackage{graphicx}
\usepackage{amsmath}
\usepackage{latexsym}
\usepackage{amssymb,mathrsfs}
\usepackage[ruled]{algorithm2e}
\usepackage[flushmargin]{footmisc}

\theorembodyfont{\itshape}
\newtheorem{thm}{Theorem}[section]
\newtheorem{lem}{Lemma}[section]
\newtheorem{prop}{Proposition}[section]
\newtheorem{coro}{Corollary}[section]

\theorembodyfont{\rmfamily}
\newtheorem{cond}{Condition}
\newtheorem{defn}{Definition}[section]{\bf}{\rm}
\newtheorem{assumpt}{Assumption}[section]{\bf}{\rm}

\newtheorem{rem}{Remark}[section]{\itshape}{\rmfamily}
\newenvironment{proof}{\noindent{\it Proof.~~}}{\medskip}

\makeatletter
\def\eqnarray{\stepcounter{equation}\let\@currentlabel=\theequation
\global\@eqnswtrue
\global\@eqcnt\z@\tabskip\@centering\let\\=\@eqncr
$$\halign to \displaywidth\bgroup\@eqnsel\hskip\@centering
  $\displaystyle\tabskip\z@{##}$&\global\@eqcnt\@ne 
  \hfil$\;{##}\;$\hfil
  &\global\@eqcnt\tw@ $\displaystyle\tabskip\z@{##}$\hfil 
   \tabskip\@centering&\llap{##}\tabskip\z@\cr}
\makeatother
\makeatletter
    \renewcommand{\theequation}{%
    \thesection.\arabic{equation}}
    \@addtoreset{equation}{section}
  \makeatother

\newcommand{\vc}{\bm}

\makeatletter
\DeclareRobustCommand\widecheck[1]{{\mathpalette\@widecheck{#1}}}
\def\@widecheck#1#2{%
    \setbox\z@\hbox{\m@th$#1#2$}%
    \setbox\tw@\hbox{\m@th$#1%
       \widehat{%
          \vrule\@width\z@\@height\ht\z@
          \vrule\@height\z@\@width\wd\z@}$}%
    \dp\tw@-\ht\z@
    \@tempdima\ht\z@ \advance\@tempdima2\ht\tw@ \divide\@tempdima\thr@@
    \setbox\tw@\hbox{%
       \raise\@tempdima\hbox{\scalebox{1}[-1]{\lower\@tempdima\box
\tw@}}}%
    {\ooalign{\box\tw@ \cr \box\z@}}}
\makeatother
\def\wc#1{\widecheck{#1}}

\newcommand{\ol}{\overline}

\newcommand{\wt}{\widetilde}
\newcommand{\wh}{\widehat}

\newcommand{\ang}[1]{\hspace{-0.05em}\langle #1 \rangle}
\newcommand{\down}[2]{\smash{\lower#1\hbox{#2}}}
\newcommand{\up}[2]{\smash{\lower-#1\hbox{#2}}}

\newcommand{\dm}{\displaystyle}
\newcommand{\qed}{\hspace*{\fill}$\Box$}

\def\presub#1{\hspace{0.1em}{}_{#1}\hspace{-0.05em}}

\newcommand{\PP}{\mathsf{P}}

\newcommand{\bbA}{\mathbb{A}}
\newcommand{\bbB}{\mathbb{B}}
\newcommand{\bbC}{\mathbb{C}}

\newcommand{\bbL}{\mathbb{L}}

\newcommand{\bbN}{\mathbb{N}}

\newcommand{\bbS}{\mathbb{S}}

\newcommand{\bbZ}{\mathbb{Z}}

\newcommand{\diag}{\mathrm{diag}}
\newcommand{\abs}{\mathrm{abs}}

\newcommand{\rd}{{\rm d}}

\newcommand{\dd}[1]{\if#11 1\!\!1 
\else {\if#1C I\!\!\!C
\else {\if#1G I\!\!\!G 
\else {\if#1J J\!\!\!J 
\else {\if#1S S\!\!\!S
\else {\if#1Z Z\!\!\!Z
\else {\if#1Q O\!\!\!\!Q
\else I\!\!#1
\fi}
\fi}
\fi}
\fi} 
\fi} 
\fi} 
\fi} 

\makeatother

\begin{document}\thispagestyle{empty} 

\hfill

{\Large{\bf
\begin{center}
Limit formulas for the normalized fundamental matrix of the northwest-corner truncation of Markov chains: Matrix-infinite-product-form solutions of block-Hessenberg Markov chains%
\footnote[1]{
This research was supported in part by JSPS KAKENHI Grant Numbers JP15K00034.
}
%
%
\end{center}
}
}

\begin{center}
{
Hiroyuki Masuyama%
\footnote[2]{E-mail: masuyama@sys.i.kyoto-u.ac.jp}
}

\medskip

{\small
Department of Systems
Science, Graduate School of Informatics, Kyoto University\\
Kyoto 606-8501, Japan
}

\bigskip
\medskip

{\small
\textbf{Abstract}

\medskip

\begin{tabular}{p{0.85\textwidth}}
This paper considers the relation between the northwest-corner truncation and the stationary distribution vector of an ergodic (infinitesimal) generator (i.e., the infinitesimal generator of an ergodic continuous-time Markov chain). We first introduce the normalized fundamental matrix of the northwest-corner truncation, which is obtained by normalizing each row of the fundamental matrix. We then present some limit formulas associated with the normalized fundamental matrix. One of the limit formulas shows that the normalized fundamental matrix  converges, as its order (size) goes to infinity, to a stochastic matrix whose rows are all equal to the stationary distribution vector, though some technical conditions are required. This limit formula yields the matrix-infinite-product forms of the stationary distribution vectors of upper and lower block-Hessenberg Markov chains.  
\end{tabular}
}
\end{center}

\begin{center}
\begin{tabular}{p{0.90\textwidth}}
{\small

\medskip

}
\end{tabular}

\end{center}

\section{Introduction}\label{introduction}

Let $\{Z(t);t \ge 0\}$ denote an ergodic (i.e., irreducible and
positive-recurrent) continuous-time Markov chain with state space
$\bbZ_+:=\{0,1,2,\dots\}$. Let $\vc{Q}:=(q(i,j))_{i,j\in\bbZ_+}$
denote the infinitesimal generator (or $Q$-matrix; see, e.g.,
\cite[Section 2.1]{Ande91}) of the ergodic Markov chain $\{Z(t)\}$,
i.e., for all $i \in \bbZ_+$,
\[
\sum_{j\in\bbZ_+}q(i,j) = 0;\quad
q(i,i) \in (-\infty,0);\quad 
q(i,j) \ge 0,~j \in \bbZ_+ \setminus \{i\}.
\]
We refer to $\vc{Q}$ as the {\it ergodic (infinitesimal) generator}.
We then define $\vc{\pi}:=(\pi(i))_{i\in\bbZ_+}$ as a unique and
positive stationary distribution vector of the ergodic generator
$\vc{Q}$ (see, e.g., \cite[Chapter 5, Theorems 4.4 and 4.5]{Ande91}),
i.e., $\vc{\pi}$ is a positive vector such that $\vc{\pi}\vc{Q} =
\vc{0}$ and $\vc{\pi}\vc{e}=1$, where $\vc{e}$ denotes a column vector
of ones whose order depends on the context.

For $n \in \bbZ_+$, let $\presub{[n]}\vc{Q}:=(\presub{[n]}q(i,j))_{i,j\in
  \{0,1,\dots,n\}}$ denote the $(n+1) \times (n+1)$ northwest-corner
(NW-corner) truncation of the ergodic generator $\vc{Q}$, i.e.,
$\presub{[n]}q(i,j)=q(i,j)$ for all $i,j \in \bbZ_n:=\{0,1,\dots,n\}$. It should be noted that $\presub{[n]}\vc{Q}$ can be considered
the {\it transient} generator of an absorbing Markov chain with
transient states $\bbZ_n$ and absorbing states $\ol{\bbZ}_n :=\bbZ_+ \setminus
\bbZ_n$ (see, e.g., \cite[Chapter 8, Section 6.2]{Brem99}). The
absorbing Markov chain characterized by $\presub{[n]}\vc{Q}$ eventually
reaches the absorbing states from any transient state with probability
one due to the ergodicity of the original Markov chain $\{Z(t);t \ge
0\}$. Therefore, $(-\presub{[n]}\vc{Q})^{-1} \ge \vc{O}, \neq \vc{O}$ (see, e.g., \cite[Theorem 2.4.3]{Lato99}), where $\vc{O}$
denotes the zero matrix, which has an appropriate number of elements (depending on the context). The nonnegative matrix
$(-\presub{[n]}\vc{Q})^{-1}$ is called the {\it fundamental matrix} of
the transient generator $\presub{[n]}\vc{Q}$, which is the continuous-time
counterpart of the fundamental matrix defined for the discrete-time
absorbing Markov chains (see \cite[Chapter 4, Section 6]{Brem99} and
\cite[Chapter 5]{Keme76}).

We now define $\presub{[n]}\vc{F}$, $n\in\bbZ_+$, as
\begin{equation}
\presub{[n]}\vc{F}
= \diag^{-1}\{(- \presub{[n]}\vc{Q})^{-1} \vc{e}\}
(- \presub{[n]}\vc{Q})^{-1},
\label{defn-[n]F}
\end{equation}
where ${\rm diag}\{\,\cdot\,\}$ denotes the diagonal matrix whose
$i$-th diagonal element is equal to the $i$-th element of the vector
in the braces. It follows from (\ref{defn-[n]F}) that
$\presub{[n]}\vc{F} \ge \vc{O}$ and $\presub{[n]}\vc{F}\vc{e}=\vc{e}$,
i.e., $\vc{F}$ is row stochastic (stochastic, for short,
hereafter). We refer to $\presub{[n]}\vc{F}$ as the {\it normalized
  fundamental matrix} of $\presub{[n]}\vc{Q}$.

The main purpose of this paper is twofold. The first purpose is to
present a limit formula for the normalized
fundamental matrix $\presub{[n]}\vc{F}$:
\begin{equation}
\presub{[n]}\vc{F} 
\to \vc{e}\vc{\pi}\quad \mbox{as $n \to \infty$},
\label{limit-formula-02}
\end{equation}
which is derived from the
following formula (see Theorem~\ref{thm-lim-[n]F} below):
\begin{equation}
\presub{[n]}\ol{\vc{\pi}}:=
{ \presub{[n]}\vc{\alpha}(- \presub{[n]}\vc{Q})^{-1} 
\over 
\presub{[n]}\vc{\alpha}(- \presub{[n]}\vc{Q})^{-1} \vc{e}
} \to \vc{\pi}\quad \mbox{as $n \to \infty$},
\label{limit-formula}
\end{equation}
where $\presub{[n]}\vc{\alpha}$ is a $1 \times (n+1)$ probability
vector. The limit formula
(\ref{limit-formula}) requires its convergence conditions related to the
$\vc{f}$-modulated drift condition \cite[Section~14.2.1]{Meyn09} (which is often called the Foster-Lyapunov drift condition; see \cite{Meyn93-III}). Without such convergence conditions, we can also derive a limit formula for the normalized linear combination of the row subvectors of $(-\presub{[n]}\vc{Q})^{-1}$: For any fixed finite $\bbB \subsetneq \bbZ_+$,
\begin{equation}
{ \presub{[n]}\vc{\alpha}(- \presub{[n]}\vc{Q})^{-1}\presub{[n]}\vc{E}_{\bbB} 
\over 
\presub{[n]}\vc{\alpha}(- \presub{[n]}\vc{Q})^{-1}\presub{[n]}\vc{E}_{\bbB} \vc{e}
} \to {\vc{\pi}_{\bbB} \over \vc{\pi}_{\bbB}\vc{e}} \quad \mbox{as $n \to \infty$},
\label{limit-formula-submatrix}
\end{equation}
where $\vc{\pi}_{\bbA} := (\pi(i))_{i\in\bbA}$, $\bbA \in \bbZ_+$, is a subvector of $\vc{\pi}$, and where $\presub{[n]}\vc{E}_{\bbA}$, $\bbA \subseteq \bbZ_n$, is a matrix that can be permuted such that
\[
\presub{[n]}\vc{E}_{\bbA}
= \bordermatrix{
       & 
\bbA 
\cr
\phantom{\bbZ_n \setminus\ }\bbA & 
\vc{I} 
\cr
\bbZ_n \setminus \bbA & 
\vc{O}
}.
\]
Note here $\vc{I}$ denotes the identity matrix of an appropriate order. 

The second purpose is to derive, from the limit formulas
(\ref{limit-formula}) and/or (\ref{limit-formula-submatrix}), the
matrix-infinite-product (MIP) forms of the stationary distribution vectors in upper and lower block-Hessenberg Markov chains (upper and lower BHMCs, for short)\footnote{Upper and lower BHMCs are sometimes called level-dependent M/G/1-type and GI/M/1-type Markov chains, respectively.}. For convenience, we denote, by the MIP-form solution, the MIP form of the stationary distribution vector.

The first limit formula (\ref{limit-formula}) yields MIP-form solutions with {\it single limits} for both upper and lower BHMCs. For convenience, we refer to such MIP-form solutions as {\it single-limit MIP-form solutions}. The single-limit MIP-form solution of the upper BHMC requires some technical conditions (though that of the lower BHMC does not). Moreover, the second limit formula (\ref{limit-formula-submatrix}) yields another MIP-form solution of the upper BHMC, which has a double-limit expression but does not require any technical conditions. Such MIP-form solutions with double limits may be referred to as the {\it double-limit MIP solutions} to distinguish them from the single-limit MIP-form solutions.

It should be noted that he first limit formula (\ref{limit-formula}) is related to the {\it truncation approximation} of Markov chains. According to Gibson and Seneta~\cite{Gibs87-JAP}, the probability
vector $\presub{[n]}\ol{\vc{\pi}}$ in (\ref{limit-formula}) can be considered the {\it linearly augmented
  truncation approximation} to the stationary distribution vector
$\vc{\pi}$ of the original Markov chain $\{Z(t);t \ge 0\}$. Note that if 
\[
\presub{[n]}\vc{\alpha} =
(0,\dots,0,\overset{\mbox{\small $\nu$-th}}{1},0,\dots,0) 
=: \presub{[n]}\vc{e}_{\ang{\nu}}^{\top},
\]
then $\presub{[n]}\ol{\vc{\pi}}$ in (\ref{limit-formula}) is reduced to the {\it $\nu$-th-column-augmented truncation approximation} $\presub{[n]}\vc{\pi}_{\ang{\nu}}$ to $\vc{\pi}$. Especially, $\presub{[n]}\vc{\pi}_{\ang{0}}$ and $\presub{[n]}\vc{\pi}_{\ang{n}}$ are referred to as the first- and last-column-augmented truncation approximations, respectively.

For discrete-time ergodic
Markov chains, Wolf~\cite{Wolf80} discusses the convergence of several
augmented truncation approximations including the linearly augmented
one (see also \cite{Gibs87-JAP}). Wolf~\cite{Wolf80}'s results are
directly applicable to uniformizable continuous-time Markov chains
(see, e.g., \cite[Section~4.5.2]{Tijm03}).  As for the continuous-time
Markov chain, there are some studies on the convergence of augmented
truncation approximations. Hart and Tweedie~\cite{Hart12} prove that
$\lim_{n\to\infty}\presub{[n]}\vc{\pi}_{\ang{0}} = \vc{\pi}$ under the condition
that $\vc{Q}$ is exponentially ergodic. Hart and Tweedie~\cite{Hart12}
also assume that $\vc{Q}$ is (stochastically) monotone, under which
they proved the convergence of any augmented truncation approximation.
Masuyama \cite{Masu17-LAA} presents computable and convergent error
bounds for the last-column-{\it block}-augmented truncation approximation,
under the condition that $\vc{Q}$ is block monotone and exponentially
ergodic. Without block monotonicity, Masuyama~\cite{Masu17-JORSJ}
derives such convergent error bounds under the condition that $\vc{Q}$
satisfies the $\vc{f}$-modulated drift condition and some technical ones.

We now remark that the idea of the normalized fundamental matrix is
inspired by the studies of Shin~\cite{Shin09} and
Takine~\cite{Taki16}. Shin~\cite{Shin09} presents an algorithm for
computing the fundamental matrix of the transient generator of finite
level-dependent quasi-birth-and-death processes (LD-QBDs) with
absorbing states, based on matrix analytic methods
\cite{Gras00-BC,Lato99,Neut89}. The matrix analytic methods are the
foundation of many iterative algorithms
\cite{Baum12-Procedia,Baum12-COR,Brig95,Klim06,Phun10-QTNA,Shin98} for computing
the stationary distribution vectors of upper and lower
BHMCs (including LD-QBDs). These iterative
algorithms require to solve the system of linear equations
for the boundary probabilities. Unlike the existing algorithms, Takine~\cite{Taki16} proposes an algorithm for a special class of upper
BHMCs, which does not require to solve the system of linear equations for the boundary probabilities. The algorithm proposed in \cite{Taki16} computes the conditional stationary distribution vector of the levels below a given one, by using a
limit formula associated with the submatrix of $(- \presub{[n]}\vc{Q})^{-1}$, which is
related to our limit formula (\ref{limit-formula-submatrix}) but different from ours (for details,
see Section~\ref{subsec-upper-case}).

In this paper, we develop algorithm for computing the stationary distribution vectors of upper and lower BHMCs, based on their MIP-form solutions. As mentioned above, there are two types of the MIP-form solutions: the single-limit and double-limit MIP-form solutions. Using the single-limit MIP-form solutions, we establish numerically stable algorithms for upper and lower BHMCs. These algorithms do not require to determine the maximum
number of blocks (or levels) involved in computing. On the other hand, the existing
algorithms in \cite{Baum12-Procedia,Baum12-COR,Brig95,Klim06,Phun10-QTNA,Shin98}, as well as Takine's algorithm~\cite{Taki16}, 
require such input parameters, though it is, in general, difficult to
determine the parameters appropriately. 
This problem does not arise from our algorithms originating from the single-limit MIP-form solutions. 

We note that, as far as upper BHMCs are concerned, our algorithm (mentioned above) does not always stop within a finite number of iterations.
Thus, we develop an alternative algorithm for upper BHMCs by using the double-limit MIP-form solution, instead of the single-limit one. This alternative algorithm always stops after a finite number of iterations, though it requires  the maximum number of blocks involved in computing, as with the existing algorithms.

The rest of this paper is divided into two
sections. Section~\ref{sec-fundamental-matrix} derives the limit
formulas associated with the normalized fundamental matrix.
Section~\ref{sec-Hessenberg} presents the MIP-form solutions of
BHMCs and the algorithms for computing the
solutions.

\section{Normalized fundamental matrix of the NW-corner truncation}\label{sec-fundamental-matrix}

In this section, we first discuss the relation between the normalized fundamental matrix
$\presub{[n]}\vc{F}$ and the linearly augmented
truncation of the ergodic generator $\vc{Q}$. We then present limit formulas associated with $\presub{[n]}\vc{F}$. We also consider the normalized linear combination of the row subvectors of the fundamental matrix $(-\presub{[n]}\vc{Q})^{-1}$, which is related to the linearly augmented
truncation of $\vc{Q}$.

We now describe the notation used hereafter. We denote by $[\,\cdot\,]_{i,j}$ (resp.~$[\,\cdot\,]_i$), the $(i,j)$-th (rep.~$i$-th) element of the matrix (resp.~vector) in the square brackets. We may also extend, depending on
the context, a finite matrix (resp.~vector) to an infinite matrix by
appending an infinite number of zeros to the original matrix (resp.~vector) in such a
way that the existing elements remain in their original positions. For
example, when we write $\presub{[n]}\vc{Q} - \presub{[n+1]}\vc{Q} -
\vc{Q}$, we set $\presub{[n]}q(i,j) = 0$ for $(i,j) \in (\bbZ_+)^2 \setminus (\bbZ_n)^2$
and $\presub{[n+1]}q(i,j) = 0$ for $(i,j) \in (\bbZ_+)^2 \setminus (\bbZ_{n+1})^2$. Finally, we introduce the following notation: If
$\vc{H}_n:=(h_n(i,j))_{i,j \in \bbZ_n}$, $n \in \bbZ_+$, is a matrix
such that $\lim_{n\to\infty}h_n(i,j) = h(i,j)$ for $(i,j) \in
\bbZ_+$, then we represent this situation as $\lim_{n\to\infty}\vc{H}_n =
\vc{H}:=(h(i,j))_{i,j\in\bbZ_+}$. 

\subsection{Relation between the normalized fundamental matrix and linearly augmented truncation approximation}

For any $n \in \bbZ_+$, let
$\presub{[n]}\ol{\vc{Q}}:=(\presub{[n]}\ol{q}(i,j))_{i,j\in\bbZ_n}$ denote
\begin{equation}
\presub{[n]}\ol{\vc{Q}}
= \presub{[n]}\vc{Q} 
- \presub{[n]}\vc{Q}\vc{e} \presub{[n]}\vc{\alpha},
\label{defn-(n)ol{Q}}
\end{equation}
where $\presub{[n]}\vc{\alpha}:=(\presub{[n]}\alpha(i))_{i \in \bbZ_n}$ is
an arbitrary probability vector such that $\sum_{i=0}^n
\presub{[n]}\alpha(i) = 1$. Note here that $\presub{[n]}\vc{Q}\vc{e} \le
\vc{0}$ for all $n \in \bbZ_+$ and
$\lim_{n\to\infty}\presub{[n]}\vc{Q}\vc{e} = \vc{0}$. It thus follows
from (\ref{defn-(n)ol{Q}}) that
\begin{equation}
\lim_{n\to\infty}\presub{[n]}\ol{\vc{Q}} = \vc{Q},
\label{lim-(n)Q=Q}
\end{equation}
and that $\presub{[n]}\ol{\vc{Q}}\vc{e}=\vc{0}$ and $\presub{[n]}\ol{q}(i,j) \ge 0$ for $i \neq j$, which implies
that $\presub{[n]}\ol{\vc{Q}}$ is a {\it conservative} $Q$-matrix (see, e.g.,
\cite[Section 1.2]{Ande91}). Moreover, since $\presub{[n]}\ol{\vc{Q}}$ is finite, it has at least one stationary distribution vector. Indeed,
let $\presub{[n]}\ol{\vc{\pi}}:=(\presub{[n]}\ol{\pi}(i))_{i \in \bbZ_+}$,
denote
\begin{equation}
\presub{[n]}\ol{\vc{\pi}}
= 
{ \presub{[n]}\vc{\alpha}(- \presub{[n]}\vc{Q})^{-1} 
\over 
\presub{[n]}\vc{\alpha}(- \presub{[n]}\vc{Q})^{-1} \vc{e}
}, \qquad n \in \bbZ_+,
\label{eqn-(n)ol{pi}}
\end{equation}
which is a stationary distribution vector of $\presub{[n]}\ol{\vc{Q}}$.

We note that the $Q$-matrix
$\presub{[n]}\ol{\vc{Q}}$ is the continuous-time counterpart of the {\it linear-augmented truncation} of the transition probability matrix (see, e.g., \cite{Gibs87-JAP}). 
Thus, we refer to $\presub{[n]}\ol{\vc{Q}}$ as the {\it linear-augmented truncation} of
$\vc{Q}$. We also refer to $\presub{[n]}\vc{\alpha}$ as the {\it augmentation distribution vector} of
$\presub{[n]}\ol{\vc{Q}}$. Furthermore, we define $\presub{[n]}\vc{Q}_{\ang{\nu}}:=(\presub{[n]}q_{\ang{\nu}}(i,j))_{i,j\in\bbZ_n}$ as a linear-augmented truncation $\presub{[n]}\ol{\vc{Q}}$ with $\presub{[n]}\vc{\alpha} =
\presub{[n]}\vc{e}_{\ang{\nu}}^{\top}$, where $\presub{[n]}\vc{e}_{\ang{\nu}}$, $\nu
\in \bbZ_n$, denotes the $(n+1) \times 1$ unit vector whose $\nu$-th
element is equal to one. By definition,
\begin{equation}
\presub{[n]}\vc{Q}_{\ang{\nu}} 
= \presub{[n]}\vc{Q} - (\presub{[n]}\vc{Q}\vc{e}) \presub{[n]}\vc{e}_{\ang{\nu}}^{\top}.
\label{defn-(n)Q_l}
\end{equation}
We call $\presub{[n]}\vc{Q}_{\ang{\nu}}$ the {\it $\nu$-th-column-augmented
  $(n+1) \times (n+1)$ NW-corner truncation ($\nu$-th-column-augmented
  truncation, for short)} of $\vc{Q}$.

For any $n \in \bbZ_+$ and $\nu \in \bbZ_n$, let
$\presub{[n]}\vc{\pi}_{\ang{\nu}}:=(\presub{[n]}\pi_{\ang{\nu}}(i))_{i\in\bbZ_n}$
denote the stationary distribution vector of the
$\nu$-th-column-augmented truncation $\presub{[n]}\vc{Q}_{\ang{\nu}}$. It then
follows from (\ref{eqn-(n)ol{pi}}) and $\presub{[n]}\vc{\alpha} = \presub{[n]}\vc{e}_{\ang{\nu}}^{\top}$ that
\begin{eqnarray}
\presub{[n]}\vc{\pi}_{\ang{\nu}}
= { \presub{[n]}\vc{e}_{\ang{\nu}}^{\top}(- \presub{[n]}\vc{Q})^{-1} 
\over \presub{[n]}\vc{e}_{\ang{\nu}}^{\top} (- \presub{[n]}\vc{Q})^{-1} \vc{e}
}= \presub{[n]}\vc{e}_{\ang{\nu}}^{\top}\presub{[n]}\vc{F},\qquad  
n \in \bbZ_+,\ \nu \in \bbZ_n,
\label{eqn-(n)pi-02}
\end{eqnarray}
where the second equality holds due to (\ref{defn-[n]F}). Furthermore, (\ref{eqn-(n)pi-02}) leads to
\begin{equation*}
\presub{[n]}\vc{F} 
=
\left(
\begin{array}{c}
\presub{[n]}\vc{\pi}_{\ang{0}}
\\
\presub{[n]}\vc{\pi}_{\ang{1}}
\\
\vdots
\\
\presub{[n]}\vc{\pi}_{\ang{n}}
\end{array}
\right), \qquad  n \in \bbZ_+.
\end{equation*}

\subsection{Limit formulas for the normalized fundamental matrix}

For later use, we introduce the notation. Let $\bbN =
\{1,2,3,\dots\}$. For any row vector $\vc{\alpha}:=(\alpha(j))$ and
matrix $\vc{A}:=(a(i,j))$, let
\[
\|\vc{\alpha}\| = \sum_j | \alpha(j) |,
\qquad 
\|\vc{A}\| = \sup_i \sum_j | a(i,j) |,
\]
respectively. Let $\abs\{\,\cdot\,\}$ denote the
element-wise absolute operator for vectors and matrices.  Thus, $\abs\{\vc{\alpha}\}\vc{e} =
\|\vc{\alpha}\|$. 
For any set $\bbS \subseteq \bbZ_+$, let 
$\vc{1}_{\bbS}:=(1_{\bbS}(i))_{i\in\bbZ_+}$ denote a column
vector whose $i$-th element $1_{\bbS}(i)$ is given by
\[
1_{\bbS}(i)
=\left\{
\begin{array}{ll}
1, & \qquad i \in \bbS,
\\
0, & \qquad i \in \bbZ_+\setminus\bbS.
\end{array}
\right.
\]
Finally, we define the empty sum as zero, e.g.,
$\sum_{m=\ell}^k a_m = 0$ if $\ell > k$, where
$\{a_m;m=0,\pm1,\pm2,\dots\}$ are a sequence of numbers.

We now make two assumptions to show our limit formula for $\presub{[n]}\vc{F}$.
\begin{assumpt}\label{assumpt-1}
There exist some $b \in (0,\infty)$, column vector
$\vc{v}:=(v(i))_{i\in\bbZ_+} \ge \vc{e}$ and finite set $\bbC \subset
\bbZ_+$ such that
\begin{equation}
\vc{Q}\vc{v} 
\le  - \vc{q}^+ - \vc{e} + b \vc{1}_{\bbC},
\label{ineqn-Qv}
\end{equation}
where $\vc{q}^+=(|q(i,i)|)_{i\in\bbZ_+} > \vc{0}$
\end{assumpt}

\begin{rem}\label{assumpt-drift-condition}
Assumption~\ref{assumpt-1} is a special case of the $\vc{f}$-modulated drift condition for continuous-time Markov chains (see Condition~\ref{cond-f-modulated-drift} below), which is the continuous time counterpart of the drift condition described in \cite[Section~14.2.1]{Meyn09}.
Suppose that the generator $\vc{Q}$ is irreducible and regular
(non-explosive). It then follows from \cite[Theorem~1.1]{Kont16} that
Assumption~\ref{assumpt-1} holds if and only if $\vc{Q}$ is
ergodic with a unique stationary distribution vector $\vc{\pi}$ such
that $\vc{\pi}\vc{q}^+ < \infty$.
\end{rem}

\begin{assumpt}\label{assumpt-2}
The following hold:
\begin{align}
\presub{[n]}\ol{\vc{Q}}\vc{v} 
&\le - \vc{q}^+ - \vc{e} + b' \vc{1}_{\bbC},\qquad n \in \bbZ_+,
\label{ineqn-[n]Qv}
\\
\lim_{n\to\infty}\presub{[n]}\ol{\vc{Q}}\vc{v} 
&= \vc{Q}\vc{v}.
\label{potential-02}
\\
\sup_{i \in \bbZ_+ \setminus \bbC}&
\left\{ 
\sum_{j\in \bbZ_+ \setminus \bbC}
{q(i,j) \over |q(i,i)| + 1} (v(j) - v(i))^+
\right\} < \infty,
\label{potential-03}
\end{align}
where $(x)^+ = \max(x,0)$ for $x \in (-\infty,\infty)$.
\end{assumpt}

\begin{rem}
Suppose that $\vc{v}$, appearing in Assumption~\ref{assumpt-1}, satisfies the following:
\begin{eqnarray}
1 \le v(0) \le v(1) \le v(2) &\le& \cdots.
\label{nondecreasing-v}
\end{eqnarray}
From (\ref{defn-(n)ol{Q}}) and (\ref{nondecreasing-v}), we have $\presub{[n]}\ol{\vc{Q}}\vc{v} \le \vc{Q}\vc{v}$ for $n \in \bbZ_+$. Substituting this inequality into (\ref{ineqn-Qv}) yields (\ref{ineqn-[n]Qv}) with $b'=b$. Moreover, using (\ref{lim-(n)Q=Q}) and the dominated convergence theorem, we obtain (\ref{potential-02}). Therefore, Assumption~\ref{assumpt-1} together with (\ref{nondecreasing-v}) is sufficient for the first two conditions (\ref{ineqn-Qv}) and (\ref{potential-02}) of Assumption~\ref{assumpt-2}. Furthermore, we have, from (\ref{ineqn-Qv}),
\begin{eqnarray*}
&&
\sup_{i \in \bbZ_+ \setminus \bbC}
\left\{
\sum_{j\in \bbZ_+ \setminus \bbC}
{q(i,j) \over |q(i,i)| + 1} (v(j) - v(i))^+]
\right\}
\nonumber
\\
&& \quad
\le
\sup_{i \in \bbZ_+}
\left\{
\sum_{j=i+1}^{\infty}
{q(i,j) \over |q(i,i)| + 1}(v(j) - v(i))
\right\},
\end{eqnarray*}
which shows that the last the condition (\ref{potential-03}) holds if
\begin{equation*}
\sup_{i \in \bbZ_+}
\left\{
\sum_{j=i+1}^{\infty}
{q(i,j) \over |q(i,i)| + 1}(v(j) - v(i))
\right\} < \infty.
\end{equation*}
We can find a similar statement in \cite[Theorem~5.2]{Wolf80}. 
\end{rem}

\begin{rem}\label{remark-assumpt-f-drift}
The condition (\ref{potential-03}) implies
(\ref{potential-ol{P}-03}), which guarantees that a superharmonic
vector $\vc{v} \ge \vc{0}$ (see, e.g., \cite[Definition~4.1 and Lemma~4.4]{Wolf80}) of $\vc{P}_{\ol{\bbC}} :=
\diag\{\vc{1}_{\bbZ_+\setminus\bbC}\} \vc{P}
\diag\{\vc{1}_{\bbZ_+\setminus\bbC}\}$ (i.e.,
$\vc{P}_{\ol{\bbC}}\vc{v} \le \vc{v}$) is a potential of
$\vc{P}_{\ol{\bbC}}$, or equivalently,
\[
\vc{v} = \sum_{m=0}^{\infty} (\vc{P}_{\ol{\bbC}})^m 
\left( \vc{v} - \vc{P}_{\ol{\bbC}}\vc{v} \right).
\]
\end{rem}

\medskip

Under Assumptions~\ref{assumpt-1} and \ref{assumpt-2}, we obtain the following result.
\begin{thm}\label{thm-lim-[n]F}
If Assumptions~\ref{assumpt-1} and \ref{assumpt-2} hold, then
\begin{equation}
\lim_{n\to\infty} 
\| \presub{[n]}\ol{\vc{\pi}} - \vc{\pi} \| = 0,
\label{lim-[n]ol{pi}}
\end{equation}
\end{thm}

\begin{proof}
Let $\vc{P} := (p(i,j))_{i,j\in\bbZ_+}$ denote
\begin{equation}
\vc{P}  = \vc{I} + \vc{\Delta}^{-1}\vc{Q}.
\label{defn-wt{P}}
\end{equation}
where
\begin{equation}
\vc{\Delta} = \diag\{\vc{q}^+ + \vc{e}\}.
\label{defn-Delta}
\end{equation}
Since the generator $\vc{Q}$ is ergodic, $\vc{P}$ is an irreducible
stochastic matrix. Let $\vc{\varpi}$ denote
\begin{equation}
\vc{\varpi} 
= {
\vc{\pi}\vc{\Delta} 
\over \vc{\pi}\vc{\Delta}\vc{e}
},
\label{eqn-wt{pi}}
\end{equation}
which is well-defined due to Assumption~\ref{assumpt-1} (see
Remark~\ref{assumpt-drift-condition}). From (\ref{defn-wt{P}}),
(\ref{eqn-wt{pi}}) and $\vc{\pi}\vc{Q}=\vc{0}$, we have
$\vc{\varpi}\vc{P} = \vc{\varpi}$ and thus $\vc{\varpi}$ is a unique
stationary distribution vector of $\vc{P}$. From (\ref{eqn-wt{pi}}),
we also have
\begin{equation}
\vc{\pi} 
= {
\vc{\varpi} \vc{\Delta}^{-1} 
\over \vc{\varpi} \vc{\Delta}^{-1}\vc{e}
}.
\label{eqn-pi-wt{pi}}
\end{equation}

Let $\presub{[n]}\vc{P}$, $n \in \bbZ_+$, denote the $(n+1) \times
(n+1)$ NW-corner truncation of $\vc{P}$, i.e.,
\begin{equation}
\presub{[n]}\vc{P} =
\vc{I} + \presub{[n]}\vc{\Delta}^{-1}\presub{[n]}\vc{Q},
\label{defn-(n)wt{P}}
\end{equation}
where $\presub{[n]}\vc{\Delta}$ denotes the $(n+1) \times (n+1)$
NW-corner truncation of $\vc{\Delta}$. We then define $\presub{[n]}\ol{\vc{P}}$, $n\in\bbZ_+$, as
\begin{equation}
\presub{[n]}\ol{\vc{P}} 
= \presub{[n]}\vc{P} + (\vc{I} - \presub{[n]}\vc{P})\vc{e} \presub{[n]}\vc{\alpha},
\label{defn-[n]ol{P}}
\end{equation}
We also define $\presub{[n]}\ol{\vc{\varpi}}$, $n \in \bbZ_+$, as the stationary distribution vector of
$\presub{[n]}\vc{P}_{\ang{\nu}}$. It is known \cite[Lemma 7.2]{Sene06} that
\begin{eqnarray}
\presub{[n]}\ol{\vc{\varpi}}
&=& { \presub{[n]}\vc{\alpha}(\vc{I} - \presub{[n]}\vc{P})^{-1}
\over \presub{[n]}\vc{\alpha}(\vc{I} - \presub{[n]}\vc{P})^{-1}\vc{e}}.
\label{defn-[n]ol{varpi}}
\end{eqnarray}
By (\ref{defn-(n)wt{P}}), we rewrite (\ref{defn-[n]ol{P}}) and 
(\ref{defn-[n]ol{varpi}}) as
\begin{eqnarray}
\presub{[n]}\ol{\vc{P}} 
&=& \vc{I} + \presub{[n]}\vc{\Delta}^{-1}
(\presub{[n]}\vc{Q} - \presub{[n]}\vc{Q} \vc{e} \presub{[n]}\vc{\alpha})
= \vc{I} + \presub{[n]}\vc{\Delta}^{-1}\presub{[n]}\ol{\vc{Q}},
\label{eqn-[n]ol{P}}
\\
\presub{[n]}\ol{\vc{\varpi}}
&=& 
{ \presub{[n]}\vc{\alpha}(- \presub{[n]}\vc{Q})^{-1}\presub{[n]}\vc{\Delta} 
\over \presub{[n]}\vc{\alpha} ( - \presub{[n]}\vc{Q})^{-1} \presub{[n]}\vc{\Delta}\vc{e}
}.
\label{eqn-(n)ol{varpi}}
\end{eqnarray}
From (\ref{eqn-(n)ol{pi}}) and (\ref{eqn-(n)ol{varpi}}), we also have
\begin{equation}
\presub{[n]}\ol{\vc{\pi}}
= {
\presub{[n]}\ol{\vc{\varpi}} \vc{\Delta} 
\over \presub{[n]}\ol{\vc{\varpi}} \vc{\Delta} \vc{e}
}.
\label{eqn-(n)ol{pi}-02}
\end{equation}
Furthermore, by (\ref{defn-wt{P}}) and (\ref{eqn-[n]ol{P}}), we rewrite the conditions (\ref{ineqn-Qv})--(\ref{potential-03}) as
\begin{eqnarray*}
\vc{P}\vc{v} 
&\le& \vc{v} - \vc{e} + b  \vc{1}_{\bbC},
\label{potential-ol{P}-00}
\\
\presub{[n]}\ol{\vc{P}}\vc{v} 
&\le&  \vc{v}- \vc{e} + b'  \vc{1}_{\bbC},\qquad n \in \bbZ_+,
\label{potential-ol{P}-01}
\\
\lim_{n\to\infty}\presub{[n]}\ol{\vc{P}}\vc{v} 
&=& \vc{P}\vc{v},
\label{potential-ol{P}-02}
\end{eqnarray*}
and 
\begin{eqnarray}
\sup_{i \in \bbZ_+ \setminus \bbC}
\left\{ 
\sum_{j\in \bbZ_+ \setminus \bbC}
p(i,j) (v(j) - v(i))^+
\right\} &<& \infty.
\label{potential-ol{P}-03}
\end{eqnarray}
Therefore, Corollary~4.5 of \cite{Wolf80} implies that
\begin{equation}
\lim_{n\to\infty}\|\presub{[n]}\ol{\vc{\varpi}} - \vc{\varpi}\| 
= 0.
\label{eqn-lim-[n]ol{pi}}
\end{equation}

We are now ready to prove (\ref{lim-[n]ol{pi}}). Let $\vc{d} = \vc{\Delta}^{-1}\vc{e}$. It then follows from
(\ref{eqn-(n)pi-02}), (\ref{eqn-pi-wt{pi}}) and (\ref{eqn-(n)ol{pi}-02})  that
\begin{eqnarray}
\presub{[n]}\ol{\vc{\pi}} - \vc{\pi}
&=& {
\presub{[n]}\ol{\vc{\varpi}} \vc{\Delta}^{-1} 
\over 
\presub{[n]}\ol{\vc{\varpi}} \vc{d} 
}
-
{
\vc{\varpi} \vc{\Delta}^{-1} 
\over 
\vc{\varpi} \vc{d} 
}
\nonumber
\\
&=& \left[
{1 \over \presub{[n]}\ol{\vc{\varpi}} \vc{d} } 
(\presub{[n]}\ol{\vc{\varpi}} - \vc{\varpi} )
+ 
\left( 
{1 \over \presub{[n]}\ol{\vc{\varpi}} \vc{d} }
-
{1 \over \vc{\varpi} \vc{d} }
\right)
\vc{\varpi} 
\right]
 \vc{\Delta}^{-1}
\nonumber
\\
&=&
{1 \over \presub{[n]}\ol{\vc{\varpi}} \vc{d} }
\left[
( \presub{[n]}\ol{\vc{\varpi}} - \vc{\varpi} )
+
( \vc{\varpi} - \presub{[n]}\ol{\vc{\varpi}} )
\vc{d}
{ 
\vc{\varpi}
\over 
\vc{\varpi} \vc{d}
}
\right] \vc{\Delta}^{-1},
\label{eqn-170312-01}
\end{eqnarray}
where, as mentioned in the beginning of this section, the finite
probability vectors $\presub{[n]}\vc{\pi}_{\ang{\nu}}$ and
$\presub{[n]}\vc{\varpi}_{\ang{\nu}}$ are extended to the infinite ones by
appending zeros to these vectors.  From (\ref{eqn-170312-01}), we have
\begin{eqnarray}
\| \presub{[n]}\ol{\vc{\pi}} - \vc{\pi} \|
&=&
\abs\{ \presub{[n]}\ol{\vc{\pi}} - \vc{\pi} \} \vc{e}
\nonumber
\\
&\le&
{1 \over \presub{[n]}\ol{\vc{\varpi}} \vc{d} }
\left[
\abs\{ \presub{[n]}\ol{\vc{\varpi}} - \vc{\varpi} \}
+
\abs\{ \vc{\varpi} - \presub{[n]}\ol{\vc{\varpi}} \}
\vc{d}
{ 
\vc{\varpi}
\over 
\vc{\varpi} \vc{d}
}
\right] \vc{d}
\nonumber
\\
&=&
{2 \cdot \abs\{ \presub{[n]}\ol{\vc{\varpi}} - \vc{\varpi} \}
\vc{d}
\over \presub{[n]}\ol{\vc{\varpi}} \vc{d} }.
\label{eqn-pi-hat{pi}-02}
\end{eqnarray}

We discuss the convergence of the right-hand side of
(\ref{eqn-pi-hat{pi}-02}).  It follows from (\ref{defn-Delta}) that
\begin{equation}
\vc{d} = \vc{\Delta}^{-1}\vc{e} \le \vc{e}.
\label{ineqn-q^*}
\end{equation}
It also follows from
(\ref{ineqn-q^*}) that $\presub{[n]}\ol{\vc{\varpi}} \vc{d} \le
\presub{[n]}\ol{\vc{\varpi}}\vc{e} = 1$ for $n \in \ol{\bbZ}_{\nu-1}$. Therefore, using (\ref{eqn-lim-[n]ol{pi}}) and the
dominated convergence theorem, we obtain
\begin{equation}
\lim_{n\to\infty} \presub{[n]}\ol{\vc{\varpi}} \vc{d} 
=
\vc{\varpi}\vc{d} \in (0,\infty)\quad \mbox{for any fixed $\nu\in\bbZ_+$}.
\label{limit-(n)wt{pi}_{nu}-01}
\end{equation}
In addition, using (\ref{eqn-lim-[n]ol{pi}}) and (\ref{ineqn-q^*}),
we have
\begin{eqnarray}
\abs\{ \presub{[n]}\ol{\vc{\varpi}} - \vc{\varpi} \}
\vc{d}
&\le& \abs\{ \presub{[n]}\ol{\vc{\varpi}} - \vc{\varpi} \} \vc{e}
\nonumber
\\
&=& \| \presub{[n]}\ol{\vc{\varpi}} - \vc{\varpi}  \| \to 0
\quad
\mbox{as $n \to \infty$},
\label{limit-(n)wt{pi}_{nu}-02}
\end{eqnarray}
where the limit holds for any fixed $\nu\in\bbZ_+$. 
Applying (\ref{limit-(n)wt{pi}_{nu}-01}) and
(\ref{limit-(n)wt{pi}_{nu}-02}) to (\ref{eqn-pi-hat{pi}-02}) yields (\ref{lim-[n]ol{pi}}).\qed
\end{proof}

\begin{coro}\label{coro-limit-[n]F}
Suppose that Assumption~\ref{assumpt-1} and (\ref{potential-03}) are satisfied. Furthermore, suppose that for each $n \in \bbN$ and $\nu \in \bbZ_n$ there exists some $b_{n,\nu} \in (0,\infty)$ such that
\begin{equation}
\presub{[n]}\vc{Q}_{\nu} \vc{v} 
\le - \vc{q}^+ - \vc{e} + b_{n,\nu} \vc{1}_{\bbC}.
\label{ineqn-[n]Q_{nu}v}
\end{equation}
We then have
\begin{equation}
\lim_{n\to\infty} 
\presub{[n]}\vc{F}
= \lim_{n\to\infty} \left(
\begin{array}{c}
\presub{[n]}\vc{\pi}_{\ang{0}}
\\
\presub{[n]}\vc{\pi}_{\ang{1}}
\\
\vdots
\\
\presub{[n]}\vc{\pi}_{\ang{n}}
\end{array}
\right)
= \vc{e}\vc{\pi}.
\label{limit-formula-(n)F}
\end{equation}
\end{coro}

\begin{proof}
We can see that, for all $n \in \bbN$ and $\nu \in \bbZ_n$,
\[
\lim_{n\to\infty}\presub{[n]}\vc{Q}_{\nu}\vc{v} 
=\vc{Q}\vc{v}.
\]
Thus, Assumption~\ref{assumpt-2} holds with $\presub{[n]}\ol{\vc{Q}}=\presub{[n]}\vc{Q}_{\nu}$ for all $n \in \bbN$ and $\nu \in \bbZ_n$. Consequently, Theorem~\ref{thm-lim-[n]F} implies this corollary.
\end{proof}

\medskip

Corollary~\ref{coro-limit-[n]F} presents the limit formula for the {\it whole} $\presub{[n]}\vc{F}$ and thus requires the technical conditions originated from Assumption~\ref{assumpt-2}. However, without these technical condition, we obtain a limit formula for a {\it single and fixed} row of $\presub{[n]}\vc{F}$.
\begin{thm}\label{thm-limit-row-[n]F}
If Assumption~\ref{assumpt-1} holds, then
\begin{equation}
\lim_{n\to\infty} 
\| \presub{[n]}\vc{e}_{\ang{\nu}}^{\top}\presub{[n]}\vc{F} - \vc{\pi} \| = 0
\quad\mbox{for any fixed $\nu \in \bbZ_+$}.
\label{lim-(n)e*(n)F}
\end{equation}
\end{thm}

\begin{proof}
For $n
\in \bbZ_+$, let $\presub{[n]}\ol{\vc{\varpi}}_{\ang{\nu}}$ and
$\presub{[n]}\ol{\vc{P}}_{\ang{\nu}}$, $n \in \ol{\bbZ}_{\nu-1}$, denote
\begin{eqnarray}
\presub{[n]}\ol{\vc{\varpi}}_{\ang{\nu}}
&=& { 
\presub{[n]}\vc{e}_{\ang{\nu}}^{\top} (\vc{I} - \presub{[n]}\vc{P})^{-1}
\over 
\presub{[n]}\vc{e}_{\ang{\nu}}^{\top} (\vc{I} - \presub{[n]}\vc{P})^{-1}\vc{e}},\label{defn-[n]ol{varpi}_{nu}}
\\
\presub{[n]}\ol{\vc{P}}_{\ang{\nu}} 
&=& \presub{[n]}\vc{P} + (\vc{I} - \presub{[n]}\vc{P})\vc{e} 
\presub{[n]}\vc{e}_{\ang{\nu}}^{\top},
\label{defn-[n]ol{P}_{nu}}
\end{eqnarray}
where $\presub{[n]}\vc{P}$ is given in (\ref{defn-(n)wt{P}}), i.e.,
$\presub{[n]}\vc{P}$ is the $(n+1) \times (n+1)$ NW-corner of $\vc{P}$
in (\ref{defn-wt{P}}).  It then follows that
$\presub{[n]}\ol{\vc{\varpi}}$ is a stationary distribution vector of
$\presub{[n]}\ol{\vc{P}}$. Substituting (\ref{defn-(n)wt{P}}) into
(\ref{defn-[n]ol{varpi}_{nu}}) yields
\begin{equation*}
\presub{[n]}\ol{\vc{\varpi}}_{\ang{\nu}}
= 
{ \presub{[n]}\vc{e}_{\ang{\nu}}^{\top}
(- \presub{[n]}\vc{Q})^{-1}\presub{[n]}\vc{\Delta} 
\over 
\presub{[n]}\vc{e}_{\ang{\nu}}^{\top} 
( - \presub{[n]}\vc{Q})^{-1} \presub{[n]}\vc{\Delta}\vc{e}
}.
\end{equation*}
From this and (\ref{eqn-(n)ol{pi}}), we have
\begin{equation*}
\presub{[n]}\ol{\vc{\pi}}
= {
\presub{[n]}\ol{\vc{\varpi}} \,\presub{[n]}\vc{\Delta}^{-1}
\over 
\presub{[n]}\ol{\vc{\varpi}} \,\presub{[n]}\vc{\Delta}^{-1}\vc{e}
},\qquad n \in \bbZ_+.
\end{equation*}
In addition, by (\ref{defn-(n)ol{Q}}) and (\ref{defn-(n)wt{P}}), we rewrite
(\ref{defn-[n]ol{P}_{nu}}) as
\begin{eqnarray}
\presub{[n]}\ol{\vc{P}}_{\ang{\nu}}
&=& \vc{I} + \presub{[n]}\vc{\Delta}^{-1}
(\presub{[n]}\vc{Q} - \presub{[n]}\vc{Q} \vc{e} \presub{[n]}\vc{\alpha})
\nonumber
\\
&=&  \vc{I} + \presub{[n]}\vc{\Delta}^{-1}\presub{[n]}\ol{\vc{Q}}.
\label{eqn-(n)ol{P}}
\end{eqnarray}

In fact, it is known \cite[Theorem 5.1]{Wolf80} (see also \cite[Theorem
  3.1]{Gibs87-JAP}) that
$\lim_{n\to\infty}\|\presub{[n]}\ol{\vc{\varpi}}_{\ang{\nu}} - \vc{\varpi}\| =
0$. Therefore, using the above results, and following the proof of
Theorem~\ref{thm-lim-[n]F}, we can readily prove that
$\lim_{n\to\infty}\|\presub{[n]}\ol{\vc{\pi}} - \vc{\pi}\| = 0$. \qed
\end{proof}

\begin{coro}\label{coro-limit-[n]ol{pi}}
Suppose that, for all $n \in \bbZ_+$, the augmentation distribution vector $\presub{[n]}\vc{\alpha}$ satisfies 
\[
\presub{[n]}\vc{\alpha} = (\alpha_0,\alpha_1,\dots\alpha_m,0,0,\dots,0),
\]
where $m \in \bbZ_n$ is fixed arbitrarily and independently of $n$.  We then have
\[
\lim_{n\to\infty} 
\| \presub{[n]}\ol{\vc{\pi}} - \vc{\pi} \| = 0.
\]
\end{coro}

\begin{proof}
In the present setting, (\ref{eqn-(n)ol{pi}}) and (\ref{eqn-(n)pi-02}) yield
\begin{eqnarray}
\presub{[n]}\ol{\vc{\pi}} 
&=& { 
\sum_{\nu=1}^m \alpha_{\nu} \presub{[n]}\vc{e}_{\ang{\nu}}^{\top}
(-\presub{[n]}\vc{Q})^{-1}
\over  
\sum_{j=1}^m \alpha_{j} \presub{[n]}\vc{e}_{\ang{j}}^{\top}
(-\presub{[n]}\vc{Q})^{-1}\vc{e}
}
\nonumber
\\
&=& \sum_{\nu=1}^m \presub{[n]}\vc{\pi}_{\ang{\nu}}
{ 
\alpha_{\nu}\presub{[n]}\vc{e}_{\ang{\nu}}^{\top}
(-\presub{[n]}\vc{Q})^{-1}\vc{e}
\over  
\sum_{j=1}^m \alpha_{j} \presub{[n]}\vc{e}_{\ang{j}}^{\top}
(-\presub{[n]}\vc{Q})^{-1}\vc{e}
}
= \sum_{\nu=1}^m \gamma_{\nu} \presub{[n]}\vc{\pi}_{\ang{\nu}},
\label{eqn-170709-01}
\end{eqnarray}
where
\begin{equation*}
\gamma_{\nu}
= { 
\alpha_{\nu}\presub{[n]}\vc{e}_{\ang{\nu}}^{\top}
(-\presub{[n]}\vc{Q})^{-1}\vc{e}
\over  
\sum_{j=1}^m \alpha_{j} \presub{[n]}\vc{e}_{\ang{j}}^{\top}
(-\presub{[n]}\vc{Q})^{-1}\vc{e}
},\qquad \nu \in \bbZ_m.
\end{equation*}
Note here that $\sum_{\nu=1}^m\gamma_{\nu}=1$. Thus, applying Theorem~\ref{thm-limit-row-[n]F} to (\ref{eqn-170709-01}), we have
\begin{eqnarray*}
\lim_{n\to\infty}
\| \presub{[n]}\ol{\vc{\pi}} - \vc{\pi} \|
&\le& \sum_{\nu=1}^m\gamma_{\nu}
\lim_{n\to\infty} 
\| \presub{[n]}\vc{\pi}_{\ang{\nu}} - \vc{\pi} \| = 0.
\end{eqnarray*}
The proof is completed \qed.
\end{proof}

\begin{rem}
Theorem~\ref{thm-limit-row-[n]F} does not necessarily implies that if
Assumption~\ref{assumpt-1} holds then, for any fixed $m \in
\bbZ_+$,
\begin{equation*}
\lim_{n\to\infty} 
\| \presub{[n+m]}\vc{e}_{\ang{n}}^{\top} \, \presub{[n+m]}\vc{F} - \vc{\pi} \| = 0,
\label{lim-(n+m)e*(n+m)F}
\end{equation*}
or equivalently, $\lim_{n\to\infty} \| \presub{[n+m]}\vc{\pi}_{\ang{n}} -
\vc{\pi} \| = 0$.  In fact, although this is the discrete-time case,
Gibson and Seneta~\cite{Gibs87-JAP} provide an example such that the
last-column-augmented truncation approximation $\presub{[n]}\vc{\pi}_{\{n\}}$
does not converge to $\vc{\pi}$ as $n \to \infty$ (see also
\cite{Wolf80}). 
\end{rem}

\subsection{Limit formula for the normalized linear combination of the row-subvectors of the fundamental matrix}

In this subsection, we consider the normalized linear combination of the row-subvectors of the fundamental matrix $(-\presub{[n]}\vc{Q})^{-1}$. To this end,  fix a finite $\bbB \subset \bbZ_+$ arbitrarily, and let $B^* = \max\{i\in\bbB\}$. For $n \ge B^*$, let $\bbB_n^+$ denote the index set of the nonzero rows (if any) of $(-\presub{[n]}\vc{Q})^{-1}\vc{E}_{\bbB}$, i.e.,
\begin{eqnarray}
\bbB_n^+ 
&=& \{\nu\in\bbZ_n; \presub{[n]}\vc{e}_{\nu}^{\top} (-\presub{[n]}\vc{Q})^{-1}\presub{[n]}\vc{E}_{\bbB}\vc{e} > 0\}
\nonumber
\\
&=& \left\{
\nu\in\bbZ_n; 
\sum_{j\in\bbB}\left[ (-\presub{[n]}\vc{Q})^{-1} \right]_{\nu,j}  > 0 \right\}.
\label{defn-B_n^+-02}
\end{eqnarray}
For convenience, let $\bbB_n^+  = \varnothing$ for $n < B^*$.

\begin{lem}\label{lem-B_n^+}
The following hold:
\begin{eqnarray}
\bbB_n^+ &\subseteq& \bbB_m^+, \qquad 
\mbox{if~~$B^* \le n \le m$},
\label{eqn-B_n^+-B_m^+}
\\
\bigcup_{n=0}^{\infty} \bbB_n^+ &\supseteq& \bbB.
\label{eqn-cup-B_n^+}
\end{eqnarray} 
\end{lem}

\begin{proof}
Let $\vc{P}^{(t)}:=(p^{(t)}(i,j))_{i,j\in\bbZ_+}$ denote the
transition matrix function of the Markov chain $\{Z(t)\}$ with ergodic
generator $\vc{Q}$, i.e., $\PP(Z(t) = j \mid Z(0) = i)$ for all
$i,j\in\bbZ_+$. It then follows from \cite[Chapter 2, Proposition
  2.14]{Ande91} that, for all $i,j \in \bbZ_+$ and $t \ge 0$,
\[
 \left[ \exp\{\presub{[n]}\vc{Q} t\}  \right]_{i,j}
\nearrow \left[ \vc{P}^{(t)}  \right]_{i,j} > 0
\quad \mbox{as $n \to \infty$}.
\]
Thus, using the monotone convergence theorem, we have, for all $i,j \in \bbZ_+$,
\begin{eqnarray}
\lefteqn{
\left[ (-\presub{[n]}\vc{Q})^{-1} \right]_{i,j}
}
\quad &&
\nonumber
\\
&=& \left[ \int_0^{\infty} \exp\{\presub{[n]}\vc{Q} t\} \rd t  \right]_{i,j}
\nearrow \left[ \int_0^{\infty} \vc{P}^{(t)} \rd t  \right]_{i,j}=\infty
\quad \mbox{as $n \to \infty$}. \qquad
\label{lim-(-(n)Q)^{-1}]}
\end{eqnarray}
From (\ref{defn-B_n^+-02}) and (\ref{lim-(-(n)Q)^{-1}]}), we have (\ref{eqn-B_n^+-B_m^+}). Recall that the finite set $\bbB$ is fixed independently of $n$. It thus follows from (\ref{lim-(-(n)Q)^{-1}]}) that, for all sufficiently large $n \ge B^*$, 
\[
\left[ (-\presub{[n]}\vc{Q})^{-1} \right]_{i,j} > 0\quad \mbox{for all $i,j \in \bbB$},
\]
which implies that (\ref{eqn-cup-B_n^+}) holds.
\qed
\end{proof}

We now define the normalized linear combination of the row-subvectors of $(-\presub{[n]}\vc{Q})^{-1}$.
Let $\presub{[n]}\ol{\vc{\pi}}_{\bbB}^{\ast}$, $n \ge B^*$, denote 
\begin{equation}
\presub{[n]}\ol{\vc{\pi}}_{\bbB}^{\ast}
= {
\presub{[n]}\vc{\alpha}_{\bbB}
(-\presub{[n]}\vc{Q})^{-1} \presub{[n]}\vc{E}_{\bbB}
\over
\presub{[n]}\vc{\alpha}_{\bbB}
(-\presub{[n]}\vc{Q})^{-1} \presub{[n]} \vc{E}_{\bbB}\vc{e}
},
\label{defn-[n]ol{pi}_B^*}
\end{equation}
where $\presub{[n]}\vc{\alpha}_{\bbB}:=(\presub{[n]}\alpha_{\bbB}(i))_{i \in \bbZ_n}$ is a probability vector such that
\begin{equation}
\sum_{i\in\bbB_n^+}\presub{[n]}\alpha_{\bbB}(i) > 0.
\label{defn-[n]alpha_B}
\end{equation}
Note here that $\presub{[n]}\ol{\vc{\pi}}_{\bbB}^{\ast}$ in (\ref{defn-[n]ol{pi}_B^*}) is well-defined due to (\ref{eqn-cup-B_n^+}) of Lemma~\ref{lem-B_n^+}.
It follows from (\ref{defn-[n]ol{pi}_B^*}) and (\ref{eqn-(n)ol{pi}}) that, for $n \ge B^*$,
\begin{equation}
\presub{[n]}\ol{\vc{\pi}}_{\bbB}^{\ast}
= {\presub{[n]}\ol{\vc{\pi}} \presub{[n]}\vc{E}_{\bbB}
\over 
\presub{[n]}\ol{\vc{\pi}} \presub{[n]}\vc{E}_{\bbB} \vc{e}
}
= {\presub{[n]}\ol{\vc{\pi}}_{\bbB} \over \presub{[n]}\ol{\vc{\pi}}_{\bbB}\vc{e}},
\label{eqn-[n]ol{pi}_B^*}
\end{equation}
where $\presub{[n]}\ol{\vc{\pi}}_{\bbB}=\presub{[n]}\ol{\vc{\pi}} \presub{[n]}\vc{E}_{\bbB}$. A limit formula for $\presub{[n]}\ol{\vc{\pi}}_{\bbB}^{\ast}$ is presented in the following theorem, which does not require Assumption~\ref{assumpt-1} or \ref{assumpt-2}. 
\begin{thm}\label{thm-submatrix-(n)F}
Let $\vc{\pi}_{\bbB}^{\ast} = \vc{\pi}_{\bbB}/(\vc{\pi}_{\bbB}\vc{e})$. We then have
\begin{eqnarray}
\lim_{n\to\infty} 
\presub{[n]}\ol{\vc{\pi}}_{\bbB}^{\ast}
&=& \lim_{n\to\infty} {
\presub{[n]}\vc{\alpha}_{\bbB}
(-\presub{[n]}\vc{Q})^{-1} \presub{[n]}\vc{E}_{\bbB}
\over
\presub{[n]}\vc{\alpha}_{\bbB}
(-\presub{[n]}\vc{Q})^{-1} \presub{[n]}\vc{E}_{\bbB}\vc{e}
}
= \vc{\pi}_{\bbB}^{\ast},
\label{lim-conditional}
\end{eqnarray}
where the first equality is due to (\ref{defn-[n]ol{pi}_B^*}).
\end{thm}

\begin{proof}
It is not assumed that the ergodic generator $\vc{Q}$ has a special structure. Thus, without loss of generality, it suffices to prove that, for an arbitrary $m \in \bbZ_+$,
\begin{equation}
\lim_{n\to\infty}
\presub{[n]}\ol{\vc{\pi}}_{\bbZ_m}^{\ast}
= \vc{\pi}_{\bbZ_m}^{\ast}.
\label{lim-conditional-02}
\end{equation}

We partition $\ol{\vc{Q}}$ and $\presub{[n]}\ol{\vc{Q}}$, $n \ge m$, as 
\begin{eqnarray}
\vc{Q}
&=& \bordermatrix{
       & 
\bbZ_m & 
\ol{\bbZ}_m
\cr
\bbZ_m & 
\vc{Q}_{\bbZ_m} &
\vc{Q}_{\bbZ_m,\ol{\bbZ}_m}
\cr
\ol{\bbZ}_m & 
\vc{Q}_{\ol{\bbZ}_m,\bbZ_m} &
\vc{Q}_{\ol{\bbZ}_m}
},
\label{partition-Q}
\\
\presub{[n]}\ol{\vc{Q}}
&=& \bordermatrix{
       & 
\bbZ_m & 
\bbZ_n \setminus \bbZ_m
\cr
\bbZ_m & 
\presub{[n]}\ol{\vc{Q}}_{\bbZ_m} &
\presub{[n]}\ol{\vc{Q}}_{\bbZ_m,\ol{\bbZ}_m}
\cr
\bbZ_n \setminus \bbZ_m & 
\presub{[n]}\ol{\vc{Q}}_{\ol{\bbZ}_m,\bbZ_m} &
\presub{[n]}\ol{\vc{Q}}_{\ol{\bbZ}_m}
},
\label{partition-(n)ol{Q}}
\end{eqnarray}
respectively.
We then define $\vc{Q}_{\bbZ_m}^{\ast}$ and $\presub{[n]}\ol{\vc{Q}}_{\bbZ_m}^{\ast}$ as
\begin{eqnarray}
\vc{Q}_{\bbZ_m}^{\ast}
&=& \vc{Q}_{\bbZ_m} 
+  \vc{Q}_{\bbZ_m,\ol{\bbZ}_m} (-\vc{Q}_{\ol{\bbZ}_m})^{-1}\vc{Q}_{\ol{\bbZ}_m,\bbZ_m},
\label{defn-Q_{<=m}}
\\
\presub{[n]}\ol{\vc{Q}}_{\bbZ_m}^{\ast}
&=& \presub{[n]}\ol{\vc{Q}}_{\bbZ_m} 
+  \presub{[n]}\ol{\vc{Q}}_{\bbZ_m,\ol{\bbZ}_m} (-\presub{[n]}\ol{\vc{Q}}_{\ol{\bbZ}_m})^{-1}\presub{[n]}\ol{\vc{Q}}_{\ol{\bbZ}_m,\bbZ_m},
\label{defn-(n)ol{Q}_{<=m}}
\end{eqnarray}
respectively. The $Q$-matrix $\vc{Q}{}_{\bbZ_m}^{\ast}$
(resp.\ $\presub{[n]}\ol{\vc{Q}}_{\bbZ_m}^{\ast}$) is the generator of a
censored Markov chain with state space $\bbZ_m$, which is obtained by observing  the Markov chain with generator $\vc{Q}$ (resp.\ $\presub{[n]}\ol{\vc{Q}}$) only when it is running in $\bbZ_m$. Since $\vc{Q}$ is ergodic,  
$\vc{Q}_{\bbZ_m}^{\ast}$ is also ergodic and thus has a unique stationary distribution vector $\vc{\pi}_{\bbZ_m}^{\ast}=\vc{\pi}_{\bbZ_m}/(\vc{\pi}_{\bbZ_m}\vc{e})$. On the other hand, $\presub{[n]}\ol{\vc{Q}}$ is not necessarily ergodic, but it has a stationary
distribution vector $\presub{[n]}\ol{\vc{\pi}}_{\bbZ_m}^{\ast}={\presub{[n]}\ol{\vc{\pi}}_{\bbZ_m} / (\presub{[n]}\ol{\vc{\pi}}_{\bbZ_m}\vc{e}})$.

We note that the following holds (see Appendix~\ref{proof-lim-(n)Q_{<=m}^*}):
\begin{equation}
\lim_{n\to\infty}\presub{[n]}\ol{\vc{Q}}_{\bbZ_m}^{\ast} = \vc{Q}_{\bbZ_m}^{\ast}.
\label{lim-(n)Q_{<=m}^*}
\end{equation}
Moreover, it follows from \cite[Section~4.1, Eq.~(9)]{Heid10} that
\begin{eqnarray}
\presub{[n]}\ol{\vc{\pi}}_{\bbZ_m}^{\ast} - \vc{\pi}_{\bbZ_m}^{\ast}
= \presub{[n]}\ol{\vc{\pi}}_{\bbZ_m}^{\ast}
(\presub{[n]}\ol{\vc{Q}}_{\bbZ_m}^{\ast} - \vc{Q}_{\bbZ_m}^{\ast})
\vc{D}_{\bbZ_m}^{\ast}, \quad n \ge m,\qquad
\label{diff-(n)pi_{<=m}^*}
\end{eqnarray}
where $\vc{D}_{\bbZ_m}^{\ast}$ is the deviation matrix of the
transition matrix function with generator $\vc{Q}_{\bbZ_m}^{\ast}$,
i.e.,
\[
\vc{D}_{\bbZ_m}^{\ast}
= \int_0^{\infty}
\left( 
\exp\{\vc{Q}_{\bbZ_m}^{\ast} t\} - \vc{e}\vc{\pi}_{\bbZ_m}^{\ast} 
\right) \rd t.
\]
Applying (\ref{lim-(n)Q_{<=m}^*}) to (\ref{diff-(n)pi_{<=m}^*})
results in (\ref{lim-conditional-02}). Consequently, we have proved that (\ref{lim-conditional}) holds.
\qed
\end{proof}

\section{Matrix-infinite-product-form solutions for block-Hessenberg Markov chains}\label{sec-Hessenberg}

This section considers the case where the ergodic generator
$\vc{Q}$ is in block-Hessenberg form. To this end, we introduce some symbols and rewrite $\vc{Q}$
as a block-structured generator.

Let $m_{\ell}$'s, $\ell \in \bbZ_+$, denote positive integers. Let
$n_{-1} = -1$ and $n_s=\sum_{\ell=0}^s m_{\ell} - 1$ for $s \in
\bbZ_+$. We then partition the state space $\bbZ_+$ into the substate
spaces $\bbL_s$'s, $s \in \bbZ_+$, where
\begin{equation}
\bbL_s = \{n_{s-1}+1,n_{s-1}+2,\dots,n_s\},\qquad s \in \bbZ_+.
\label{defn-L_s}
\end{equation}
We refer to the substate space $\bbL_s$ as {\it level $s$}, 
and partition $\vc{Q}$ level-wise, i.e.,
\begin{equation}
\vc{Q} = 
\bordermatrix{
        & \bbL_0 & \bbL_1    & \bbL_2    & \bbL_3   & \cdots
\cr
\bbL_0 		& 
\vc{Q}_{0,0}  	& 
\vc{Q}_{0,1} 	& 
\vc{Q}_{0,2}  	& 
\vc{Q}_{0,3} 	&  
\cdots
\cr
\bbL_1 		&
\vc{Q}_{1,0}  	& 
\vc{Q}_{1,1} 	& 
\vc{Q}_{1,2}  	& 
\vc{Q}_{1,3} 	&  
\cdots
\cr
\bbL_2 		& 
\vc{Q}_{2,0}  	& 
\vc{Q}_{2,1} 	& 
\vc{Q}_{2,2}  	& 
\vc{Q}_{2,3} 	&  
\cdots
\cr
\bbL_3 & 
\vc{Q}_{3,0}  	& 
\vc{Q}_{3,1} 	& 
\vc{Q}_{3,2}  	& 
\vc{Q}_{3,3} 	&  
\cdots
\cr
~\vdots  	& 
\vdots     		& 
\vdots     		&  
\vdots    		& 
\vdots    		& 
\ddots
},
\label{partitioned-Q}
\end{equation}
where $\vc{Q}_{k,\ell} = (q(i,j))_{(i,j) \in \bbL_k \times
  \bbL_{\ell}}$ for $k,\ell\in\bbZ_+$. By
definition, the cardinality of $\bbL_s$ is equal to $m_s$. 
Thus, $\vc{Q}_{k,\ell}$ is an $m_k \times m_{\ell}$ matrix.
In what follows, we discuss the ergodic generator $\vc{Q}$ partitioned in (\ref{partitioned-Q}). Thus, we partition its stationary distribution vector $\vc{\pi}$ as follows:
\[
\vc{\pi}
=
\bordermatrix{
        & 
\bbL_0 & 
\bbL_1 & 
\bbL_2 & 
\cdots
\cr
			& 
\vc{\pi}_0  & 
\vc{\pi}_1 	& 
\vc{\pi}_2  & 
\cdots
},
\]
where $\vc{\pi}_k = \vc{\pi}_{\bbL_k}$ for $k\in\bbZ_+$ ($\vc{\pi}_{\bbL_k}$ is introduced before Theorem~\ref{thm-submatrix-(n)F}). 

In this section, we consider two linearly augmented truncations. To this end, we use two augmentation distribution vectors $\presub{(s)}\wc{\vc{\alpha}}$ and $\presub{(s)}\wh{\vc{\alpha}}$, whose probability masses are, respectively, concentrated on the first and last blocks, that is,
\begin{eqnarray}
\presub{(s)}\wc{\vc{\alpha}}
&=& \bordermatrix{
        	& 
\bbL_0 		& 
\bbL_1 		& 
\cdots  	& 
\bbL_{s-1} &
\bbL_s 
\cr
			& 
\vc{\alpha}_0  	& 
\vc{0}  	& 
\cdots		&
\vc{0}  	& 
\vc{0} 	
},
\label{defn-(s)alpha_0}
\\
\presub{(s)}\wh{\vc{\alpha}}
&=& \bordermatrix{
        	& 
\bbL_0 		& 
\bbL_1 		& 
\cdots  	& 
\bbL_{s-1} &
\bbL_s 
\cr
			& 
\vc{0}  	& 
\vc{0}  	& 
\cdots		&
\vc{0}  	& 
\vc{\alpha}_s	
}, \qquad s \in \bbZ_+,
\label{defn-(s)alpha_s}
\end{eqnarray}
where $\vc{\alpha}_{\ell}$ is a $1 \times m_{\ell}$ probability vector. 
We then define two linearly augmented truncations $\presub{(s)}\wc{\vc{Q}}$ and $\presub{(s)}\wh{\vc{Q}}$ as follows.
\begin{defn}\label{defn-simple}
For $s \in \bbZ_+$, let
\begin{eqnarray}
\presub{(s)}\wc{\vc{Q}} 
&=& \presub{(s)}\vc{Q} 
- \presub{(s)}\vc{Q} \vc{e} \presub{(s)}\wc{\vc{\alpha}},
\label{defn-(s)ol{Q}_0}
\\
\presub{(s)}\wh{\vc{Q}}
&=& \presub{(s)}\vc{Q} 
- \presub{(s)}\vc{Q} \vc{e} \presub{(s)}\wh{\vc{\alpha}},
\label{defn-(s)ol{Q}_s}
\end{eqnarray}
where $\presub{(s)}\vc{Q} = \presub{[n_s]}\vc{Q}$. 
\end{defn}

Let $\presub{(s)}\wc{\vc{\pi}}$ and $\presub{(s)}\wh{\vc{\pi}}$ denote the stationary distribution vectors of $\presub{(s)}\wc{\vc{Q}}$ and $\presub{(s)}\wh{\vc{Q}}$, respectively. 
It then follows from (\ref{eqn-(n)ol{pi}}) and Definition~\ref{defn-simple} that
\begin{align}
\presub{(s)}\wc{\vc{\pi}}
&= 
{ \vc{\alpha}_0 (- \presub{(s)}\vc{Q})^{-1} 
\over 
\vc{\alpha}_0(- \presub{(s)}\vc{Q})^{-1} \vc{e}
}
=
{ (\vc{\alpha}_0, \vc{0},\dots,\vc{0}) (- \presub{(s)}\vc{Q})^{-1} 
\over 
(\vc{\alpha}_0, \vc{0},\dots,\vc{0})(- \presub{(s)}\vc{Q})^{-1} \vc{e}
}, & s &\in\bbZ_+,
\label{eqn-(s)ol{pi}_0}
\\
\presub{(s)}\wh{\vc{\pi}}
&= { \vc{\alpha}_s (- \presub{(s)}\vc{Q})^{-1} 
\over 
\vc{\alpha}_s(- \presub{(s)}\vc{Q})^{-1} \vc{e}
}
=
{ (\vc{0},\dots,\vc{0}, \vc{\alpha}_s) (- \presub{(s)}\vc{Q})^{-1} 
\over 
(\vc{0},\dots,\vc{0}, \vc{\alpha}_s)(- \presub{(s)}\vc{Q})^{-1} \vc{e}
}, & s &\in\bbZ_+.
\label{eqn-(s)ol{pi}_s}
\end{align}
Note that $\presub{(s)}\vc{Q} = \presub{[n_s]}\vc{Q}$ (see Definition~\ref{defn-simple}), which is partitioned as follows:
\begin{align}
\presub{(s)}\vc{Q}
&=
\bordermatrix{
        & 
\bbL_0 & 
\bbL_1 & 
\cdots  & 
\bbL_s 
\cr
\bbL_0						& 
\vc{Q}_{0,0}  	& 
\vc{Q}_{0,1}  	& 
\cdots						&
\vc{Q}_{0,s}  	 
\cr
\bbL_1						& 
\vc{Q}_{1,0}  	& 
\vc{Q}_{1,1}  	& 
\cdots						&
\vc{Q}_{1,s}  	 
\cr
~\vdots						&
\vdots						&
\vdots						&
\ddots						&
\vdots					
\cr
\bbL_s						& 
\vc{Q}_{s,0}  	& 
\vc{Q}_{s,1}  	& 
\cdots						&
\vc{Q}_{s,s}  	
}, & s &\in \bbZ_+.
\nonumber
\end{align}
For later use, we partition $(-\presub{(s)}\vc{Q})^{-1}$ as
\begin{align}
(-\presub{(s)}\vc{Q})^{-1} 
&=
\bordermatrix{
        			& 
\bbL_0 			& 
\bbL_1 			& 
\cdots  			& 
\bbL_s 
\cr
\bbL_0			& 
\presub{(s)}\vc{X}_{0,0}  	& 
\presub{(s)}\vc{X}_{0,1}  	& 
\cdots			&
\presub{(s)}\vc{X}_{0,s}  	 
\cr
\bbL_1		 	& 
\presub{(s)}\vc{X}_{1,0}  	& 
\presub{(s)}\vc{X}_{1,1}  	& 
\cdots			&
\presub{(s)}\vc{X}_{1,s}  	 
\cr
~\vdots			&
\vdots			&
\vdots			&
\ddots			&
\vdots						
\cr
\bbL_s			& 
\presub{(s)}\vc{X}_{s,0}  	& 
\presub{(s)}\vc{X}_{s,1}  	& 
\cdots			&
\presub{(s)}\vc{X}_{s,s}  	 
}, & s &\in \bbZ_+.
\label{defn-(-Q)^{-1}-MG1-type}
\end{align}
We also partition $\presub{(s)}\wc{\vc{\pi}}$ and $\presub{(s)}\wh{\vc{\pi}}$ as
\begin{align*}
\presub{(s)}\wc{\vc{\pi}}
&=
\bordermatrix{
        & 
\bbL_0 & 
\bbL_1 & 
\cdots  & 
\bbL_s 
\cr
			& 
\presub{(s)}\wc{\vc{\pi}}_{0}  	& 
\presub{(s)}\wc{\vc{\pi}}_{1}  	& 
\cdots			&
\presub{(s)}\wc{\vc{\pi}}_{s}  	
}, & s &\in \bbZ_+,
\\
\presub{(s)}\wh{\vc{\pi}}
&=
\bordermatrix{
        & 
\bbL_0 & 
\bbL_1 & 
\cdots  & 
\bbL_s 
\cr
			& 
\presub{(s)}\wh{\vc{\pi}}_{0} 	& 
\presub{(s)}\wh{\vc{\pi}}_{1}  	& 
\cdots			&
\presub{(s)}\wh{\vc{\pi}}_{s}	
}, & s &\in \bbZ_+.
\end{align*}

The rest of this section is divided into three subsections. Section~\ref{subsec-upper-case} discusses the computation of the stationary distribution vector in 
upper block-Hessenberg Markov chains, based on the sequence
$\{\presub{(s)}\wh{\vc{\pi}};s\in\bbZ_+ \}$ that converges to
the stationary distribution vector $\vc{\pi}$ under some technical conditions. Section~\ref{subsec-lower-case} develops an algorithm for computing the stationary distribution vector $\vc{\pi}$ in lower block-Hessenberg Markov chains, by using the
result in Section~\ref{subsec-upper-case} and the duality of upper and lower BHMCs. This algorithm generates a sequence
$\{\presub{(s)}\wc{\vc{\pi}};s\in\bbZ_+ \}$ that always converges to $\vc{\pi}$. Finally, Section~\ref{subsec-GI-M-1} considers a special case where the generator $\vc{Q}$ is of GI/M/1 type (see, e.g., \cite{Gras00-BC}), i.e., $\vc{Q}$ is a block-Toeplitz-like generator in lower block-Hessenberg form.

\subsection{Upper block-Hessenberg Markov chain}\label{subsec-upper-case}

In this subsection, we assume that the generator $\vc{Q}$ is in upper block-Hessenberg
form (i.e., is of level-dependent M/G/1-type):
\begin{equation}
\vc{Q} = 
\bordermatrix{
        & \bbL_0 & \bbL_1    & \bbL_2    & \bbL_3   & \cdots
\cr
\bbL_0 		& 
\vc{Q}_{0,0}  	& 
\vc{Q}_{0,1} 	& 
\vc{Q}_{0,2}  	& 
\vc{Q}_{0,3} 	&  
\cdots
\cr
\bbL_1 		&
\vc{Q}_{1,0}  	& 
\vc{Q}_{1,1} 	& 
\vc{Q}_{1,2}  	& 
\vc{Q}_{1,3} 	&  
\cdots
\cr
\bbL_2 		& 
\vc{O}  		& 
\vc{Q}_{2,1} 	& 
\vc{Q}_{2,2}  	& 
\vc{Q}_{2,3} 	&  
\cdots
\cr
\bbL_3 & 
\vc{O}  		& 
\vc{O}  		& 
\vc{Q}_{3,2}  	& 
\vc{Q}_{3,3} 	&  
\cdots
\cr
~\vdots  	& 
\vdots     		& 
\vdots     		&  
\vdots    		& 
\vdots    		& 
\ddots
},
\label{defn-Q-MG1-type}
\end{equation}
where $\vc{Q}_{k,\ell} = \vc{O}$ for $k\in\bbZ_+$ and $\ell
=0,1,\dots,\max(k-1,0)$.

\begin{rem}\label{rem-LD-QBD}
If $\vc{Q}$ in (\ref{defn-Q-MG1-type}) is block-tridiagonal, i.e.,
$\vc{Q}_{k,\ell} = \vc{O}$ for all $k,\ell \in \bbZ_+$ such that $|k -
\ell| \ge 2$, then $\vc{Q}$ can be considered the generator of a
level-dependent quasi-birth-and-death process (LD-QBD) (see
\cite{Brig95,Rama96}).
\end{rem}

\subsubsection{MIP-form solution}\label{subsubsec-MIP-upper}

It follows from (\ref{defn-Q-MG1-type}) that
\begin{equation}
\presub{(s)}\vc{Q}
=
\left(
\begin{array}{ccccccc}
\vc{Q}_{0,0} 	&
\vc{Q}_{0,1} 	&
\vc{Q}_{0,2} 	&
\cdots		&
\vc{Q}_{0,s-2} 	&
\vc{Q}_{0,s-1} 	&
\vc{Q}_{0,s} 	
\\
\vc{Q}_{1,0} 	&
\vc{Q}_{1,1} 	&
\vc{Q}_{1,2} 	&
\cdots		&
\vc{Q}_{1,s-2} 	&
\vc{Q}_{1,s-1} 	&
\vc{Q}_{1,s} 	
\\
\vc{O} 		&
\vc{Q}_{2,1} 	&
\vc{Q}_{2,2} 	&
\cdots		&
\vc{Q}_{2,s-2} 	&
\vc{Q}_{2,s-1} 	&
\vc{Q}_{2,s} 	
\\
\vdots		&
\vdots		&
\vdots		&
\ddots		&
\vdots		&
\vdots		&
\vdots		
\\
\vc{O}		&
\vc{O}		&
\vc{O}		&
\cdots		&
\vc{Q}_{s-1,s-2} 	&
\vc{Q}_{s-1,s-1} 	&
\vc{Q}_{s-1,s} 
\\
\vc{O}		&
\vc{O}		&
\vc{O}		&
\cdots		&
\vc{O}		&
\vc{Q}_{s,s-1} 	&
\vc{Q}_{s,s} 
\\
\end{array}
\right),\quad s \in \bbZ_+.
\label{partition-(n_s)Q-upper}
\end{equation}
The NW-corner truncation $\presub{(s)}\vc{Q}$ of $\vc{Q}$ is also in
the upper block-Hessenberg form. Therefore, as we will see later, we
can derive an efficient recursive formula for the last block
\[
(\presub{(s)}\vc{X}_{s,0},\presub{(s)}\vc{X}_{s,1},\dots,\presub{(s)}\vc{X}_{s,s})
\]
of $(-\presub{(s)}\vc{Q})^{-1}$ in (\ref{defn-(-Q)^{-1}-MG1-type}).
To derive this formula, we define $\{\vc{U}_k^{\ast};k\in\bbZ_+\}$
recursively as follows:
\begin{equation}
\vc{U}_k^{\ast}
= 
\left\{
\begin{array}{ll}
(-\vc{Q}_{0,0})^{-1}, & \qquad k=0,
\\
\left(
-\vc{Q}_{k,k} 
- \dm\sum_{\ell=0}^{k-1} \vc{U}_{k,\ell} \vc{Q}_{\ell,k}
\right)^{-1},
& \qquad k \in \bbN,
\end{array}
\right.
\label{defn-U_k^*}
\end{equation}
where $\vc{U}_{k,\ell}$'s, $k \in \bbN$, $\ell \in \bbZ_{k-1}$, are
given by
\begin{equation}
\vc{U}_{k,\ell}
=
(\vc{Q}_{k,k-1}   \vc{U}_{k-1}^{\ast}) 
(\vc{Q}_{k-1,k-2} \vc{U}_{k-2}^{\ast})  \cdots 
(\vc{Q}_{\ell+1,\ell}\vc{U}_{\ell}^{\ast}).
\label{defn-U_{k,l}}
\end{equation}
Since the empty sum is defined as zero, Eq.~(\ref{defn-U_k^*}) is
expressed as the single equation (i.e., the equation for $k \in \bbN$
is extended to the one for $k\in\bbZ_+$). Note here that
$\vc{U}_k^{\ast}$ is nonsingular, which is proved in
Appendix~\ref{appen-T_k^*}. Note also that $\vc{U}_k^{\ast}$ is nonnegative and $\vc{U}_k^{\ast}\vc{e} > \vc{0}$ (see Remark~\ref{rem-nonnegative-U_k^*}).

\medskip

The following lemma provides a matrix-product-form expression of the
$\presub{(s)}\vc{X}_{s,\ell}$'s.
\begin{lem}\label{lem-(n_s)X_{s,l}}
If the ergodic generator $\vc{Q}$ is in upper block-Hessenberg form
(\ref{defn-Q-MG1-type}), then
\begin{align}
&&&&&&&&
\presub{(s)}\vc{X}_{s,\ell} 
&= \vc{U}_s^{\ast} \vc{U}_{s,\ell},
& s \in \bbZ_+, \ell &\in \bbZ_s,&&&&&&&&
\label{eqn-(n_sX_{s,l})}
\\
&&&&&&&&
\vc{e}^{\top}\!\! \presub{(s)}\vc{X}_{s,\ell}
&=  \vc{e}^{\top} \vc{U}_s^{\ast} \vc{U}_{s,\ell}
> \vc{0}, & \ell &\in \bbZ_s,&&&&&&&&
\label{add-170513-02}
\end{align}
where $\vc{U}_{s,s} = \vc{I}$ for $s\in\bbZ_+$.
\end{lem}

\begin{rem}
Shin~\cite{Shin09} presents the similar expressions of all the blocks
$\presub{(s)}\vc{X}_{k,\ell}$'s in a special case where $\vc{Q}$ in
(\ref{defn-Q-MG1-type}) is reduced to be block tridiagonal (see
Theorem~2.1 therein), i.e., to the generator of an LD-QBD (see
Remark~\ref{rem-LD-QBD}).
\end{rem}

\begin{rem}\label{rem-Taki16-00}
It is stated in \cite[Remark 2]{Taki16} that if the ergodic generator
$\vc{Q}$ is in upper block-Hessenberg form (\ref{defn-Q-MG1-type})
then
\begin{equation}
\vc{\pi}_{\ell} = \vc{\pi}_k \vc{U}_{k,\ell},
\qquad k \in \bbN,\ \ell \in \bbZ_{k-1}.
\label{eqn-pi_l-pi_k*U_{k,l}}
\end{equation}
For the reader's convenience, we provide a complete proof of
(\ref{eqn-pi_l-pi_k*U_{k,l}}) in
Appendix~\ref{proof-pi_l-pi_k*U_{k,l}}.
\end{rem}

\noindent
{\it Proof of Lemma~\ref{lem-(n_s)X_{s,l}}~} The inverse
$\presub{(s)}\vc{Q}^{-1}$ of $\presub{(s)}\vc{Q}$ is the unique
solution of $\presub{(s)}\vc{Q}^{-1}\presub{(s)}\vc{Q} =
\vc{I}$. Thus, the last block row
$(\presub{(s)}\vc{X}_{s,0},\presub{(s)}\vc{X}_{s,1},\dots,\presub{(s)}\vc{X}_{s,s})$
of $(-\presub{(s)}\vc{Q})^{-1}$ is the unique solution of the
following equations:
\begin{eqnarray}
\vc{O}
&=&  \presub{(s)}\vc{X}_{s,0} \vc{Q}_{0,0} 
+  \presub{(s)}\vc{X}_{s,1}\vc{Q}_{1,0},
\label{eqn-X-01}
\\
\vc{O}
&=& \sum_{\ell= 0}^{k+1} \presub{(s)}\vc{X}_{s,\ell} \vc{Q}_{\ell,k},
\qquad k=1,2,\dots,s-1,
\label{eqn-X-02}
\\
-\vc{I}
&=& \sum_{\ell= 0}^s \presub{(s)}\vc{X}_{s,\ell} \vc{Q}_{\ell,s}.
\label{eqn-X-03}
\end{eqnarray}
Solving (\ref{eqn-X-01}) with respect to $\presub{(s)}\vc{X}_{s,0}$
and applying (\ref{defn-U_k^*}) to the result, we have
\begin{equation}
\presub{(s)}\vc{X}_{s,0} 
= \presub{(s)}\vc{X}_{s,1}\vc{Q}_{1,0}\vc{U}_0^{\ast}
= \presub{(s)}\vc{X}_{s,1}\vc{U}_{1,0},
\label{eqn-X_{s,l}-product-form-k=0}
\end{equation}
where the second equality follows from (\ref{defn-U_{k,l}}).

We now suppose that, for some $k \in \{1,2,\dots,s-1\}$, 
\begin{equation}
\presub{(s)}\vc{X}_{s,\ell} = \presub{(s)}\vc{X}_{s,k} \vc{U}_{k,\ell}
\quad \mbox{for all $\ell \in \bbZ_{k-1}$},
\label{eqn-X_{s,l}-product-form}
\end{equation}
which holds at least for $k=1$ due to
(\ref{eqn-X_{s,l}-product-form-k=0}). Substituting
(\ref{eqn-X_{s,l}-product-form}) into (\ref{eqn-X-02}) and using
(\ref{defn-U_k^*}), we obtain
\begin{eqnarray*}
\vc{O}
&=& \sum_{\ell= 0}^{k-1} \presub{(s)}\vc{X}_{s,k} \vc{U}_{k,\ell}\vc{Q}_{\ell,k} 
+ \presub{(s)}\vc{X}_{s,k} \vc{Q}_{k,k}
+ \presub{(s)}\vc{X}_{s,k+1} \vc{Q}_{k+1,k}
\nonumber
\\
&=& \presub{(s)}\vc{X}_{s,k}
\left(\vc{Q}_{k,k}
+ \sum_{\ell= 0}^{k-1}  \vc{U}_{k,\ell} \vc{Q}_{\ell,k}
\right)
+ \presub{(s)}\vc{X}_{s,k+1} \vc{Q}_{k+1,k}
\\
&=& \presub{(s)}\vc{X}_{s,k} (-\vc{U}_k^{\ast})^{-1} 
+ \presub{(s)}\vc{X}_{s,k+1} \vc{Q}_{k+1,k},
\end{eqnarray*}
which leads to
\begin{eqnarray}
\presub{(s)}\vc{X}_{s,k}
&=& 
\presub{(s)}\vc{X}_{s,k+1} \vc{Q}_{k+1,k} \vc{U}_k^{\ast}.
\label{recursion-(n_s)X_{s,k}}
\end{eqnarray}
Using (\ref{recursion-(n_s)X_{s,k}}) and (\ref{defn-U_{k,l}}), we
rewrite (\ref{eqn-X_{s,l}-product-form}) as
\begin{eqnarray*}
\presub{(s)}\vc{X}_{s,\ell} 
&=& \presub{(s)}\vc{X}_{s,k+1} \vc{Q}_{k+1,k} \vc{U}_k^{\ast} \vc{U}_{k,\ell}
\nonumber
\\
&=& \presub{(s)}\vc{X}_{s,k+1} \vc{U}_{k+1,\ell}
\qquad \mbox{for all $\ell \in \bbZ_k$}.
\end{eqnarray*}
Therefore, by induction, we have
\begin{equation}
\presub{(s)}\vc{X}_{s,\ell} 
= \presub{(s)}\vc{X}_{s,s} \vc{U}_{s,\ell} 
\quad \mbox{for all $\ell \in \bbZ_{s-1}$}.
\label{eqn-X_{s,l}-02}
\end{equation}

To complete the proof of (\ref{eqn-(n_sX_{s,l})}), we show that $\presub{(s)}\vc{X}_{s,s} = \vc{U}_s^{\ast}$.
Applying (\ref{eqn-X_{s,l}-02}) to (\ref{eqn-X-03}) and following the
derivation of (\ref{recursion-(n_s)X_{s,k}}), we obtain
\begin{eqnarray*}
-\vc{I}
&=& \presub{(s)}\vc{X}_{s,s} \vc{Q}_{s,s}
+ \sum_{\ell= 0}^{s-1}\presub{(s)}\vc{X}_{s,s} \vc{U}_{s,\ell} \vc{Q}_{\ell,s}
\nonumber
\\
&=& \presub{(s)}\vc{X}_{s,s}
\left( 
\vc{Q}_{s,s} + \sum_{\ell= 0}^{s-1} \vc{U}_{s,\ell} \vc{Q}_{\ell,s} \right)
= \presub{(s)}\vc{X}_{s,s} (-\vc{U}_s^{\ast})^{-1},
\end{eqnarray*}
which results in $\presub{(s)}\vc{X}_{s,s} = \vc{U}_s^{\ast}$.

Finally, we prove (\ref{add-170513-02}).
Since $\vc{Q}$ is in upper block-Hessenberg form
(\ref{defn-Q-MG1-type}), the Markov chain $\{Z(t)\}$ must go through $\bbL_s$ to move from $\bigcup_{k=s+1}^{\infty}\bbL_k$ to $\bigcup_{k=0}^{s-1}\bbL_k$. Therefore, for each $s \in \bbN$ and $j \in \bigcup_{k=0}^{s-1}\bbL_k$, there exists at least one state $i \in \bbL_s$ from which the Markov chain $\{Z(t)\}$ can reach state $j \in \bigcup_{k=0}^{s-1}\bbL_k$ avoiding $\bigcup_{k=s+1}^{\infty}\bbL_k$. This implies that, for each $s \in \bbN$, the submatrix $(\presub{(s)}\vc{X}_{s,0},\presub{(s)}\vc{X}_{s,1},\dots,\presub{(s)}\vc{X}_{s,s-1})$ of $(-\presub{(s)}\vc{Q})^{-1}$ has no zero columns. In addition, $\vc{U}_s^{\ast} \ge (-\vc{Q}_{s,s})^{-1}$, which shows that all the diagonal elements of $\presub{(s)}\vc{X}_{s,s} = \vc{U}_s^{\ast}$ are positive. As a result, (\ref{add-170513-02}) holds.
\qed

\medskip

Using Lemma~\ref{lem-(n_s)X_{s,l}}, we obtain an expression of $\presub{(s)}\wh{\vc{\pi}} = (\presub{(s)}\wh{\vc{\pi}}_{0},\presub{(s)}\wh{\vc{\pi}}_{1},\dots,\presub{(s)}\wh{\vc{\pi}}_{s})$ for $s \in \bbZ_+$.
\begin{lem}\label{lem-producet-form-(s)ol{pi}_s}
If the conditions of Lemma~\ref{lem-(n_s)X_{s,l}} are satisfied, then 
\begin{eqnarray}
\presub{(s)}\wh{\vc{\pi}}_{k}
&=&
{
\vc{\alpha}_s \vc{U}_s^{\ast} \vc{U}_{s,k}
\over 
\vc{\alpha}_s \sum_{\ell=0}^s \vc{U}_s^{\ast} \vc{U}_{s,\ell}\vc{e}
}, \qquad s \in \bbZ_+,\ k \in \bbZ_s. 
\label{eqn-(n_s)pi_{n_s,l}}
\end{eqnarray}
\end{lem}

\proof 
It follows from (\ref{eqn-(s)ol{pi}_s}) and (\ref{defn-(-Q)^{-1}-MG1-type}) that %
\begin{eqnarray}
\presub{(s)}\wh{\vc{\pi}}
=
{\vc{\alpha}_s 
(\vc{X}_{s,0},\presub{(s)}\vc{X}_{s,1},\dots,\presub{(s)}\vc{X}_{s,s})
\over
\vc{\alpha}_s \sum_{\ell=0}^s \presub{(s)}\vc{X}_{s,\ell}\vc{e}
}.
\label{eqn-(n_s)pi_{n_s}-02}
\end{eqnarray}
Applying Lemma~\ref{lem-(n_s)X_{s,l}}
to (\ref{eqn-(n_s)pi_{n_s}-02}), we have (\ref{eqn-(n_s)pi_{n_s,l}}).
\qed

\medskip

The following theorem is an immediate consequence of 
Lemma~\ref{lem-producet-form-(s)ol{pi}_s} and Theorem~\ref{thm-lim-[n]F}.
\begin{thm}[Single-limit MIP-form solution]\label{thm-MIP-upper}
Suppose that the ergodic generator $\vc{Q}$ is in upper
block-Hessenberg form (\ref{defn-Q-MG1-type}), and that $\vc{Q}$ satisfies Assumption~\ref{assumpt-1}. If Assumption~\ref{assumpt-2}, with replaced $\presub{[n]}\ol{\vc{Q}}$ by $\presub{(s)}\wh{\vc{Q}}$, holds for $s \in \bbZ_+$, then
\begin{equation}
\lim_{s\to\infty} \|\presub{(s)}\wh{\vc{\pi}} - \vc{\pi}\| = 0,
\label{lim-(s)wh{pi}}
\end{equation}
and thus
\begin{eqnarray}
\vc{\pi}_k
&=& \lim_{s\to\infty}{ 
\vc{\alpha}_s \vc{U}_s^{\ast} \vc{U}_{s,k}
\over 
\vc{\alpha}_s \sum_{\ell=0}^s \vc{U}_s^{\ast} \vc{U}_{s,\ell}\vc{e}
},
\qquad k \in \bbZ_+.
\label{pi_l-MIP-upper-case}
\end{eqnarray}
\end{thm}

\medskip

We now define $\vc{U}_k$, $k\in\bbZ_+$, as
\begin{equation*}
\vc{U}_k = \vc{Q}_{k+1,k} \vc{U}_k^{\ast},
\qquad k \in \bbZ_+.
\end{equation*}
It then follows from (\ref{defn-U_{k,l}}) that
\begin{equation}
\vc{U}_{s,\ell} = \vc{U}_{s-1} \vc{U}_{s-2} \cdots \vc{U}_{\ell},
\qquad \ell \in \bbZ_{s-1}.
\label{defn-U_{s,l}}
\end{equation}
Substituting this into (\ref{pi_l-MIP-upper-case}) yields
\begin{equation}
\vc{\pi}_k
= \lim_{s\to\infty}{ 
\vc{\alpha}_s \vc{U}_s^{\ast}
\vc{U}_{s-1} \vc{U}_{s-2} \cdots \vc{U}_k
\over 
\vc{\alpha}_s 
\sum_{\ell=0}^s \vc{U}_s^{\ast} \vc{U}_{s-1} \vc{U}_{s-2} \cdots \vc{U}_{\ell} \vc{e}
},
\qquad k \in \bbZ_+,
\label{pi_l-MIP-upper-case-02}
\end{equation}
which shows that $\vc{\pi}_k$ is expressed as a (normalized) matrix infinite product. We thus refer to the
expression (\ref{pi_l-MIP-upper-case-02}), or equivalently,
(\ref{pi_l-MIP-upper-case}) as the {\it matrix-infinite-product-form
  (MIP-form) solution}. We also refer to this MIP-form solution as the {\it single-limit MIP-form solution} to emphasize that the solution is expressed by single limit (i.e., taking the limit with respect to the single variable $n$).

\subsubsection{Computation of $\vc{\pi}$ based on the MIP-form solution}\label{subsub-sec-computation-upper}

We discuss the computation of $\vc{\pi}=(\vc{\pi}_0,\vc{\pi}_1,\dots)$ based on (\ref{lim-(s)wh{pi}}) and the MIP-from solution (\ref{pi_l-MIP-upper-case}). From
(\ref{eqn-(n_s)pi_{n_s,l}}), we have
\begin{eqnarray}
\presub{(s)}\wh{\vc{\pi}}_{k}
= {
\vc{\alpha}_s \vc{U}^{\ast}_{s,k}
\over 
\vc{\alpha}_s \vc{u}^{\ast}_s
},
\qquad s \in \bbZ_+,\ k \in \bbZ_s,
\label{defn-x_{s,l}}
\end{eqnarray}
where
\begin{align}
\vc{U}^{\ast}_{s,k}
&= \vc{U}_s^{\ast} \vc{U}_{s,k},
& s &\in \bbZ_+,\ k \in \bbZ_s,
\label{defn-overline{U}_{s,l}}
\\
\vc{u}^{\ast}_s
&= \sum_{\ell=0}^s \vc{U}^{\ast}_{s,\ell}\vc{e}
=   \sum_{\ell=0}^s \vc{U}_s^{\ast} \vc{U}_{s,\ell}\vc{e},
& s &\in \bbZ_+.
\label{defn-overline{u}_s}
\end{align}
Combining (\ref{defn-overline{U}_{s,l}}) with
(\ref{defn-U_k^*}) and (\ref{defn-U_{k,l}}), we
have
\begin{eqnarray}
\vc{U}^{\ast}_{0,0}
&=& \vc{U}^{\ast}_0 = (-\vc{Q}_{0,0})^{-1},
\label{eqn-U_{0,0}}
\\
\vc{U}^{\ast}_{s,k}
&=& 
\left\{
\begin{array}{lll}
\vc{U}_s^{\ast} \vc{Q}_{s,s-1} \cdot \vc{U}^{\ast}_{s-1,k},
& \quad s \in \bbN,\ & k \in \bbZ_{s-1},
\\
\vc{U}_s^{\ast},
& \quad s \in \bbN,\ &  k = s.
\end{array}
\right.
\label{eqn-U_{s-1,l}}
\end{eqnarray}
From (\ref{defn-overline{u}_s}), (\ref{eqn-U_{0,0}}) and
(\ref{eqn-U_{s-1,l}}), we also have
\begin{eqnarray}
\vc{u}^{\ast}_0
&=& \vc{U}_0^{\ast} \vc{e} = (-\vc{Q}_{0,0})^{-1}\vc{e},
\label{eqn-u^*_0}
\\
\vc{u}^{\ast}_s
&=& \vc{U}_s^{\ast} 
\left( \vc{e} + \vc{Q}_{s,s-1} \vc{u}^{\ast}_{s-1}
\right),
\qquad s \in \bbN.
\label{recursion-u^*_s}
\end{eqnarray}
Furthermore, applying (\ref{defn-U_{k,l}})
to (\ref{defn-U_k^*}) with $k=s \in \bbN$, we obtain
\begin{eqnarray}
\vc{U}_s^{\ast}
&=& 
\left(
- \vc{Q}_{s,s} 
- \vc{Q}_{s,s-1} \sum_{\ell=0}^{s-1} \vc{U}_{s-1}^{\ast}\vc{U}_{s-1,\ell}\vc{Q}_{\ell,s}
\right)^{-1}
\nonumber
\\
&=& 
\left(
- \vc{Q}_{s,s} 
- \vc{Q}_{s,s-1} \sum_{\ell=0}^{s-1} \vc{U}_{s-1,\ell}^{\ast}\vc{Q}_{\ell,s}
\right)^{-1},
\qquad s \in \bbN,
\label{eqn-U_s^*}
\end{eqnarray}
where the second equality follows from
(\ref{defn-overline{U}_{s,l}}).

As a result, under these assumptions, we can establish an algorithm for
computing $\vc{\pi}$ under the conditions of Theorem~\ref{thm-MIP-upper}.

\medskip

{
\renewcommand{\baselinestretch}{0.8}\selectfont
\begin{algorithm}[H]
\caption{Upper block-Hessenberg Markov chain}\label{algo-upper}
{\bf Input}: $\vc{Q}$, $\varepsilon \in (0,1)$, $\{\vc{\alpha}_0,\vc{\alpha}_1,\vc{\alpha}_2,\dots\}$.
\\ 
{\bf Output}: 
$\presub{(s)}\wh{\vc{\pi}} = (\presub{(s)}\wh{\vc{\pi}}_{0},\presub{(s)}\wh{\vc{\pi}}_{1},\dots,\presub{(s)}\wh{\vc{\pi}}_{s})$,
where $s \in \bbZ_+$ is fixed when the iteration stops.
%
\smallskip
\begin{enumerate}
\item Set $s = 0$.
\item Compute $\vc{U}_0^{\ast} = \vc{U}_{0,0}^{\ast} =
  (-\vc{Q}_{0,0})^{-1}$ and $\vc{u}_0^{\ast} = \vc{U}_0^{\ast}\vc{e}$.
\item Compute $\presub{(0)}\ol{\vc{\pi}}_0 = \vc{\alpha}_0\vc{U}_0^{\ast}/(\vc{\alpha}_0\vc{u}_0^{\ast})$.
\item Iterate the following:
\begin{enumerate}
\item Increment $s$ by one.
\item Compute $\vc{U}_s^{\ast}=\vc{U}_{s,s}^{\ast}$ by (\ref{eqn-U_s^*}).
\item Compute $\{\vc{U}_{s,k}^{\ast};k \in \bbZ_{s-1}\}$ by
  (\ref{eqn-U_{s-1,l}}) and  $\vc{u}^{\ast}_s$ by
  (\ref{recursion-u^*_s}).
\item Compute $\{\presub{(s)}\wh{\vc{\pi}}_{k};k\in\bbZ_s\}$ by (\ref{defn-x_{s,l}}).
\item If $\| \presub{(s)}\wh{\vc{\pi}} - \presub{(s-1)}\wh{\vc{\pi}} \| < \varepsilon$, then stop the
  iteration; otherwise return to step~(a).
\end{enumerate}
\end{enumerate}
\end{algorithm}
\renewcommand{\baselinestretch}{1.1}\selectfont
}

\begin{rem}\label{add-rem-algo-1}
If the conditions of Theorem~\ref{thm-MIP-upper} does not hold, then Algorithm~\ref{algo-upper} may not stop after a finite number of iterations. 
\end{rem}

\begin{rem}\label{rem-algo-1}
If Algorithm~\ref{algo-upper} stops at $s = N$, then an $N+1$ number
of the inverse matrices $\vc{U}_s^{\ast}$'s are computed because one
inverse matrix $\vc{U}_s^{\ast}$ is computed for each $s \in
\bbZ_+$. The computation of such inverse matrices is the most
time-consuming part of Algorithm~\ref{algo-upper}. However, we can
compute $\vc{U}_s^{\ast}$ by a stable and efficient procedure proposed
by Le Boudec~\cite{Le-Boud91} (see Proposition~\ref{prop-Le-Boud91} in
Appendix~\ref{appen-comp-U_k^*}).
\end{rem}

\subsubsection{Comparison with previous studies}

We discuss the comparison of our  Algorithm~\ref{algo-upper} with the existing ones in previous studies. As mentioned in Remark~\ref{rem-LD-QBD}, the LD-QBD is a special case
of upper BHMCs. Thus,
Algorithm~\ref{algo-upper} is applicable to LD-QBDs and works more
efficiently for them because (\ref{eqn-U_s^*}) is reduced to
\begin{eqnarray*}
\vc{U}_s^{\ast}
&=& 
\left(
- \vc{Q}_{s,s} 
- \vc{Q}_{s,s-1}  \vc{U}_{s-1,s-1}^{\ast}\vc{Q}_{s-1,s}
\right)^{-1},
\qquad s \in \bbN.
\end{eqnarray*}
It is well-known (see \cite{Brig95,Rama96}) that if $\vc{Q}$ in
(\ref{defn-Q-MG1-type}) is reduced to be the generator of the LD-QBD
then the stationary distribution vector
$\vc{\pi}=(\vc{\pi}_0,\vc{\pi}_1,\dots)$ is in the matrix product form
\begin{equation*}
\vc{\pi}_k = \vc{\pi}_0 \vc{R}^{(1)}\vc{R}^{(2)} \cdots \vc{R}^{(k)},
\qquad k \in \bbN,
\end{equation*}
where $\vc{\pi}_0$ is the solution of
\begin{eqnarray*}
\vc{\pi}_0(\vc{Q}_{0,0} + \vc{R}^{(1)}\vc{Q}_{1,0}) &=& \vc{0},
\\
\vc{\pi}_0 \left(\vc{e} + \sum_{\ell=1}^{\infty}  
\vc{R}^{(1)}\vc{R}^{(2)} \cdots \vc{R}^{(\ell)} \vc{e}\right) &=& 1,
\end{eqnarray*}
and where the matrices $\vc{R}^{(k)}$'s are the minimal nonnegative
solutions of
\begin{eqnarray}
\vc{Q}_{k-1,k} + \vc{R}^{(k)}\vc{Q}_{k,k} + \vc{R}^{(k)}\vc{R}^{(k+1)} \vc{Q}_{k+1,k}
= \vc{O},\qquad k \in \bbN.
\label{eqn-R^{(k)}}
\end{eqnarray}
From (\ref{eqn-R^{(k)}}), we obtain
\begin{equation}
\vc{R}^{(k)}
= \vc{Q}_{k-1,k}( - \vc{Q}_{k,k} - \vc{R}^{(k+1)} \vc{Q}_{k+1,k})^{-1},
\quad k \in \bbN,
\label{recursion-R^{(k)}}
\end{equation}
which enables us to compute the matrices
$\vc{R}^{(N-1)},\vc{R}^{(N-2)},\dots,\vc{R}^{(1)}$, given
$\vc{R}^{(N)}$ for some $N \in \bbN$. Thus, we refer to
(\ref{recursion-R^{(k)}}) as the backward recursion for
$\{\vc{R}^{(k)};k\in \bbN\}$.

Based on the above results, Phung-Duc et al.~\cite{Phun10-QTNA}
propose a simple algorithm for LD-QBDs (a similar algorithm is
discussed in \cite{Baum12-Procedia}). According to the algorithm, we first
choose a sufficiently large $N \in \bbN$ such that
$\sum_{k=N+1}^{\infty} \vc{\pi}_k\vc{e}$ is expected to be
negligible. Next, for a sufficiently large $L \in \bbN$, we compute an
approximation $\vc{R}_L^{(N)}$ to $\vc{R}^{(N)}$ by the backward
recursion (\ref{recursion-R^{(k)}}) with $\vc{R}^{(N+L)} = \vc{O}$. It
is shown (see \cite[Proposition~2.4]{Phun10-QTNA}) that
$\lim_{L\to\infty}\vc{R}_L^{(N)} = \vc{R}^{(N)}$.  However, we have to
determine $L$ by trial and error (for the details, see
\cite[Algorithms 1 and 3]{Phun10-QTNA}).

Bright and Taylor~\cite{Brig95} develop an elaborate algorithm for
LD-QBDs, which generates $\vc{R}_L^{(N)}$ with $L=2^{\ell+1} - 1$
($\ell\in\bbZ_+$) by using the logarithmic reduction approach originally developed for {\it level-independent} QBDs (see \cite{Lato93}). However, logarithmic reduction does not work well for LD-QBDs. Indeed, Bright and Taylor's algorithm \cite{Brig95} has the same computational complexity as 
that of Phung-Duc et al.'s algorithm~\cite{Phun10-QTNA}, and the former is more memory-consuming than the latter (for details, see \cite[Section~3.2]{Phun10-QTNA}).

We now suppose that $L$ is given in advance. In this case, Phung-Duc
et al.'s algorithm~\cite{Phun10-QTNA} computes $L$ inverse matrices to
obtain $\vc{R}_L^{(N)}$, and then computes $N-1$ inverse matrices in
generating $\vc{R}_L^{(k)}$'s, $k=N-1,N-2,\dots,1$, by the backward
recursion (\ref{recursion-R^{(k)}}) with
$\vc{R}^{(N)}=\vc{R}_L^{(N)}$, in order to obtain
$(\vc{\pi}_0,\vc{\pi}_1,\dots,\vc{\pi}_N)$ (of course, approximately).
It should be noted that the computation of $\vc{R}_L^{(N)}$ corresponds to the
situation that the iteration index $s$ of Algorithm~\ref{algo-upper}
reaches $N+L-1$ and thus $\presub{(N+L-1)}\ol{\vc{\pi}}$, which is an
approximation to $(\vc{\pi}_0,\vc{\pi}_1,\dots,\vc{\pi}_{N+L-1})$, is
obtained as the result of computing $N+L$ inverse matrices (see
Remark~\ref{rem-algo-1}). Consequently, Algorithm~\ref{algo-upper} has the same  computational complexity as that of
Phung-Duc et al.'s algorithm~\cite{Phun10-QTNA} even when the
situation is best for the latter one, i.e., $L$ is given in advance.

There are some studies on upper BHMCs.  Shin
and Pearce~\cite{Shin98} establish an algorithm for computing the
stationary distribution vector of the discrete-time upper
BHMC. Similar algorithms are proposed
for a BMAP/M/1 generalized processor-sharing queue (Li et al.~\cite{LiQuan05-RG}) and for an asymptotically level-independent M/G/1-type Markov chain (Klimenok and Dudin~\cite{Klim06}). These algorithms transform the original transition probability matrix (or the original generator in the continuous-time case) into a level-independent one (i.e., block-Toeplitz-like one) except for a finite number, say $N$, of levels. They thus require us to compute,
from scratch, Neuts' $G$-matrix \cite{Neut89} and the stationary
probabilities of the first $N$ levels (i.e., $\bbL_0,\bbL_1,\dots,\bbL_{N-1}$) every time $N$ is incremented one by one. On the other hand, each iteration of Algorithm~\ref{algo-upper} inherits the results from the previous iteration.

In the rest of this subsection, we compare Takine~\cite{Taki16}'s results with our results presented in Sections~\ref{subsubsec-MIP-upper} and \ref{subsub-sec-computation-upper}. For this purpose, we first present a corollary of Theorem~\ref{thm-submatrix-(n)F} and Lemma~\ref{lem-(n_s)X_{s,l}}.
\begin{coro}\label{coro-Takine-thm3}
Suppose that the ergodic generator $\vc{Q}$ is in upper block-Hessenberg form
(\ref{defn-Q-MG1-type}). We then have
\begin{equation}
\lim_{s\to\infty}
{
\vc{e}^{\top} \vc{U}_s^{\ast} \vc{U}_{s,k} \over 
\vc{e}^{\top} \vc{U}_s^{\ast} \vc{U}_{s,k}\vc{e}
}
= {\vc{\pi}_k \over \vc{\pi}_k \vc{e} },
\qquad k \in \bbZ_+.
\label{limt-Takine-02}
\end{equation}
Furthermore, if there exists some $k \in \bbZ_+$ such that
\begin{equation}
\vc{U}_s^{\ast} \vc{U}_{s,k}\vc{e} > \vc{0}\quad
\mbox{for all sufficiently large $s \in \ol{\bbZ}_k$},
\label{cond-170514-01}
\end{equation}
then
\begin{equation}
\lim_{s\to\infty} \diag^{-1}\{\vc{U}_s^{\ast} \vc{U}_{s,k}\vc{e}\} 
\vc{U}_s^{\ast} \vc{U}_{s,k}
= {\vc{e}\vc{\pi}_k \over \vc{\pi}_k \vc{e} }.
\label{limt-Takine-02-add}
\end{equation}
\end{coro}

\begin{proof}
In the proof of this corollary, we fix $\bbB=\bbL_k$ and $n=n_s$. It follows from (\ref{eqn-cup-B_n^+}) in Lemma~\ref{lem-B_n^+} that, for each $k \in \bbZ_s$, at least one row of $(-\presub{[n_s]}\vc{Q})^{-1} \presub{[n_s]}\vc{E}_{\bbL_k}$ is nonnegative and thus $\bbB_{n_s}^+ \cap \bbL_k \neq \varnothing$.
We now fix 
\[
\presub{[n_s]}\vc{\alpha}_{\bbL_k} 
=
\bordermatrix{
& \bbL_0 & \bbL_1 & \cdots &\bbL_{s-1} & \bbL_s 
\cr 
& \vc{0} & \vc{0} & \cdots & \vc{0}    & \vc{e}^{\top}/m_s
}.
\]
Substituting this into (\ref{defn-[n]ol{pi}_B^*}) (with $n=n_s$ and $\bbB=\bbL_k$), we have
\begin{eqnarray}
\presub{[n_s]}\ol{\vc{\pi}}_{\bbL_k}^{\ast}
&=& {
(\vc{0},\dots,\vc{0},\vc{e}^{\top}/m_s)
(-\presub{[n_s]}\vc{Q})^{-1} \presub{[n_s]}\vc{E}_{\bbL_k}
\over
(\vc{0},\dots,\vc{0},\vc{e}^{\top}/m_s)
(-\presub{[n_s]}\vc{Q})^{-1} \presub{[n_s]}\vc{E}_{\bbL_k} \vc{e}
}
\nonumber
\\
&=& 
{
\vc{e}^{\top} \presub{(s)}\vc{X}_{s,k}
\over
\vc{e}^{\top} \presub{(s)}\vc{X}_{s,k} \vc{e}
}
=
{\vc{e}^{\top} \vc{U}_s^{\ast} \vc{U}_{s,k}
\over
\vc{e}^{\top} \vc{U}_s^{\ast} \vc{U}_{s,k}\vc{e}
},
\label{eqn-171020-03}
\end{eqnarray}
where the second and third equalities are due to (\ref{defn-(-Q)^{-1}-MG1-type}) and Lemma~\ref{lem-(n_s)X_{s,l}}, respectively.
Combining (\ref{eqn-171020-03}) and Theorem~\ref{thm-submatrix-(n)F} to this equation yields
\[
\lim_{s\to\infty}
{
\vc{e}^{\top}\vc{U}_s^{\ast} \vc{U}_{s,k}
\over 
\vc{e}^{\top}\vc{U}_s^{\ast} \vc{U}_{s,k} \vc{e}
}
= {\vc{\pi}_k \over \vc{\pi}_k\vc{e}},\qquad k \in \bbZ_+,
\]
which shows that (\ref{limt-Takine-02}) holds.

Next, we prove (\ref{limt-Takine-02-add}), which requires the additional condition (\ref{cond-170514-01}) with some $k \in \bbZ_+$. Under this additional condition, 
\[
\bbB_{n_s}^+ \supseteq \bbL_k\quad \mbox{for all sufficiently large $s \in \ol{\bbZ}_k$}.
\]
To proceed further, fix
\[
\presub{[n_s]}\vc{\alpha}_{\bbL_k} 
= (\vc{0},\dots,\vc{0},\presub{[m_s]}\vc{e}_{\ang{\nu}}^{\top}/m_s).
\]
It then follows from (\ref{defn-(-Q)^{-1}-MG1-type}) and Lemma~\ref{lem-(n_s)X_{s,l}} and (\ref{defn-[n]ol{pi}_B^*}) (with $n=n_s$ and $\bbB=\bbL_k$) that
\begin{equation}
\presub{[n_s]}\ol{\vc{\pi}}_{\bbL_k}^{\ast}
= {
\presub{[m_s]}\vc{e}_{\ang{\nu}}^{\top}\vc{U}_s^{\ast} \vc{U}_{s,k}
\over 
\presub{[m_s]}\vc{e}_{\ang{\nu}}^{\top}\vc{U}_s^{\ast} \vc{U}_{s,k} \vc{e}
},\qquad k \in \bbZ_+,\ s \in \ol{\bbZ}_k,~\nu=1,2,\dots,m_s.
\label{eqn-170630-01}
\end{equation}
Therefore, from (\ref{eqn-170630-01}), Theorem~\ref{thm-submatrix-(n)F} and $\vc{\pi}_{\bbL_k}=\vc{\pi}_k$, we obtain
\[
\lim_{s\to\infty} 
{
\presub{[m_s]}\vc{e}_{\ang{\nu}}^{\top}\vc{U}_s^{\ast} \vc{U}_{s,k}
\over 
\presub{[m_s]}\vc{e}_{\ang{\nu}}^{\top}\vc{U}_s^{\ast} \vc{U}_{s,k} \vc{e}
}
= {\vc{\pi}_k \over \vc{\pi}_k \vc{e} }\quad \mbox{for all $\nu=1,2,\dots,m_s$}, \]
which implies that (\ref{limt-Takine-02-add}) holds.
\qed
\end{proof}

\medskip

The first limit formula (\ref{limt-Takine-02}) of Corollary~\ref{coro-Takine-thm3} holds for 
the general ergodic generator $\vc{Q}$ in the
upper block-Hessenberg form (\ref{defn-Q-MG1-type}). On the other hand, the second one (\ref{limt-Takine-02-add}) requires the additional condition that (\ref{cond-170514-01}) holds for some $k \in \bbZ_+$.
A similar formula to (\ref{limt-Takine-02-add}) is presented by Takine~\cite{Taki16} (see Theorem~3 therein):
\begin{equation}
\lim_{s\to\infty}
\diag^{-1}\{\vc{U}_{s,k}\vc{e}\} \cdot \vc{U}_{s,k} 
= { \vc{e}\vc{\pi}_k \over \vc{\pi}_k \vc{e} },
\qquad k \in \bbZ_+,
\label{limt-Takine-03}
\end{equation}
under \cite[Assumption 1]{Taki16}\footnote{This additional assumption is removed by \cite{M.Kimu18}.} that, for
all sufficiently large $s \in \bbN$, the $\vc{Q}_{s,s-1}$'s are nonsingular
and the $\vc{Q}_{s,s}$'s are of the same order. 

Assumption 1 in \cite{Taki16} implies that $\vc{U}_{s,k}\vc{e} > \vc{0}$ for all sufficiently large
$s \in \bbN$ (see (21) therein) and thus (\ref{cond-170514-01}) holds. Therefore, under this assumption, we have (\ref{limt-Takine-02-add}). In fact, from
(\ref{limt-Takine-02-add}), we can derive 
 (\ref{limt-Takine-03}) as follows:
\begin{eqnarray*}
\lefteqn{
\lim_{s\to\infty}
\diag^{-1}\{\vc{U}_{s,k}\vc{e}\}
\cdot \vc{U}_{s,k} 
}
\quad &&
\nonumber
\\
&=& \lim_{s\to\infty}
\diag^{-1}\{\vc{U}_{s,k}\vc{e}\}
(\vc{U}_s^{\ast})^{-1}
\diag\{\vc{U}_s^{\ast} \vc{U}_{s,k}\vc{e}\}
\nonumber
\\
&& {} \qquad~~ \times 
\diag^{-1}\{\vc{U}_s^{\ast} \vc{U}_{s,k}\vc{e}\}
\cdot \vc{U}_s^{\ast} \vc{U}_{s,k} 
\nonumber
\\
&=& \lim_{s\to\infty}
\diag^{-1}\{\vc{U}_{s,k}\vc{e}\}
\cdot (\vc{U}_s^{\ast})^{-1}
\diag\{\vc{U}_s^{\ast} \vc{U}_{s,k}\vc{e}\} 
\cdot {\vc{e} \vc{\pi}_k \over \vc{\pi}_k \vc{e} }
\nonumber
\\
&=& \lim_{s\to\infty}
\diag^{-1}\{\vc{U}_{s,k}\vc{e}\}
\cdot \vc{U}_{s,k}\vc{e}
{\vc{\pi}_k \over \vc{\pi}_k \vc{e} }
\nonumber
\\
&=& {\vc{e}\vc{\pi}_k \over \vc{\pi}_k \vc{e} },
\end{eqnarray*}
where the third equality holds because 
\begin{eqnarray*}
(\vc{U}_s^{\ast})^{-1}
\diag\{\vc{U}_s^{\ast} \vc{U}_{s,k}\vc{e}\} \vc{e}
&=& (\vc{U}_s^{\ast})^{-1}
\vc{U}_s^{\ast} \vc{U}_{s,k}\vc{e}
= \vc{U}_{s,k}\vc{e}.
\end{eqnarray*}
Similarly, from (\ref{limt-Takine-03}), we can also derive (\ref{limt-Takine-02-add}) and thus (\ref{limt-Takine-02}).

Based on (\ref{eqn-pi_l-pi_k*U_{k,l}}) and (\ref{limt-Takine-03}),
Takine~\cite{Taki16} proposes an algorithm that computes
$\{\vc{\pi}_k;k \in \bbZ_N\}$ for a given (sufficiently large) $N$. The outline of the
algorithm is as follows (for details, see \cite[Section 3]{Taki16}): 
\begin{enumerate}
\item Start with choosing $N \in \bbN$ sufficiently large; 
\item compute
\begin{equation}
\vc{x}_{s,N} 
:= m_s^{-1}\vc{e}^{\top}
\diag^{-1}\{\vc{U}_{s,N}\vc{e}\} \cdot \vc{U}_{s,N},
\label{defn-ol{x}_{s,N}}
\end{equation}
for a sufficiently large $s \in \ol{\bbZ}_N$, where $m_s$ is equal to the
cardinality of $\bbL_s$; and 
\item compute
\begin{equation*}
\vc{x}_{s,k}
:= \vc{x}_{s,N} \vc{U}_{N,k},\qquad k = N-1,N-2,\dots,0,
\end{equation*}
and then
\begin{equation}
\vc{x}_{s,k}^{(N)}
:= {
\vc{x}_{s,k} 
\over 
\sum_{\ell=0}^N \vc{x}_{s,\ell}\vc{e}
}
= {
\vc{x}_{s,N} \vc{U}_{N,k}
\over 
\sum_{\ell=0}^N \vc{x}_{s,N} \vc{U}_{N,\ell} \vc{e}
},
\qquad k \in \bbZ_N.
\label{defn-breve{x}_{s,k}}
\end{equation}
\end{enumerate}

This algorithm generates the
probability vector
$\vc{x}_s^{(N)}:=(\vc{x}_{s,0}^{(N)},\vc{x}_{s,1}^{(N)},\dots,\vc{x}_{s,N}^{(N)})$, which
can be considered an approximation to $\vc{\pi}=(\vc{\pi}_0,\vc{\pi}_1,\dots)$. Indeed, applying (\ref{limt-Takine-03}) to (\ref{defn-ol{x}_{s,N}}) yields
\begin{equation}
\lim_{s\to\infty} \vc{x}_{s,N} 
= {\vc{\pi}_N \over \vc{\pi}_N \vc{e} };
\label{lim-ol{x}_{s,N}}
\end{equation}
and combining (\ref{lim-ol{x}_{s,N}}), (\ref{defn-breve{x}_{s,k}}) and (\ref{eqn-pi_l-pi_k*U_{k,l}}) leads to
\begin{eqnarray*}
\lim_{s \to \infty} \vc{x}_{s,k}^{(N)}
&=& {
\vc{\pi}_N \vc{U}_{N,k}
\over 
\sum_{\ell=0}^N \vc{\pi}_N \vc{U}_{N,\ell} \vc{e}
} 
= {
\vc{\pi}_k
\over 
\sum_{\ell=0}^N \vc{\pi}_{\ell} \vc{e}
},
\qquad k \in \bbZ_N,
\end{eqnarray*}
which results in
\begin{equation*}
\lim_{N\to\infty} \lim_{s \to \infty} \vc{x}_{s,k}^{(N)}
= \lim_{N\to\infty} 
{
\vc{\pi}_k
\over 
\sum_{\ell=0}^N \vc{\pi}_{\ell} \vc{e}
} 
= \vc{\pi}_k,\qquad k \in \bbZ_+.
\end{equation*}

Takine's algorithm~\cite{Taki16}, described above, requires the truncation parameter $N \in \bbN$. On the other hand, Algorithm~\ref{algo-upper} does not require such a parameter, though it {\it does require} the conditions of Theorem~\ref{thm-MIP-upper} for its convergence (see Remark~\ref{add-rem-algo-1}). Actually, provided that the truncation parameter $N \in \bbN$ is given, we can establish an alternative algorithm to Algorithm~\ref{algo-upper}, based on the following result.
\begin{coro}[Double-limit-MIP-form solution]\label{coro-mu_{s,k}^{(N)}}
Suppose that the ergodic generator $\vc{Q}$ is in upper block-Hessenberg form
(\ref{defn-Q-MG1-type}). For $N \in \bbN$ and $s \in \ol{\bbZ}_N$, let $\vc{\eta}_s^{(N)} := (\vc{\eta}_{s,0}^{(N)},\vc{\eta}_{s,1}^{(N)},\dots,\vc{\eta}_{s,N}^{(N)})$ denote a probability vector such that
\begin{equation}
\vc{\eta}_{s,k}^{(N)}
=  {
\vc{e}^{\top}  \vc{U}_s^{\ast}\vc{U}_{s,k}
\over 
\vc{e}^{\top}  \sum_{\ell=0}^N \vc{U}_s^{\ast} \vc{U}_{s,\ell} \vc{e}
}, \qquad k \in \bbZ_N,\ s \in \ol{\bbZ}_N. 
\label{defn-mu_{s,k}^{(N)}}
\end{equation}
Furthermore, let $\vc{\pi}^{(N)} := (\vc{\pi}_0^{(N)},\vc{\pi}_1^{(N)},\dots,\vc{\pi}_{N}^{(N_)})$ denote a probability vector such that
\begin{equation}
\vc{\pi}_k^{(N)}
=
{
\vc{\pi}_k
\over 
\sum_{\ell=0}^{N} \vc{\pi}_{\ell} \vc{e}
},\qquad k \in \bbZ_N.
\label{defn-pi_k^{(N)}}
\end{equation}
We then have
\begin{equation}
\lim_{s\to\infty} \| \vc{\eta}_s^{(N)} - \vc{\pi}^{(N)}\| = 0,
\label{convergence-eta_s^{(N)}}
\end{equation}
and thus
\begin{eqnarray}
\vc{\pi}_k
&=&
\lim_{N \to \infty}\lim_{s \to \infty}
{
\vc{e}^{\top}  \vc{U}_s^{\ast}\vc{U}_{s,k}
\over 
\vc{e}^{\top}  \sum_{\ell=0}^N \vc{U}_s^{\ast} \vc{U}_{s,\ell} \vc{e}
},
\qquad k \in \bbZ_+.
\label{MIP-upper-submatrix}
\end{eqnarray}
\end{coro}

\medskip

\begin{proof}
This corollary can be proved in the same way as the proof of (\ref{limt-Takine-02}) in Corollary~\ref{coro-Takine-thm3}. 
\qed
\end{proof}

\begin{rem}\label{rem-coro-algo-upper-02}
Since the $\vc{\eta}_s^{(N)}$'s are probability vectors of finite order $N$, it follows from (\ref{convergence-eta_s^{(N)}}) that $\{\vc{\eta}_s^{(N)}; s \in \ol{\bbZ}_N\}$ is a Cauchy sequence, which leads to
\[
\lim_{s\to\infty}\| \vc{\eta}_s^{(N)} - \vc{\eta}_{s-1}^{(N)} \|=0.
\]
\end{rem}

It follows from (\ref{defn-U_{s,l}}) and (\ref{MIP-upper-submatrix}) that
\begin{eqnarray}
\vc{\pi}_k
&=& \lim_{N \to \infty}\lim_{s \to \infty}
{ 
\vc{e}^{\top} \vc{U}_s^{\ast}
\vc{U}_{s-1} \vc{U}_{s-2} \cdots \vc{U}_k
\over 
\vc{e}^{\top}
\sum_{\ell=0}^N \vc{U}_s^{\ast} \vc{U}_{s-1} \vc{U}_{s-2} \cdots \vc{U}_{\ell} \vc{e}
},
\qquad k \in \bbZ_+.
\label{MIP-upper-submatrix-02}
\end{eqnarray}
We refer to the limit formula (\ref{MIP-upper-submatrix-02}) (and thus (\ref{MIP-upper-submatrix})) as the {\it double-limit MIP-form solution} to distinguish the difference of it from (\ref{pi_l-MIP-upper-case-02}). 
The double-limit MIP-form solution (\ref{MIP-upper-submatrix})
shows that, for all sufficiently large $N \in \bbN$ and $s \in \ol{\bbZ}_N$, the probability vector $\vc{\eta}_s^{(N)}$ is an approximation to $\vc{\pi} = (\vc{\pi}_0,\vc{\pi}_1,\dots)$. 

From this perspective, we consider a way of computing $\vc{\pi}$.
We define $\vc{u}_s^{(N)}$, $s \in \bbZ_+$, as
\begin{equation}
\vc{u}_s^{(N)}
= \sum_{\ell=0}^{\min(N,s)} \vc{U}^{\ast}_{s,\ell}\vc{e}
= \sum_{\ell=0}^{\min(N,s)} \vc{U}_s^{\ast} \vc{U}_{s,\ell}\vc{e},
\qquad s \in \bbZ_+,
\label{defn-u_s^{(N)}}
\end{equation}
where the second equality follows from (\ref{defn-overline{U}_{s,l}}).
Using (\ref{defn-overline{U}_{s,l}}) and (\ref{defn-u_s^{(N)}}), we rewrite (\ref{defn-mu_{s,k}^{(N)}}) as
\begin{equation}
\vc{\eta}_{s,k}^{(N)}
=  {
\vc{e}^{\top} \vc{U}_{s,k}^{\ast}
\over 
\vc{e}^{\top} \vc{u}_s^{(N)}
}, \qquad k \in \bbZ_N,\ s \in \ol{\bbZ}_N.
\label{defn-mu_{s,k}^{(N)}-02}
\end{equation}
Note here that (\ref{defn-overline{u}_s}) and (\ref{defn-u_s^{(N)}}) lead to $\vc{u}_s^{(N)} = \vc{u}_s^{\ast}$ for $s \in \bbZ_N$. Thus, (\ref{eqn-u^*_0}) and (\ref{recursion-u^*_s}) yield
\begin{eqnarray}
\vc{u}_0^{(N)}
&=& \vc{U}_0^{\ast} \vc{e} = (-\vc{Q}_{0,0})^{-1}\vc{e},
\nonumber
\\
\vc{u}_s^{(N)}
&=& \vc{U}_s^{\ast} 
\left( \vc{e} + \vc{Q}_{s,s-1} \vc{u}_{s-1}^{(N)}
\right),
\qquad s=1,2,\dots,N.
\label{recursion-u_s^{(N)}-s<=N}
\end{eqnarray}
In addition, it follows from (\ref{eqn-U_{s-1,l}}) and (\ref{defn-u_s^{(N)}}) with $s > N$ that
\begin{eqnarray}
\vc{u}_s^{(N)}
&=& \vc{U}_s^{\ast} \vc{Q}_{s,s-1} \vc{u}_{s-1}^{(N)},
\qquad s \in \ol{\bbZ}_N =\{N+1,N+2,\dots\}.
\label{recursion-u_s^{(N)}-s>N}
\end{eqnarray}
Combining (\ref{recursion-u_s^{(N)}-s<=N}) and (\ref{recursion-u_s^{(N)}-s>N}) results in
\begin{equation}
\vc{u}_s^{(N)}
=
\left\{
\begin{array}{ll}
\vc{U}_s^{\ast} 
\left( \vc{e} + \vc{Q}_{s,s-1} \vc{u}_{s-1}^{(N)}
\right), & s=1,2,\dots,N,
\\
\vc{U}_s^{\ast} \vc{Q}_{s,s-1} \vc{u}_{s-1}^{(N)},& s = N+1,N+2,\dots.
\end{array}
\right.
\label{recursion-u_s^{(N)}}
\end{equation}
Consequently, we can compute the approximation $\vc{\eta}_s^{(N)}$ to $\vc{\pi}$ by the following algorithm.

{
\renewcommand{\baselinestretch}{0.8}\selectfont
\begin{algorithm}[H]
\caption{Upper block-Hessenberg Markov chain}\label{algo-upper-02}
{\bf Input}: $\vc{Q}$, $\varepsilon \in (0,1)$, $\{\vc{\alpha}_0,\vc{\alpha}_1,\dots,\vc{\alpha}_N\}$ and the truncation parameter $N \in \bbN$.
\\ 
{\bf Output}: 
$\vc{\eta}_s^{(N)} = (\vc{\eta}_{s,0}^{(N)},\vc{\eta}_{s,1}^{(N)},\dots,\vc{\eta}_{s,N}^{(N)})$,
where $s \in \bbZ_+$ is fixed when the iteration stops.
%
\smallskip
\begin{enumerate}
\item Set $s = 0$.
\item Compute $\vc{U}_0^{\ast} = \vc{U}_{0,0}^{\ast} =
  (-\vc{Q}_{0,0})^{-1}$ and $\vc{u}_0^{(N)} = \vc{U}_0^{\ast}\vc{e}$.
\item Iterate the following:
\begin{enumerate}
\item Increment $s$ by one.
\item Compute $\vc{U}_s^{\ast}=\vc{U}_{s,s}^{\ast}$ by (\ref{eqn-U_s^*}).
\item Compute $\{\vc{U}_{s,k}^{\ast};k \in \bbZ_{s-1}\}$ by
  (\ref{eqn-U_{s-1,l}}) and $\vc{u}_s^{(N)}$ by (\ref{recursion-u_s^{(N)}}).
\item If $s \ge N+1$, compute $\{\vc{\eta}_{s,k}^{(N)};k\in\bbZ_N\}$ by (\ref{defn-mu_{s,k}^{(N)}-02}).
\item If $s \ge N+2$ and $\| \vc{\eta}_s^{(N)} - \vc{\eta}_{s-1}^{(N)} \| < \varepsilon$, then stop the iteration; otherwise return to step~(a).
\end{enumerate}
\end{enumerate}
\end{algorithm}
\renewcommand{\baselinestretch}{1.1}\selectfont
}

\medskip

Algorithm~\ref{algo-upper-02} always stops after a finite number of iterations (see Remark~\ref{rem-coro-algo-upper-02}), though it requires the truncation parameter $N \in \bbN$. In fact, the parameter $N$ can be determined by using the $\vc{f}$-modulated drift condition (for details, see \cite{Kont16} and \cite[Section 14.2.1]{Meyn09}).
\begin{cond}[$\vc{f}$-modulated drift condition]\label{cond-f-modulated-drift}
There exist some $b \in (0,\infty)$, column vectors $\vc{v}:=(v(i))_{i\in\bbZ_+} \ge \vc{0}$ and
$\vc{f}:=(f(i))_{i\in\bbZ_+} \ge \vc{e}$ and finite set $\bbC \subset \bbZ_+$ such that
\begin{equation}
\vc{Q}\vc{v} \le  - \vc{f} + b \vc{1}_{\bbC}.
\label{ineqn-QV<=-f+b1_C}
\end{equation}
\end{cond}

It is implied in \cite[Theorem~1.1]{Kont16} that Condition~\ref{cond-f-modulated-drift} holds if and only if $\vc{Q}$ is ergodic, provided $\vc{Q}$ is irreducible. It also follows from (\ref{ineqn-QV<=-f+b1_C}) that $\vc{\pi}\vc{f} \le b$ and thus
\begin{equation}
\pi(i) \le {b \over f(i)},\qquad i \in \bbZ_+.
\label{ineqn-pi}
\end{equation}
We can use this inequality to estimate the tail of the stationary distribution vector $\vc{\pi}$.

We assume that
\[
\sum_{i=0}^{\infty} {1 \over f(i)} 
= \sum_{k=0}^{\infty} \sum_{i\in\bbL_k}
{1 \over f(i)}< \infty,
\]
where the second equality is due to (\ref{defn-L_s}).
Thus, for any $\varepsilon > 0$, there exists some $N_{\varepsilon} \in \bbN$ such that
\begin{equation}
\sum_{k=N_{\varepsilon}+1}^{\infty} \sum_{i\in\bbL_k}
{b \over f(i)} \le {\varepsilon \over 2}.
\label{defn-N_{epsilon}}
\end{equation}
Combining (\ref{ineqn-pi}) and (\ref{defn-N_{epsilon}}) yields
\begin{equation}
\sum_{k=N_{\varepsilon}+1}^{\infty} \vc{\pi}_k\vc{e}
= \sum_{k=N_{\varepsilon}+1}^{\infty} \sum_{i\in\bbL_k}
\pi(i)
\le {\varepsilon \over 2}.
\label{bound-sum-pi}
\end{equation}
We also define $\vc{\pi}^{(N_{\varepsilon})} := (\vc{\pi}_0^{(N_{\varepsilon})},\vc{\pi}_1^{(N_{\varepsilon})},\dots,\vc{\pi}_{N_{\varepsilon}}^{(N_{\varepsilon})})$ as a probability vector such that 
\begin{equation*}
\vc{\pi}_k^{(N_{\varepsilon})}
=
{
\vc{\pi}_k
\over 
\sum_{\ell=0}^{N_{\varepsilon}} \vc{\pi}_{\ell} \vc{e}
},\qquad k \in \bbZ_{N_{\varepsilon}}.
\end{equation*}
Using (\ref{bound-sum-pi}), we have
\begin{eqnarray}
\| \vc{\pi}^{(N_{\varepsilon})} - \vc{\pi} \|
&=& \sum_{k=0}^{N_{\varepsilon}} 
(\vc{\pi}_k^{(N_{\varepsilon})} - \vc{\pi}_k)\vc{e}
+ \sum_{k=N_{\varepsilon}+1}^{\infty} \vc{\pi}_k \vc{e}
\nonumber
\\
&=& \sum_{k=0}^{N_{\varepsilon}} 
\left({ \vc{\pi}_k \over  \sum_{\ell=0}^{N_{\varepsilon}} \vc{\pi}_{\ell} \vc{e} } 
- \vc{\pi}_k \right)\vc{e}
+ \sum_{k=N_{\varepsilon}+1}^{\infty} \vc{\pi}_k \vc{e}
\nonumber
\\
&=& 1 - \sum_{k=0}^{N_{\varepsilon}}\vc{\pi}_k \vc{e}
+ \sum_{k=N_{\varepsilon}+1}^{\infty} \vc{\pi}_k \vc{e}
\nonumber
\\
&=& 2\sum_{k=N_{\varepsilon}+1}^{\infty} \vc{\pi}_k \vc{e}
\le \varepsilon,
\label{error-bound-pi^{(N)}}
\end{eqnarray}
which shows that $\vc{\pi}$ is approximated by $\vc{\pi}^{(N_{\varepsilon})}$ within error $\varepsilon$ measured in terms of the total variation distance. Furthermore, from (\ref{error-bound-pi^{(N)}}) and (\ref{convergence-eta_s^{(N)}}), we have
\begin{eqnarray*}
\| \vc{\eta}_s^{(N_{\varepsilon})} - \vc{\pi} \|
&\le& \| \vc{\eta}_s^{(N_{\varepsilon})} - \vc{\pi}^{(N_{\varepsilon})} \| 
+
\| \vc{\pi}^{(N_{\varepsilon})} - \vc{\pi} \|
\nonumber
\\
&\le& \| \vc{\eta}_s^{(N_{\varepsilon})} - \vc{\pi}^{(N_{\varepsilon})} \|  + \varepsilon \to \varepsilon\quad \mbox{as $s \to \infty$}.
\end{eqnarray*}
Therefore, for a sufficiently large $s \in \ol{\bbZ}_N$, the total variation error of the approximation $\vc{\eta}_s^{(N_{\varepsilon})}$ for $\vc{\pi}$ is less than the tolerance $\varepsilon$. Unfortunately, it is in general difficult to determine such a value of $s$ in advance. 

We now go back to Algorithm~\ref{algo-upper}. Although Algorithm~\ref{algo-upper} requires the convergence conditions (see Remark~\ref{add-rem-algo-1}), it is free from the problem of determining parameter $N$. Algorithm~\ref{algo-upper} also has another remarkable feature. Recall here that
Algorithm~\ref{algo-upper} generates a sequence of the
linear-augmented truncation approximations
$\presub{(s)}\wh{\vc{\pi}}$'s. Therefore, we can obtain an upper bound
for $\| \presub{(s)}\wh{\vc{\pi}} - \vc{\pi} \|$, following the studies
\cite{LiuYuan10,LiuYuan15,Masu17-LAA,Masu15-ADV,Masu16-SIAM,Masu17-JORSJ,Twee98} on the error estimation of the truncation approximation of Markov
chains (therein the $\vc{f}$-modulated drift condition plays an important role). Using such an upper bound, we can establish
sophisticated stopping criteria for Algorithm~\ref{algo-upper}, which
guarantee the accuracy of the resulting approximation to
$\vc{\pi}$. The details of this topic are beyond the scope of this
paper and thus are omitted here.

\subsection{Lower block-Hessenberg Markov chain}\label{subsec-lower-case}

We assume that the ergodic generator $\vc{Q}$ is in lower
block-Hessenberg form  (i.e., is of level-dependent GI/M/1-type):
\begin{equation}
\vc{Q} = 
\bordermatrix{
        & \bbL_0 & \bbL_1    & \bbL_2    & \bbL_3   & \cdots
\cr
\bbL_0 		& 
\vc{Q}_{0,0}  	& 
\vc{Q}_{0,1} 	& 
\vc{O}  		& 
\vc{O}  		& 
\cdots
\cr
\bbL_1 		&
\vc{Q}_{1,0}  	& 
\vc{Q}_{1,1} 	& 
\vc{Q}_{1,2}  	& 
\vc{O}  		& 
\cdots
\cr
\bbL_2 		& 
\vc{Q}_{2,0}  	& 
\vc{Q}_{2,1} 	& 
\vc{Q}_{2,2}  	& 
\vc{Q}_{2,3} 	&  
\cdots
\cr
\bbL_3 & 
\vc{Q}_{3,0}  	& 
\vc{Q}_{3,1} 	& 
\vc{Q}_{3,2}  	& 
\vc{Q}_{3,3} 	&  
\cdots
\cr
~\vdots  	& 
\vdots     		& 
\vdots     		&  
\vdots    		& 
\vdots    		& 
\ddots
},
\label{defn-Q-lower}
\end{equation}
where $\vc{Q}_{k,\ell} = \vc{O}$ for $k\in\bbZ_+$ and
$\ell\in\bbZ_{k+1}$. We then have
\begin{equation*}
\presub{(s)}\vc{Q}
=
\left(
\begin{array}{ccccccc}
\vc{Q}_{0,0} 	&
\vc{Q}_{0,1} 	&
\vc{O} 			&
\cdots			&
\vc{O} 			&
\vc{O} 			&
\vc{O} 			
\\
\vc{Q}_{1,0} 	&
\vc{Q}_{1,1} 	&
\vc{Q}_{1,2} 	&
\cdots			&
\vc{O} 			&
\vc{O} 			&
\vc{O} 			
\\
\vc{Q}_{2,0} 	&
\vc{Q}_{2,1} 	&
\vc{Q}_{2,2} 	&
\cdots			&
\vc{O} 			&
\vc{O} 			&
\vc{O} 	
\\
\vdots			&
\vdots			&
\vdots			&
\ddots			&
\vdots			&
\vdots			&
\vdots			
\\
\vc{Q}_{s-2,0} 	&
\vc{Q}_{s-2,1} 	&
\vc{Q}_{s-2,2} 	&
\cdots			&
\vc{Q}_{s-2,s-2} 	&
\vc{Q}_{s-2,s-1} 	&
\vc{O}
\\
\vc{Q}_{s-1,0} 	&
\vc{Q}_{s-1,1} 	&
\vc{Q}_{s-1,2} 	&
\cdots			&
\vc{Q}_{s-1,s-2} 	&
\vc{Q}_{s-1,s-1} 	&
\vc{Q}_{s-1,s} 
\\
\vc{Q}_{s,0} 	&
\vc{Q}_{s,1} 	&
\vc{Q}_{s,2} 	&
\cdots			&
\vc{Q}_{s,s-2} 	&
\vc{Q}_{s,s-1} 	&
\vc{Q}_{s,s} 
\\
\end{array}
\right).
\end{equation*}

We permutate the columns and rows of $\presub{(s)}\vc{Q}$ by
arranging the subsets $\{\bbL_k;k=0,1,\dots,s\}$ of the state space
$\bbZ_+$ in the descending order
$(\bbL_s,\bbL_{s-1},\dots,\bbL_0)$. We denote the resulting matrix by
$\presub{(s)}\wt{\vc{Q}}$. Clearly, $\presub{(s)}\wt{\vc{Q}}$ is in
the same form as $\presub{(s)}\vc{Q}$ in
(\ref{partition-(n_s)Q-upper}), i.e., in the upper block-Hessenberg
form:
\begin{equation}
\presub{(s)}\wt{\vc{Q}}
=
\left(
\begin{array}{ccccccc}
\vc{Q}_{s,s} 	&
\vc{Q}_{s,s-1} 	&
\vc{Q}_{s,s-2} 	&
\cdots			&
\vc{Q}_{s,2} 	&
\vc{Q}_{s,1} 	&
\vc{Q}_{s,0} 	
\\
\vc{Q}_{s-1,s} 	&
\vc{Q}_{s-1,s-1} 	&
\vc{Q}_{s-1,s-2} 	&
\cdots			&
\vc{Q}_{s-1,2} 	&
\vc{Q}_{s-1,1} 	&
\vc{Q}_{s-1,0} 	
\\
\vc{O} 			&
\vc{Q}_{s-2,s-1} 	&
\vc{Q}_{s-2,s-2} 	&
\cdots			&
\vc{Q}_{s-2,2} 	&
\vc{Q}_{s-2,1} 	&
\vc{Q}_{s-2,0} 	
\\
\vdots			&
\vdots			&
\vdots			&
\ddots			&
\vdots			&
\vdots			&
\vdots			
\\
\vc{O}			&
\vc{O}			&
\vc{O}			&
\cdots			&
\vc{Q}_{1,2} 	&
\vc{Q}_{1,1} 	&
\vc{Q}_{1,0} 
\\
\vc{O}			&
\vc{O}			&
\vc{O}			&
\cdots			&
\vc{O}			&
\vc{Q}_{0,1} 	&
\vc{Q}_{0,0} 
\\
\end{array}
\right).
\label{partition-(n_s)wt{Q}-upper}
\end{equation}
We also partition $(-\presub{(s)}\wt{\vc{Q}})^{-1}$ as
\begin{equation}
(-\presub{(s)}\wt{\vc{Q}})^{-1} 
= 
\bordermatrix{
        & 
\bbL_s & 
\bbL_{s-1} & 
\cdots  & 
\bbL_0 
\cr
\bbL_s						& 
\presub{(s)}\vc{Y}_{0,0}  	& 
\presub{(s)}\vc{Y}_{0,1}  	& 
\cdots						&
\presub{(s)}\vc{Y}_{0,s}  	 
\cr
\bbL_{s-1}					& 
\presub{(s)}\vc{Y}_{1,0}  	& 
\presub{(s)}\vc{Y}_{1,1}  	& 
\cdots						&
\presub{(s)}\vc{Y}_{1,s}  	 
\cr
~\vdots						&
\vdots						&
\vdots						&
\ddots						&
\vdots						
\cr
\bbL_0						& 
\presub{(s)}\vc{Y}_{s,0}  	& 
\presub{(s)}\vc{Y}_{s,1}  	& 
\cdots						&
\presub{(s)}\vc{Y}_{s,s}  	 
}.
\label{eqn-170306-01}
\end{equation}
It also follows from (\ref{eqn-(s)ol{pi}_s}), (\ref{eqn-170306-01}) and the definition of $\presub{(s)}\wt{\vc{Q}}$ that
\begin{eqnarray}
(\presub{(s)}\wc{\vc{\pi}}_{s},\presub{(s)}\wc{\vc{\pi}}_{s-1},\dots,
\presub{(s)}\wc{\vc{\pi}}_{0})
&=& { (\vc{0},\dots,\vc{0},\vc{\alpha}_0) (- \presub{(s)}\wt{\vc{Q}})^{-1} 
\over 
(\vc{0},\dots,\vc{0}, \vc{\alpha}_0)(- \presub{(s)}\wt{\vc{Q}})^{-1} \vc{e}
}
\nonumber
\\
&=& { \vc{\alpha}_0 (\presub{(s)}\vc{Y}_{s,0},\presub{(s)}\vc{Y}_{s,1},\dots,\presub{(s)}\vc{Y}_{s,s}) 
\over 
\vc{\alpha}_0 \sum_{\ell=0}^s \presub{(s)}\vc{Y}_{s,\ell} \vc{e}
}, \qquad s \in\bbZ_+.
\label{eqn-(s)ol{pi}_0_reverse}
\end{eqnarray}

We now define
$\{\presub{(s)}\wt{\vc{U}}_k^{\ast};k\in\bbZ_s\}$,$s \in \bbZ_+$, recursively  as follows: 
\begin{equation}
\presub{(s)}\wt{\vc{U}}_k^{\ast}
= 
\left(
-\vc{Q}_{s-k,s-k} 
- \dm\sum_{\ell=0}^{k-1} \presub{(s)}\wt{\vc{U}}_{k,\ell} \vc{Q}_{s-\ell,s-k}
\right)^{-1},
\qquad k \in \bbZ_s,
\label{defn-(n_s)wt{U}_k^*}
\end{equation}
where $\presub{(s)}\wt{\vc{U}}_{k,\ell}$'s, $k=1,2,\dots,s$, $\ell \in
\bbZ_{k-1}$, are given by
\begin{eqnarray}
\presub{(s)}\wt{\vc{U}}_{k,\ell}
&=&
(\vc{Q}_{s-k,s-k+1}   \presub{(s)}\wt{\vc{U}}_{k-1}^{\ast}) 
\nonumber
\\
&& {}
\times 
(\vc{Q}_{s-k+1,s-k+2} \presub{(s)}\wt{\vc{U}}_{k-2}^{\ast})  \cdots 
(\vc{Q}_{s-\ell-1,s-\ell}\presub{(s)}\wt{\vc{U}}_{\ell}^{\ast}).
\label{defn-(n_s)wt{U}_{k,l}}
\end{eqnarray}
Note here that $\presub{(s)}\wt{\vc{U}}_k^{\ast}$ and
$\presub{(s)}\wt{\vc{U}}_{k,\ell}$ are obtained by replacing, with
$\vc{Q}_{s-k,s-\ell}$, $\vc{Q}_{k,\ell}$ in (\ref{defn-U_k^*}) and
(\ref{defn-U_{k,l}}), respectively. Thus, Lemma~\ref{lem-(n_s)X_{s,l}}
implies that
\begin{equation}
\presub{(s)}\vc{Y}_{s,\ell} 
= \presub{(s)}\wt{\vc{U}}_s^{\ast} \presub{(s)}\wt{\vc{U}}_{s,\ell},
\qquad s \in \bbZ_+,\ \ell \in \bbZ_s,
\label{eqn-(n_s)Y(0,l)}
\end{equation}
where $\presub{(s)}\wt{\vc{U}}_{\ell,\ell} = \vc{I}$ for $\ell \in
\bbZ_s$. Substituting (\ref{eqn-(n_s)Y(0,l)}) into (\ref{eqn-(s)ol{pi}_0_reverse}), we have
\begin{eqnarray}
\presub{(s)}\wc{\vc{\pi}}_{k}
&=&
{
\vc{\alpha}_0 \presub{(s)}\wt{\vc{U}}_s^{\ast}\, \presub{(s)}\wt{\vc{U}}_{s,s-k}
\over 
\vc{\alpha}_0 
\sum_{\ell=0}^s \presub{(s)}\wt{\vc{U}}_s^{\ast} \presub{(s)}\wt{\vc{U}}_{s,\ell} \vc{e}
}, \qquad k \in \bbZ_s,
\label{eqn-(n_s)pi_{0,l}-03}
\end{eqnarray}
which is the counterpart of the expression of $\presub{(s)}\wh{\vc{\pi}}_{k}$ presented in Lemma~\ref{lem-producet-form-(s)ol{pi}_s}.

Let
\begin{align}
\presub{(s)}\vc{R}_k^{\ast}
&= \presub{(s)}\vc{U}_{s-k}^{\ast},& s &\in \bbZ_+,\ k \in \bbZ_s,
\label{eqn-(s)R_k^*-02}
\\
\presub{(s)}\vc{R}_{k,\ell}
&= \presub{(s)}\vc{U}_{s-k,s-\ell},& s &\in \bbZ_+,\ k \in \bbZ_s,\ \
\ell = k,k+1,\dots,s.
\label{eqn-(s)R_{k,l}-02}
\end{align}
where $\presub{(s)}\vc{R}_{\ell,\ell} = \vc{I}$ for $\ell
\in \bbZ_s$.
Using (\ref{eqn-(s)R_k^*-02}) and (\ref{eqn-(s)R_{k,l}-02}), we rewrite (\ref{eqn-(n_s)pi_{0,l}-03}) as
\begin{eqnarray}
\presub{(s)}\wc{\vc{\pi}}_{k}
&=&
{
\vc{\alpha}_0 \presub{(s)}\vc{R}_0^{\ast}\, \presub{(s)}\vc{R}_{0,k}
\over 
\vc{\alpha}_0 
\sum_{\ell=0}^s \presub{(s)}\vc{R}_0^{\ast} \presub{(s)}\vc{R}_{0,\ell} \vc{e}
}, \qquad k \in \bbZ_s.
\label{eqn-(n_s)pi_{0,l}}
\end{eqnarray}
Furthermore, using (\ref{defn-(n_s)wt{U}_k^*}), (\ref{defn-(n_s)wt{U}_{k,l}}), (\ref{eqn-(s)R_k^*-02}) and (\ref{eqn-(s)R_{k,l}-02}), we establish the recursion of $\{\presub{(s)}\vc{R}_k^{\ast};s\in\bbZ_+, k\in\bbZ_s\}$: For
$k=s,s-1,\dots,0$,
\begin{eqnarray}
\presub{(s)}\vc{R}_k^{\ast}
&=& 
\left(
- \vc{Q}_{k,k} 
- \dm\sum_{\ell=k+1}^s \presub{(s)}\vc{R}_{k,\ell} \vc{Q}_{\ell,k}
\right)^{-1},
\label{defn-(s)R_k^*}
\end{eqnarray}
where $\presub{(s)}\vc{R}_{k,\ell}$'s, $k \in \bbZ_{s-1}$,
$\ell=k+1,k+2,\dots,s$, are given by
\begin{equation}
\presub{(s)}\vc{R}_{k,\ell}
= 
(\vc{Q}_{k,k+1} \presub{(s)}\vc{R}_{k+1}^{\ast})
(\vc{Q}_{k+1,k+2} \presub{(s)}\vc{R}_{k+2}^{\ast})
\cdots
(\vc{Q}_{\ell-1,\ell} \presub{(s)}\vc{R}_{\ell}^{\ast}).
\label{defn-(s)R_{k,l}}
\end{equation}

We now define $\presub{(s)}\vc{R}_k$, $s\in\bbN$, $k\in\bbZ_s$,
as
\begin{equation}
\presub{(s)}\vc{R}_k
= \vc{Q}_{k-1,k} \presub{(s)}\vc{R}_k^{\ast},
\qquad k \in \bbZ_s.
\label{defn-(s)R_k}
\end{equation}
Substituting (\ref{defn-(s)R_k}) into
(\ref{defn-(s)R_{k,l}}), we have, for $s \in \bbN$ and $k \in
\bbZ_{s-1}$,
\begin{eqnarray}
\presub{(s)}\vc{R}_{k,\ell}
= \presub{(s)}\vc{R}_{k+1} \presub{(s)}\vc{R}_{k+2}
\cdots
\presub{(s)}\vc{R}_{\ell}, 
\qquad \ell=k+1,k+2,\dots,s.
\label{defn-(s)R_{k,l}-02}
\end{eqnarray}
Using (\ref{defn-(s)R_{k,l}-02}), we rewrite (\ref{eqn-(n_s)pi_{0,l}})
as
\begin{eqnarray}
\presub{(s)}\wc{\vc{\pi}}_{k}
=
{
\vc{\alpha}_0 \presub{(s)}\vc{R}_0^{\ast}
\presub{(s)}\vc{R}_1 \presub{(s)}\vc{R}_2 \cdots \presub{(s)}\vc{R}_k
\over 
\vc{\alpha}_0 \sum_{\ell=0}^s \presub{(s)}\vc{R}_0^{\ast}
 \presub{(s)}\vc{R}_1 \presub{(s)}\vc{R}_2 \cdots 
\presub{(s)}\vc{R}_{\ell} \vc{e}
},
\qquad k \in \bbZ_+.  
\label{eqn-(n_s)pi_{0,l}-02}
\end{eqnarray}
Furthermore, it follows from Theorem~\ref{thm-limit-row-[n]F} and Definition~\ref{defn-simple} that if
Assumption~\ref{assumpt-1} holds then
\[
\lim_{s\to\infty} \|\presub{(s)}\wc{\vc{\pi}} - \vc{\pi}\| = 0.
\] 
Combining this with (\ref{eqn-(n_s)pi_{0,l}}) and
(\ref{eqn-(n_s)pi_{0,l}-02}) yields the (single-limit) MIP-form solution of
$\vc{\pi}=(\vc{\pi}_0,\vc{\pi}_1,\dots)$ in the lower block-Hessenberg
case, which is summarized in the following theorem.
\begin{thm}[MIP-form solution]\label{thm-MIP-lower}
Suppose that the ergodic generator $\vc{Q}$ is in lower
block-Hessenberg form (\ref{defn-Q-lower}). If
Assumption~\ref{assumpt-1} holds, then
\begin{eqnarray}
\vc{\pi}_k
&=& 
\lim_{s\to\infty}
{
\vc{\alpha}_0 \presub{(s)}\vc{R}_0^{\ast}\, \presub{(s)}\vc{R}_{0,k}
\over 
\vc{\alpha}_0 \sum_{\ell=0}^s \presub{(s)}\vc{R}_0^{\ast} \presub{(s)}\vc{R}_{0,\ell} \vc{e}
},
\qquad k \in \bbZ_+,
\label{pi_l-MIP-lower-case}
\end{eqnarray}
or equivalently,
\begin{eqnarray*}
\vc{\pi}_k
&=& \lim_{s\to\infty}
{
\vc{\alpha}_0 \presub{(s)}\vc{R}_0^{\ast}
\presub{(s)}\vc{R}_1 \presub{(s)}\vc{R}_2 \cdots \presub{(s)}\vc{R}_k
\over 
\vc{\alpha}_0 \sum_{\ell=0}^s \presub{(s)}\vc{R}_0^{\ast}
 \presub{(s)}\vc{R}_1 \presub{(s)}\vc{R}_2 \cdots 
\presub{(s)}\vc{R}_{\ell} \vc{e}
},\qquad k \in \bbZ_+.
\end{eqnarray*}
\end{thm}

\medskip

Using the MIP-form solution (\ref{pi_l-MIP-lower-case}), we
establish an algorithm that generates the sequence
$\{\presub{(s)}\wc{\vc{\pi}};s\in\bbZ_+\}$ convergent to $\vc{\pi}$ in
the lower block-Hessenberg case. The MIP-form solution (\ref{pi_l-MIP-lower-case}) consists of $\presub{(s)}\vc{R}_0^{\ast}$ and $\{\presub{(s)}\vc{R}_{0,k}\;k=1,2,\dots,s\}$. In order to obtain $\presub{(s)}\vc{R}_0^{\ast}$, we compute $\presub{(s)}\vc{R}_s^{\ast},\presub{(s)}\vc{R}_{s-1}^{\ast},  \dots,\presub{(s)}\vc{R}_1^{\ast}$ by (\ref{defn-(s)R_k^*}) and (\ref{defn-(s)R_{k,l}}), where $\presub{(s)}\vc{R}_s^{\ast} = (-\vc{Q}_{s,s})^{-1}$. Given the $\presub{(s)}\vc{R}_k^{\ast}$'s, we can compute $\{\presub{(s)}\vc{R}_{0,k};k=1,2,\dots,s\}$ by the recursion:
\begin{eqnarray}
\presub{(s)}\vc{R}_{0,k}
= \left\{
\begin{array}{ll}
\vc{I},&\quad k=0,
\\
\presub{(s)}\vc{R}_{0,k-1}\vc{Q}_{k-1,k}\, \presub{(s)}\vc{R}_k^{\ast},&\quad k=1,2,\dots,s,
\end{array}
\right.
\label{recursion-(s)R_{0,k}}
\end{eqnarray}
which follows from (\ref{defn-(s)R_{k,l}}).
It should be noted that, for different values of $s$, we have to independently
compute the component matrices $\presub{(s)}\vc{R}_0^{\ast}$ and
$\presub{(s)}\vc{R}_{0,k}$'s of the MIP-form solution
(\ref{pi_l-MIP-lower-case}). This fact implies that the algorithm
in the lower block-Hessenberg case (Algorithm~\ref{algo-lower} below)
is less effective than Algorithm~\ref{algo-upper} in the upper
block-Hessenberg case.

{
\renewcommand{\baselinestretch}{0.8}\selectfont
\begin{algorithm}[H]
\caption{Lower block-Hessenberg Markov chain}\label{algo-lower}
{\bf Input}: $\vc{Q}$, $\varepsilon \in (0,1)$, $\vc{\alpha}_0$.
\\ 
{\bf Output}: 
$\presub{(s)}\wc{\vc{\pi}} =
(\presub{(s)}\wc{\vc{\pi}}_{0},\presub{(s)}\wc{\vc{\pi}}_{1},\dots,\presub{(s)}\wc{\vc{\pi}}_{s})$, where $s \in \bbZ_+$
is fixed when the iteration stops.
%
\smallskip
\begin{enumerate}
\item Set $s = 0$.
\item Compute $\presub{(0)}\vc{R}_0^{\ast} = (-\vc{Q}_{0,0})^{-1}$.
\item Compute $\presub{(0)}\ol{\vc{\pi}}_0 = \vc{\alpha}_0\presub{(0)}\vc{R}_0^{\ast}/(\vc{\alpha}_0\presub{(0)}\vc{R}_0^{\ast}\vc{e})$.
\item Iterate the following:
\begin{enumerate}
\item Increment $s$ by one.
\item For $k=s,s-1,\dots,0$, compute $\presub{(s)}\vc{R}_k^{\ast}$ by (\ref{defn-(s)R_k^*}) and (\ref{defn-(s)R_{k,l}}). 
\item For $k=1,2,\dots,s$, compute $\presub{(s)}\vc{R}_{0,k}$ by (\ref{recursion-(s)R_{0,k}}).
\item Compute $\{\presub{(s)}\wc{\vc{\pi}}_{k};k\in\bbZ_s\}$ by (\ref{eqn-(n_s)pi_{0,l}}).
\item If $\| \presub{(s)}\wc{\vc{\pi}} - \presub{(s-1)}\wc{\vc{\pi}} \| < \varepsilon$, 
then stop the iteration;
otherwise return to step~(a).
\end{enumerate}
\end{enumerate}
\end{algorithm}
\renewcommand{\baselinestretch}{1.1}\selectfont
}

\begin{rem}\label{rem-lower-case}
Corollary~\ref{coro-limit-[n]ol{pi}} guarantees that Algorithm~\ref{algo-lower}
stops after a finite number of iterations if
Assumption~\ref{assumpt-1} holds.
\end{rem}

\begin{rem}
Algorithm~\ref{algo-lower} increments the iteration index $s$ one by
one and thus generates the probability vector $\presub{(s)}\wc{\vc{\pi}}$
of the smallest order that satisfies the stopping criterion $\|
\presub{(s)}\wc{\vc{\pi}} - \presub{(s-1)}\wc{\vc{\pi}} \| < \varepsilon$.
Unlike Algorithm~\ref{algo-upper}, however, the iterations of
Algorithm~\ref{algo-lower} are performed independently one
another. More specifically, for each $s \in \bbZ_+$,
Algorithm~\ref{algo-lower} computes $\{\presub{(s)}\vc{R}_k^{\ast};k \in
\bbZ_s\}$ from scratch and thus $s+1$ inverse matrices.  Therefore,
until the iteration index $s$ reaches $N \in \bbN$,
Algorithm~\ref{algo-lower} computes $(N+1)(N+2)/2$ inverse matrices
whereas Algorithm~\ref{algo-upper} computes $N+1$ inverse matrices
(see Remark~\ref{rem-algo-1}). To reduce the computational cost and
accelerate the convergence of the resulting probability vectors, we
can increment the iteration index $s$ in such a way that $s =
s_0,s_1,\dots$, where $\{s_i;i\in\bbZ_+\}$ is an increasing and
divergent sequence of nonnegative integers. A possible choice of
$\{s_i\}$ is that $s_i = 2^i - 1$ for $i \in \bbZ_+$. In this case,
$2^{i+1} - 1$ inverse matrices have been computed when the $i$-th
iteration ends, i.e., when level $2^i-1$ is the maximum of levels
involved in computing, which shows that the total number of inverse
matrices computed increases linearly with the maximum level.
\end{rem}

To the best of our knowledge, there are no previous studies on
computing the stationary distribution vector of the lower
BHMC, except for Baumann and Sandmann's work
\cite{Baum12-COR}. They proposed an algorithm for a special case of lower
BHMCs, which is referred to as the {\it
  level-dependent quasi-birth-and-death process (LD-QBD) with
  catastrophes} therein. Their algorithm is very similar to the ones
for ordinary LD-QBDs in \cite{Baum12-Procedia,Phun10-QTNA} and thus requires the
maximum $N \in \bbN$ of levels involved in computing. When $N$ is
given, Baumann and Sandmann's algorithm~\cite{Baum12-COR} generates $N$
inverse matrices by the backward recursion (\ref{recursion-R^{(k)}})
and computes the system of linear equations for
$\vc{\pi}_0$. Therefore, the computational complexity of their
algorithm is of the same order as that of Algorithm~\ref{algo-lower}
in the situation where the maximum level $N$ is determined by trial
and error.

\subsection{GI/M/1-type Markov chain}\label{subsec-GI-M-1}

In this subsection, we consider the GI/M/1-type Markov chain.  Since
the GI/M/1-type Markov chain a special case of lower block-Hessenberg
Markov chains, the results presented in this subsection can be
directly obtained from those in
Section~\ref{subsec-lower-case}. However, as we will see later, we can
establish an effective algorithm like Algorithm~\ref{algo-upper} for
the upper BHMC by using the special structure
of the GI/M/1-type Markov chain. To achieve this, we utilize the
results in Section~\ref{subsec-lower-case} in an (apparently) indirect
way.

We fix $s \in \bbN$ arbitrarily.  We assume that the ergodic generator
$\vc{Q}$ in (\ref{defn-Q-lower}) is reduced to
\begin{equation}
\vc{Q} = 
\bordermatrix{
        & \bbL_0 & \bbL_1    & \bbL_2    & \bbL_3   & \cdots
\cr
\bbL_0 		& 
\vc{B}_0  	& 
\vc{B}_1 	& 
\vc{O}  	& 
\vc{O}  	& 
\cdots
\cr
\bbL_1 		&
\vc{B}_{-1} & 
\vc{A}_0 	& 
\vc{A}_1 	& 
\vc{O}  	& 
\cdots
\cr
\bbL_2 		& 
\vc{B}_{-2}  & 
\vc{A}_{-1} & 
\vc{A}_0  	& 
\vc{A}_1 	&  
\cdots
\cr
\bbL_3 		& 
\vc{B}_{-3} & 
\vc{A}_{-2} & 
\vc{A}_{-1} & 
\vc{A}_0  	&
\cdots
\cr
~\vdots  	& 
\vdots     	& 
\vdots     	&  
\vdots    	& 
\vdots    	& 
\ddots
},
\label{defn-GIM1-type}
\end{equation}
In this case, $\presub{(s)}\wt{\vc{Q}}$ in
(\ref{partition-(n_s)wt{Q}-upper}) is reduced to
\begin{equation}
\presub{(s)}\wt{\vc{Q}}
=
\bordermatrix{
& \bbL_s
& \bbL_{s-1}
& \bbL_{s-2}
& \cdots
& \bbL_2
& \bbL_1
& \bbL_0
\cr
\bbL_s			&
\vc{A}_0 		&
\vc{A}_{-1} 	&
\vc{A}_{-2} 	&
\cdots			&
\vc{A}_{-s+2} 	&
\vc{A}_{-s+1} 	&
\vc{B}_{-s} 	
\cr
\bbL_{s-1}		&
\vc{A}_1 		&
\vc{A}_0 		&
\vc{A}_{-1} 	&
\cdots			&
\vc{A}_{-s+3} 	&
\vc{A}_{-s+2} 	&
\vc{B}_{-s+1} 	
\cr
\bbL_{s-2}		&
\vc{O} 			&
\vc{A}_1 		&
\vc{A}_0 		&
\cdots			&
\vc{A}_{-s+4} 	&
\vc{A}_{-s+3} 	&
\vc{B}_{-s+2} 	
\cr
~\vdots			&
\vdots			&
\vdots			&
\vdots			&
\ddots			&
\vdots			&
\vdots			&
\vdots			
\cr
\bbL_1			&
\vc{O}			&
\vc{O}			&
\vc{O}			&
\cdots			&
\vc{A}_1 	&
\vc{A}_0 	&
\vc{B}_{-1} 
\cr
\bbL_0			&
\vc{O}			&
\vc{O}			&
\vc{O}			&
\cdots			&
\vc{O}			&
\vc{B}_1 	&
\vc{B}_0
}.
\label{partition-(n_s)wt{Q}-GIM1}
\end{equation}
Note here that $\presub{(s)}\wt{\vc{Q}}$ in
(\ref{partition-(n_s)wt{Q}-GIM1}) is equivalent to
$\presub{(s)}\wt{\vc{Q}}$ in (\ref{partition-(n_s)wt{Q}-upper}) with
\begin{equation}
\vc{Q}_{k,\ell}
=
\left\{
\begin{array}{ll}
\vc{A}_{\ell - k}, &  \qquad  k,\ell \in \{1,2,\dots,s\},
\\
\vc{B}_{\ell - k}, & \qquad   k=0 \mbox{~or~} \ell=0,
\end{array}
\right.
\label{correspondence-Q}
\end{equation}
where $\ell \le k+1$. Substituting (\ref{correspondence-Q}) into
(\ref{defn-(n_s)wt{U}_k^*}) yields, for $s \in \bbN$,
\begin{align}
\presub{(s)}\wt{\vc{U}}_s^{\ast}
&= \left(
-\vc{B}_0
- \dm\sum_{\ell=0}^{s-1} \presub{(s)}\wt{\vc{U}}_{s,\ell} \vc{B}_{\ell-s}
\right)^{-1}, 
\label{defn-(s)R_0^*-GIM1}
\\
\presub{(s)}\wt{\vc{U}}_k^{\ast}
&=
\left(
-\vc{A}_0 
- \dm\sum_{\ell=0}^{k-1} \presub{(s)}\wt{\vc{U}}_{k,\ell} \vc{A}_{\ell-k}
\right)^{-1},
& k &\in \bbZ_{s-1},
\label{defn-(s)R_k^*-GIM1}
\end{align}
Furthermore, substituting (\ref{correspondence-Q}) into
(\ref{defn-(n_s)wt{U}_{k,l}}) yields the following: For $s \in \bbN$,
\begin{equation}
\presub{(s)}\wt{\vc{U}}_{s,\ell}
= (\vc{B}_1 \presub{(s)}\wt{\vc{U}}_{s-1}^{\ast})
\presub{(s)}\wt{\vc{U}}_{s-1,\ell}, 
\qquad  \ell \in \bbZ_{s-1},
\label{defn-(s)R_{0,l}-GIM1}
\end{equation}
and, for $k=1,2,\dots,s-1$,
\begin{eqnarray}
\presub{(s)}\wt{\vc{U}}_{k,\ell}
= (\vc{A}_1 \presub{(s)}\wt{\vc{U}}_{k-1}^{\ast}) 
   (\vc{A}_1 \presub{(s)}\wt{\vc{U}}_{k-2}^{\ast}) 
\cdots 
   (\vc{A}_1 \presub{(s)}\wt{\vc{U}}_{\ell}^{\ast}), 
\qquad \ell \in \bbZ_{k-1}.
\label{defn-(s)R_{k,l}-GIM1}
\end{eqnarray}
Since $\wt{\vc{U}}_0^{\ast} = (-\vc{A}_0)^{-1}$, we can prove by
induction that $\presub{(s)}\wt{\vc{U}}_k^{\ast}$'s in
(\ref{defn-(s)R_k^*-GIM1}) and $\presub{(s)}\wt{\vc{U}}_{k,\ell}$'s in
(\ref{defn-(s)R_{k,l}-GIM1}) are independent of $s$.  To utilize this
fact, we introduce the notation:
\begin{eqnarray}
\wt{\vc{U}}_k^{\ast}
=
\left(
-\vc{A}_0 
- \dm\sum_{\ell=0}^{k-1} \wt{\vc{U}}_{k,\ell} \vc{A}_{\ell-k}
\right)^{-1},
 \qquad k \in \bbZ_+,
\label{defn-R_k^*}
\end{eqnarray}
where $\wt{\vc{U}}_{k,\ell}$'s, $k \in\bbN$, $\ell\in\bbZ_{k-1}$, are
given by
\begin{equation}
\wt{\vc{U}}_{k,\ell}
=
(\vc{A}_1 \wt{\vc{U}}_{k-1}^{\ast}) 
(\vc{A}_1 \wt{\vc{U}}_{k-2}^{\ast}) 
\cdots 
(\vc{A}_1 \wt{\vc{U}}_{\ell}^{\ast}).
\label{defn-R_{k,l}-GIM1}
\end{equation}
Using (\ref{defn-R_k^*}) and (\ref{defn-R_{k,l}-GIM1}), we rewrite
(\ref{defn-(s)R_0^*-GIM1})--(\ref{defn-(s)R_{0,l}-GIM1}) as
\begin{align}
\presub{(s)}\wt{\vc{U}}_s^{\ast}
&= \left( 
-\vc{B}_0 - \vc{B}_1 
\sum_{\ell=0}^{s-1}\wt{\vc{U}}_{s-1}^{\ast}\wt{\vc{U}}_{s-1,\ell}\vc{B}_{\ell-s}
\right)^{-1},
\label{defn-(s)R_0^*-GIM1-02}
\\
\presub{(s)}\wt{\vc{U}}_k^{\ast}
&=
\left(
-\vc{A}_0 
- \dm\sum_{\ell=0}^{k-1} \wt{\vc{U}}_{k,\ell} \vc{A}_{k-\ell}
\right)^{-1},
& k & \in \bbZ_{s-1},
\nonumber
\\
\presub{(s)}\wt{\vc{U}}_{s,\ell}
&= \vc{B}_1 \cdot \wt{\vc{U}}_{s-1}^{\ast}\wt{\vc{U}}_{s-1,\ell}, 
& \ell &\in \bbZ_{s-1},
\label{defn-(s)R_{0,l}-GIM1-02}
\end{align}
where $\wt{\vc{U}}_{\ell,\ell} = \vc{I}$ for $\ell \in \bbZ_+$. Note
here that (\ref{eqn-(n_s)pi_{0,l}-03}) holds in the present setting since the GI/M/1-type Markov chain is a special case of the lower BHMC. Thus, substituting (\ref{defn-(s)R_{0,l}-GIM1-02}) into
(\ref{eqn-(n_s)pi_{0,l}-03}), we readily obtain, for $s \in \bbN$,
\begin{eqnarray}
\presub{(s)}\wc{\vc{\pi}}_{0}
&=& 
{
\vc{\alpha}_0 \presub{(s)}\wt{\vc{U}}_s^{\ast}
\over 
\vc{\alpha}_0 \presub{(s)}\wt{\vc{U}}_s^{\ast}
\left(
\vc{e} + \vc{B}_1 \sum_{\ell=0}^{s-1} \wt{\vc{U}}_{s-1}^{\ast}\wt{\vc{U}}_{s-1,\ell} \vc{e}
\right)
}, 
\label{defn-wt{x}_{s,0}}
\\
\presub{(s)}\wc{\vc{\pi}}_{k}
&=&
{
\vc{\alpha}_0 \presub{(s)}\wt{\vc{U}}_s^{\ast}
\vc{B}_1 \cdot \wt{\vc{U}}_{s-1}^{\ast} \wt{\vc{U}}_{s-1,s-k}
\over 
\vc{\alpha}_0 \presub{(s)}\wt{\vc{U}}_s^{\ast}
\left(
\vc{e} + \vc{B}_1 
\sum_{\ell=0}^{s-1} \wt{\vc{U}}_{s-1}^{\ast}\wt{\vc{U}}_{s-1,\ell} \vc{e}
\right)
}, \quad k =1,2,\dots,s.\qquad
\label{defn-wt{x}_{s,k}}
\end{eqnarray}
In addition, Theorem~\ref{thm-limit-row-[n]F} (together with Definition~\ref{defn-simple}) shows that
$\lim_{s\to\infty}\|\presub{(s)}\wc{\vc{\pi}} - \vc{\pi}\| = 0$ and thus
the following result holds.
\begin{thm}\label{thm-MIP-GI-M-1}
If the ergodic generator $\vc{Q}$ is given by (\ref{defn-GIM1-type}),
then
\begin{eqnarray*}
\vc{\pi}_0
&=& \lim_{s\to\infty}
{
\vc{\alpha}_0 \presub{(s)}\wt{\vc{U}}_s^{\ast}
\over 
\vc{\alpha}_0 \presub{(s)}\wt{\vc{U}}_s^{\ast}
\left(
\vc{e} + \vc{B}_1 
\sum_{\ell=0}^{s-1} \wt{\vc{U}}_{s-1}^{\ast}\wt{\vc{U}}_{s-1,\ell} \vc{e}
\right)
}, 
\\
\vc{\pi}_k
&=& \lim_{s\to\infty}
{
\vc{\alpha}_0 \presub{(s)}\wt{\vc{U}}_s^{\ast}
\vc{B}_1 \cdot \wt{\vc{U}}_{s-1}^{\ast} \wt{\vc{U}}_{s-1,s-k}
\over 
\vc{\alpha}_0 \presub{(s)}\wt{\vc{U}}_s^{\ast}
\left(
\vc{e} + \vc{B}_1 \sum_{\ell=0}^{s-1} \wt{\vc{U}}_{s-1}^{\ast} \wt{\vc{U}}_{s-1,\ell} \vc{e}
\right)
}, \qquad k \in \bbN. \quad
\end{eqnarray*}
\end{thm}

\begin{rem}
The GI/M/1-type structure (\ref{defn-GIM1-type}) of $\vc{Q}$ implies
that $\sup_{i\in \bbZ_+} |q(i,i)| < \infty$, which leads to
$\sum_{i\in \bbZ_+}\pi(i)|q(i,i)|<\infty$, i.e.,
Assumption~\ref{assumpt-1} holds (see
Remark~\ref{assumpt-drift-condition}).
\end{rem}

Using Theorem~\ref{thm-MIP-GI-M-1}, we develop an algorithm for
computing the stationary distribution vector of the GI/M/1-type Markov
chain, which is performed in a similar way to
Algorithm~\ref{algo-upper}.  For $k \in \bbZ_+$, let
$\wt{\vc{u}}_k^{\ast}$ and $\wt{\vc{U}}_{k,\ell}^{\ast}$'s, $\ell \in
\bbZ_k$, denote
\begin{align}
&&&&
\wt{\vc{u}}_k^{\ast}
&=
\sum_{\ell=0}^k \wt{\vc{U}}_k^{\ast} \wt{\vc{U}}_{k,\ell} \vc{e}, &
k &\in \bbZ_+,&&&&
\label{defn-wt{u}_s^*}
\\
&&&&
\wt{\vc{U}}_{k,\ell}^{\ast}
&=
\wt{\vc{U}}_k^{\ast} \wt{\vc{U}}_{k,\ell}, &
k &\in \bbZ_+,\ \ell \in \bbZ_k,&&&&
\label{eqn-wt{U}_{s,k}^*}
\end{align}
Note that, since $\wt{\vc{U}}_{k,k} = \vc{I}$ for $k \in \bbZ_+$, we
have $\wt{\vc{u}}_0^{\ast} =
\wt{\vc{U}}_0^{\ast}\vc{e}=(-\vc{A}_0)^{-1}\vc{e}$ and
$\wt{\vc{U}}_{k,k}^{\ast} = \wt{\vc{U}}_k^{\ast}$ for $k \in \bbZ_+$.
Note also that (\ref{defn-wt{x}_{s,0}}) and (\ref{defn-wt{x}_{s,k}})
can be rewritten in terms of $\wt{\vc{u}}^{\ast}_{s-1}$, as follows:
\begin{align}
\presub{(s)}\wc{\vc{\pi}}_{0}
&=
{
\vc{\alpha}_0\presub{(s)}\wt{\vc{U}}_s^{\ast}
\over 
\vc{\alpha}_0 \presub{(s)}\wt{\vc{U}}_s^{\ast}
\left(\vc{e} + \vc{B}_1 \wt{\vc{u}}_{s-1}^{\ast} \right)
}, 
\label{defn-wt{x}_{s,0}-02}
\\
\presub{(s)}\wc{\vc{\pi}}_{k}
&=
{
\vc{\alpha}_0\presub{(s)}\wt{\vc{U}}_s^{\ast}
\vc{B}_1\wt{\vc{U}}_{s-1,s-k}^{\ast} 
\over 
\vc{\alpha}_0 \presub{(s)}\wt{\vc{U}}_s^{\ast}
\left(\vc{e} + \vc{B}_1 \wt{\vc{u}}_{s-1}^{\ast} \right)
}, & k &= 1,2,\dots,s,
\label{defn-wt{x}_{s,k}-02}
\end{align}
where $\presub{(s)}\wt{\vc{U}}_s^{\ast}$ is given by
(\ref{eqn-(s)wt{U}_s^*-02}) below (which follows from
(\ref{defn-(s)R_0^*-GIM1-02}) and (\ref{eqn-wt{U}_{s,k}^*})):
\begin{eqnarray}
\presub{(s)}\wt{\vc{U}}_s^{\ast}
&=& \left( 
-\vc{B}_0 - \vc{B}_1 
\sum_{\ell=0}^{s-1} \wt{\vc{U}}_{s-1,\ell}^{\ast} \vc{B}_{\ell-s}
\right)^{-1}, \qquad s \in \bbN.
\label{eqn-(s)wt{U}_s^*-02}
\end{eqnarray}

In what follows, we derive the recursion of $\{\wt{\vc{u}}_k^{\ast}\}$
and $\{\wt{\vc{U}}_{k,\ell}^{\ast}\}$.  From (\ref{defn-R_{k,l}-GIM1})
and (\ref{eqn-wt{U}_{s,k}^*}), we obtain
\begin{eqnarray}
\wt{\vc{U}}_{k,\ell}
= \vc{A}_1 \wt{\vc{U}}_{k-1}^{\ast} \wt{\vc{U}}_{k-1,\ell}
= \vc{A}_1 \wt{\vc{U}}^{\ast}_{k-1,\ell},
\qquad k \in \bbN,\ \ell \in \bbZ_{k-1}.
\label{eqn-wt{U}_{k,l}}
\end{eqnarray}
Applying (\ref{eqn-wt{U}_{k,l}}) to (\ref{defn-wt{u}_s^*}),
(\ref{eqn-wt{U}_{s,k}^*}) and (\ref{defn-R_k^*}) yields
\begin{align}
\wt{\vc{u}}^{\ast}_k
&= \wt{\vc{U}}_k^{\ast} 
\left( \vc{e} + \vc{A}_1 \wt{\vc{u}}^{\ast}_{k-1}
\right),& k &\in \bbN,
\label{recursion-wt{u}^*_s}
\\
\wt{\vc{U}}_{k,\ell}^{\ast}
&= \wt{\vc{U}}_{k}^{\ast}\vc{A}_1\wt{\vc{U}}_{k-1,\ell}^{\ast}, &
k &\in \bbN,\ \ell \in \bbZ_{k-1},
\label{recursion-wt{U}_{k,l}^*_s}
\end{align}
and
\begin{eqnarray}
\wt{\vc{U}}_k^{\ast}
&=&
\left(
-\vc{A}_0 
- \vc{A}_1 \dm\sum_{\ell=0}^{k-1} \wt{\vc{U}}_{k-1,\ell}^{\ast} \vc{A}_{\ell-k}
\right)^{-1}, \qquad k \in \bbN,
\label{eqn-wt{U}_k^*-02}
\end{eqnarray}
respectively.

We are now ready to present the algorithm for the GI/M/1-type Markov
chain, which is described in Algorithm~\ref{algo-GI-M-1} below.

{
\renewcommand{\baselinestretch}{0.8}\selectfont
\begin{algorithm}[H]\label{algo-GI-M-1}
\caption{GI/M/1-type Markov chain}
{\bf Input}: $\vc{Q}$, $\varepsilon \in (0,1)$, $\vc{\alpha}_0$.
\\ 
{\bf Output}: 
$\presub{(s)}\wc{\vc{\pi}} = (\presub{(s)}\wc{\vc{\pi}}_{0},\presub{(s)}\wc{\vc{\pi}}_{1},\dots,\presub{(s)}\wc{\vc{\pi}}_{s})$,
where $s \in \bbN$ is fixed when the iteration stops.
%
\smallskip
\begin{enumerate}
\item Set $s = 0$ and $\presub{(s)}\wc{\vc{\pi}} = \vc{0}$.
\item Compute 
\[
\wt{\vc{U}}_0^{\ast} = \wt{\vc{U}}_{0,0}^{\ast} = (-\vc{A}_0)^{-1},\qquad 
\wt{\vc{u}}_0^{\ast} = \wt{\vc{U}}_0^{\ast}\vc{e}.
\]
\item Iterate the following:
\begin{enumerate}
\item Increment $s$ by one.
\item Compute $\presub{(s)}\wt{\vc{U}}_s^{\ast}$ by
  (\ref{eqn-(s)wt{U}_s^*-02}).
\item Compute $\{\presub{(s)}\wc{\vc{\pi}}_{k};k\in\bbZ_s\}$ by
  (\ref{defn-wt{x}_{s,0}-02}) and (\ref{defn-wt{x}_{s,k}-02}).
\item If $\| \presub{(s)}\wc{\vc{\pi}} - \presub{(s-1)}\wc{\vc{\pi}} \| <
  \varepsilon$, then stop the iteration; otherwise go to step~(iii.e).
\item Compute $\wt{\vc{U}}_s^{\ast}\, (= \wt{\vc{U}}_{s,s}^{\ast})$,
  $\wt{\vc{u}}^{\ast}_s$ and $\{\wt{\vc{U}}_{s,\ell}^{\ast};\ell\in\bbZ_{s-1}\}$ by (\ref{eqn-wt{U}_k^*-02}),
  (\ref{recursion-wt{u}^*_s}) and (\ref{recursion-wt{U}_{k,l}^*_s}),
  where $k=s$; and then return to step~(iii.a).
\end{enumerate}
\end{enumerate}
\end{algorithm}
\renewcommand{\baselinestretch}{1.1}\selectfont
}

\appendix

\section{Proofs}

\subsection{Proof of (\ref{lim-(n)Q_{<=m}^*})}\label{proof-lim-(n)Q_{<=m}^*}

To avoid repeating the same phrase, fix $m \in \bbZ_+$ arbitrarily,
and let $n \in \bbZ_+ \setminus \bbZ_m$.  It follows from
(\ref{lim-(n)Q=Q}), (\ref{partition-Q}) and (\ref{partition-(n)ol{Q}})
that
\begin{align*}
\lim_{n\to\infty} \presub{[n]}\ol{\vc{Q}}_{\bbZ_m}
&=\vc{Q}_{\bbZ_m}, &
\lim_{n\to\infty} \presub{[n]}\ol{\vc{Q}}_{\bbZ_m,\ol{\bbZ}_m}
&=\vc{Q}_{\bbZ_m,\ol{\bbZ}_m},
\\
\lim_{n\to\infty} \presub{[n]}\ol{\vc{Q}}_{\ol{\bbZ}_m,\bbZ_m}
&=\vc{Q}_{\ol{\bbZ}_m,\bbZ_m}, &
\lim_{n\to\infty} \presub{[n]}\ol{\vc{Q}}_{\ol{\bbZ}_m}
&=\vc{Q}_{\ol{\bbZ}_m}.
\end{align*}
According to these limits together with (\ref{defn-Q_{<=m}}) and
(\ref{defn-(n)ol{Q}_{<=m}}), it suffices to show that
\begin{equation*}
\lim_{n\to\infty} (-\presub{[n]}\ol{\vc{Q}}_{\ol{\bbZ}_m})^{-1}
= (-\vc{Q}_{\ol{\bbZ}_m})^{-1}.
\label{lim-[-(n)ol{Q}_{>m}]^{-1}}
\end{equation*}
Note that $\bbZ_n \setminus\bbZ_m \nearrow Z_+ \setminus\bbZ_m$ as $n \to
\infty$ and that $\presub{[n]}\vc{Q}_{\ol{\bbZ}_m}$ is a principal submatrix of
the $Q$-matrix $\vc{Q}_{\ol{\bbZ}_m}$. Thus, we have (see \cite[Chapter~2,
  Proposition~2.14]{Ande91}), for $i,j \in \bbZ_+\setminus\bbZ_m$ and
$t > 0$,
\begin{equation}
\left[ \exp\left\{ \presub{[n]}\vc{Q}_{\ol{\bbZ}_m} t\right\} \right]_{i,j}
\nearrow
\left[ \exp\left\{ \vc{Q}_{\ol{\bbZ}_m} t\right\} \right]_{i,j} \quad  \mbox{as $n \to \infty$}.
\label{increasing-(n)p_{>m}^{(t)}}
\end{equation}
From (\ref{defn-(n)ol{Q}}), we also have
\begin{equation}
[\presub{[n]}\ol{\vc{Q}}]_{i,j}
= [\presub{[n]}\vc{Q}_{\ol{\bbZ}_m} ]_{i,j}
+ \sum_{\ell=n+1}^{\infty}q(i,\ell) \presub{[n]}\alpha(j),
\qquad
i,j\in\bbZ_n\setminus\bbZ_m,
\label{eqn-(n)ol{Q}}
\end{equation}
which leads to
\begin{eqnarray}
\left[ \exp\left\{ \presub{[n]}\ol{\vc{Q}}_{\ol{\bbZ}_m} t\right\} \right]_{i,j}
&\ge&
\left[ \exp\left\{ \presub{[n]}\vc{Q}_{\ol{\bbZ}_m} t\right\} \right]_{i,j},
\qquad
i,j\in\bbZ_n\setminus\bbZ_m,\ t > 0.
\label{ineqn-(n)ol{p}_{>m}^{(t)}-01}
\end{eqnarray}
In addition, for $i,j\in\bbZ_n\setminus\bbZ_m$ and $t > 0$,
\begin{eqnarray}
\lefteqn{
\left[ \exp\left\{ \presub{[n]}\vc{Q}_{\ol{\bbZ}_m} t\right\} \right]_{i,j}
}
\quad &&
\nonumber
\\
&=& \PP\left(Z(t) = j,
\sup_{u \in [0,t]} Z(u) \le n, \inf_{u \in [0,t]} Z(u) \ge m+1 \mid Z(0) = i
\right).
\label{eqn-170315-01}
\end{eqnarray}

We now define $\presub{[n]}\delta_{\ol{\bbZ}_m}^{(t)}\!(i)$, $i \in
\bbZ_n\setminus\bbZ_m$, $t > 0$, as
\begin{eqnarray}
\presub{[n]}\delta_{\ol{\bbZ}_m}^{(t)}\!(i) 
&=& \sum_{j=m+1}^{\infty} 
\left[ \exp\left\{ \vc{Q}_{\ol{\bbZ}_m} t\right\} \right]_{i,j} 
- 
\sum_{j=m+1}^n \left[ \exp\left\{ \presub{[n]}\vc{Q}_{\ol{\bbZ}_m} t\right\} \right]_{i,j}\nonumber
\\
&=& \PP\left(
\sup_{u \in [0,t]} Z(u) \ge n+1, \inf_{u \in [0,t]} Z(u) \ge m+1 \mid Z(0) = i
\right).
\label{defn-(n)delta^{(t)}(i)}
\end{eqnarray}
It then follows from (\ref{eqn-(n)ol{Q}}), (\ref{eqn-170315-01}),
(\ref{defn-(n)delta^{(t)}(i)}) and $\sum_{j=m+1}^n \presub{[n]}\alpha(j)
\le 1$ that, for $i,j \in \bbZ_n\setminus\bbZ_m$ and $t > 0$,
\begin{eqnarray*}
\left[ \exp\left\{ \presub{[n]}\ol{\vc{Q}}_{\ol{\bbZ}_m} t\right\} \right]_{i,j}
&\le& \left[ \exp\left\{ \presub{[n]}\vc{Q}_{\ol{\bbZ}_m} t\right\} \right]_{i,j} 
+ \presub{[n]}\delta_{\ol{\bbZ}_m}^{(t)}\!(i).
\end{eqnarray*}
Combining this and (\ref{ineqn-(n)ol{p}_{>m}^{(t)}-01}) yields, for
$i,j \in \bbZ_n\setminus\bbZ_m$ and $t > 0$,
\begin{equation}
\left[ \exp\left\{ \presub{[n]}\vc{Q}_{\ol{\bbZ}_m} t\right\} \right]_{i,j} 
\le
\left[ \exp\left\{ \presub{[n]}\ol{\vc{Q}}_{\ol{\bbZ}_m} t\right\} \right]_{i,j}
\le \left[ \exp\left\{ \presub{[n]}\vc{Q}_{\ol{\bbZ}_m} t\right\} \right]_{i,j} 
+ \presub{[n]}\delta_{\ol{\bbZ}_m}^{(t)}\!(i).
\label{ineqn-(n)ol{p}_{>m}^{(t)}}
\end{equation}
It also follows from (\ref{increasing-(n)p_{>m}^{(t)}}) and
(\ref{defn-(n)delta^{(t)}(i)}) that, for $i\in\bbZ_+$ and $t > 0$,
\begin{equation}
\presub{[n]}\delta_{\ol{\bbZ}_m}^{(t)}\!(i) 
\searrow
0 \quad  \mbox{as $n \to \infty$}.
\label{decreasing-(n)delta^{(t)}}
\end{equation}
Using (\ref{decreasing-(n)delta^{(t)}}) and the monotone convergence
theorem, we have
\begin{eqnarray*}
\lefteqn{
\int_0^{\infty} \presub{[m+1]}\delta_{\ol{\bbZ}_m}^{(t)}(i)\rd t
- 
\lim_{n\to\infty}\int_0^{\infty}
\presub{[n]}\delta_{\ol{\bbZ}_m}^{(t)}\!(i)
\rd t
}
\nonumber
\\
&=&
\lim_{n\to\infty}\int_0^{\infty}
\left\{ 
\presub{[m+1]}\delta_{\ol{\bbZ}_m}^{(t)}(i) - \presub{[n]}\delta_{\ol{\bbZ}_m}^{(t)}\!(i)
\right\} \rd t
\nonumber
\\
&=& \int_0^{\infty}
\presub{[m+1]}\delta_{\ol{\bbZ}_m}^{(t)}(i) \rd t,
\end{eqnarray*}
which yields
\begin{equation}
\lim_{n\to\infty}\int_0^{\infty}
\presub{[n]}\delta_{\ol{\bbZ}_m}^{(t)}\!(i)
\rd t = 0,\qquad i\in\bbZ_+,\ t > 0.
\label{lim_int_(n)delta_{>m}^{(t)}(i)}
\end{equation}
Furthermore, using (\ref{increasing-(n)p_{>m}^{(t)}})
(\ref{ineqn-(n)ol{p}_{>m}^{(t)}}),
(\ref{lim_int_(n)delta_{>m}^{(t)}(i)}) and the monotone convergence
theorem, we obtain
\begin{eqnarray*}
\lim_{n\to\infty}\int_0^{\infty} \exp\{ \presub{[n]}\ol{\vc{Q}}_{\ol{\bbZ}_m} t \}  \rd t
&=& \lim_{n\to\infty}\int_0^{\infty} \exp\{ \presub{[n]}\vc{Q}_{\ol{\bbZ}_m} t \} \rd t
= \int_0^{\infty}\exp\{ \vc{Q}_{\ol{\bbZ}_m} t \} \rd t,
\end{eqnarray*}
which leads to
\begin{eqnarray*}
\lim_{n\to\infty} (- \presub{[n]}\ol{\vc{Q}}_{\ol{\bbZ}_m} )^{-1}
&=& \lim_{n\to\infty} (- \presub{[n]}\vc{Q}_{\ol{\bbZ}_m} )^{-1} = (- \vc{Q}_{\ol{\bbZ}_m})^{-1}.
\end{eqnarray*}
The proof is completed.

\subsection{Proof of (\ref{eqn-pi_l-pi_k*U_{k,l}})}\label{proof-pi_l-pi_k*U_{k,l}}

It follows from (\ref{defn-Q-MG1-type}) and
(\ref{partition-(n_s)Q-upper}) that the first $k$ blocks of
$\vc{\pi}\vc{Q}=\vc{0}$ are given by
\begin{eqnarray*}
&&
(\vc{\pi}_0,\vc{\pi}_1,\dots,\vc{\pi}_{k-1})\presub{(k-1)}\vc{Q}
 {} + 
(\vc{\pi}_k,\vc{\pi}_{k+1},\dots)
\left(
\begin{array}{cccc}
\vc{O} &  \cdots & \vc{O} & \vc{Q}_{k,k-1}
\\
\vc{O} &  \cdots & \vc{O} & \vc{O}
\\
\vdots &  \ddots & \vdots & \vdots
\\
\vc{O} &  \cdots & \vc{O} & \vc{O}
\end{array}
\right) = \vc{0},
\end{eqnarray*}
and thus
\begin{eqnarray*}
\lefteqn{
(\vc{\pi}_0,\vc{\pi}_1,\dots,\vc{\pi}_{k-1})
}
\quad &&
\nonumber
\\
&=& (\vc{\pi}_k,\vc{\pi}_{k+1},\dots)
\left(
\begin{array}{cccc}
\vc{O} &  \cdots & \vc{O} & \vc{Q}_{k,k-1}
\\
\vc{O} &  \cdots & \vc{O} & \vc{O}
\\
\vdots &  \ddots & \vdots & \vdots
\\
\vc{O} &  \cdots & \vc{O} & \vc{O}
\end{array}
\right)
(- \presub{(k-1)}\vc{Q} )^{-1}
\nonumber
\\
&=& (\vc{O},\dots,\vc{O}, \vc{\pi}_k\vc{Q}_{k,k-1})
(- \presub{(k-1)}\vc{Q} )^{-1}
\nonumber
\\
&=& \vc{\pi}_k\vc{Q}_{k,k-1}
(\presub{(k-1)}\vc{X}_{k-1,0},\presub{(k-1)}\vc{X}_{k-1,1},
\dots
\presub{(k-1)}\vc{X}_{k-1,k-1}),
\end{eqnarray*}
where the last equality follows from
(\ref{defn-(-Q)^{-1}-MG1-type}). Therefore,
\[
\vc{\pi}_{\ell}
= \vc{\pi}_k\vc{Q}_{k,k-1} \presub{(k-1)}\vc{X}_{k-1,\ell},
\qquad \ell \in \bbZ_{k-1}.
\]
Applying Lemma~\ref{lem-(n_s)X_{s,l}} and
(\ref{defn-U_{k,l}}) to the above equality, we have
\[
\vc{\pi}_{\ell}
= \vc{\pi}_k\vc{Q}_{k,k-1} \vc{U}_{k-1}^{\ast}\vc{U}_{k-1,\ell}
= \vc{\pi}_k \vc{U}_{k,\ell},
\qquad \ell \in \bbZ_{k-1},
\]
which shows that (\ref{eqn-pi_l-pi_k*U_{k,l}}) holds.

\section{Discussion on matrix $\vc{U}_k^{\ast}$}

\subsection{Nonsingularity of $\vc{U}_k^{\ast}$}\label{appen-T_k^*}

Let $\vc{T}_k^{\ast}$, $k\in\bbZ_+$, denote
\begin{eqnarray}
\vc{T}_k^{\ast}
&=& \vc{Q}_{k,k} 
+ \sum_{\ell=0}^{k-1} \vc{U}_{k,\ell}\vc{Q}_{\ell,k},
\qquad k\in\bbZ_+.
\label{defn-T_k^*}
\end{eqnarray} 
Substituting (\ref{defn-T_k^*}) into (\ref{defn-U_k^*}) yields
$\vc{U}_k^{\ast} = (-\vc{T}_k^{\ast})^{-1}$ for $k \in
\bbZ_+$. Furthermore, (\ref{defn-T_k^*}) shows $\vc{T}_0^{\ast} =
\vc{Q}_{0,0}$ since the empty sum is defined as zero. Note here that
$\vc{Q}_{0,0}$ is the $(0,0)$-th block, i.e., the zero-th diagonal
block of the partitioned ergodic generator $\vc{Q}$ in
(\ref{partitioned-Q}), which implies that $\vc{T}_0^{\ast} =
\vc{Q}_{0,0}$ is nonsingular. In what follows, we prove by induction
the nonsingularity of $\vc{T}_k^{\ast}$ for $k \in \bbN$.

We suppose that there exists some $k \in \bbN$ such that, for all $m
\in \bbZ_{k-1}$, $\vc{T}_m^{\ast}$ is nonsingular and thus
$\vc{U}_m^{\ast} = (-\vc{T}_m^{\ast})^{-1}$ is well-defined. We then
partition $\presub{(k)}\vc{Q}$ as
\begin{equation}
\presub{(k)}\vc{Q}
=\left(
\begin{array}{cccc|c}
 				&
 				&
 				&
				&
\vc{Q}_{0,k} 
\\
 				&
 				&
 \mbox{\down{3mm}{$\presub{(k-1)}\vc{Q}$}} 				&
				&
\vc{Q}_{1,k} 	
\\
 				&
 				&
 				&
				&
\vdots
\\
 				&
 				&
 				&
				&
\vc{Q}_{k-1,k} 	 	
\\
\hline
\vc{O}			&
~~\cdots \hspace{-2mm}		&
\vc{O}			&
\vc{Q}_{k,k-1}	&
\vc{Q}_{k,k}
\end{array}
\right).
\label{partition-(n_k)Q}
\end{equation}
Since the generator $\vc{Q}$ is ergodic, its diagonal blocks
$\vc{Q}_{k,k}$ and $\presub{(k-1)}\vc{Q}$ are
nonsingular. Furthermore, $\presub{(k)}\vc{X}_{k,k}$, i.e., the
$(k,k)$-th block of $(-\presub{(k)}\vc{Q})^{-1}$, is given by (see,
e.g., \cite[Section 0.7.3]{Horn90})
\begin{eqnarray}
\presub{(k)}\vc{X}_{k,k}
&=& \left[
- \vc{Q}_{k,k}
- (\vc{O},\dots,\vc{O},\vc{Q}_{k,k-1})
(-\presub{(k-1)}\vc{Q})^{-1}
\left(
\begin{array}{c}
\vc{Q}_{0,k} 
\\
\vc{Q}_{1,k} 
\\
\vdots
\\
\vc{Q}_{k-1,k} 
\end{array}
\right)
\right]^{-1}
\nonumber
\\
&=& \left(
- \vc{Q}_{k,k}
- \vc{Q}_{k,k-1} \sum_{\ell=0}^{k-1} \presub{(k-1)}\vc{X}_{k-1,\ell}
\vc{Q}_{\ell,k} 
\right)^{-1},
\label{eqn-(n_k)X_{k,k}}
\end{eqnarray}
where the second equality follows from (\ref{defn-(-Q)^{-1}-MG1-type}); more specifically, the fact that the last block row of
$(-\presub{(k-1)}\vc{Q})^{-1}$ is equal to
\[
(\presub{(k-1)}\vc{X}_{k-1,0},\presub{(k-1)}\vc{X}_{k-1,1},\dots,\presub{(k-1)}\vc{X}_{k-1,k-1}).
\]
Applying
Lemma~\ref{lem-(n_s)X_{s,l}} to (\ref{eqn-(n_k)X_{k,k}}) and using
(\ref{defn-U_{k,l}}) and (\ref{defn-T_k^*}) yields
\begin{eqnarray}
\presub{(k)}\vc{X}_{k,k}
&=& \left(
- \vc{Q}_{k,k}
- \vc{Q}_{k,k-1}\vc{U}_{k-1}^{\ast}
\sum_{\ell=0}^{k-1} \vc{U}_{k-1,\ell}\vc{Q}_{\ell,k} 
\right)^{-1}
\nonumber
\\
&=& \left(
- \vc{Q}_{k,k}
- \sum_{\ell=0}^{k-1} \vc{U}_{k,\ell}\vc{Q}_{\ell,k} 
\right)^{-1}
= (-\vc{T}_k^{\ast})^{-1}.
\label{eqn-(n_k)X_{k,k}-02}
\end{eqnarray}
As a
result, we have proved by induction that $\vc{T}_k^{\ast}$ is
nonsingular for all $k \in \bbZ_+$.

\begin{rem}\label{rem-nonnegative-U_k^*}
Since $(-\presub{(k)}\vc{Q})^{-1}$ is nonnegative, it follows from (\ref{eqn-(n_k)X_{k,k}-02}) and $\vc{U}_k^{\ast} = (-\vc{T}_k^{\ast})^{-1}$ that $\presub{(k)}\vc{X}_{k,k} = (-\vc{T}_k^{\ast})^{-1} = \vc{U}_k^{\ast} \ge \vc{O}$.
In addition, since
the original generator $\vc{Q}$ is ergodic, its diagonal elements are all finite and negative. Therefore, all the diagonal elements of $(-\presub{(k)}\vc{Q})^{-1}$ is positive. This fact implies that
\[
\presub{(k)}\vc{X}_{k,k}\vc{e}
= (-\vc{T}_k^{\ast})^{-1}\vc{e} 
= \vc{U}_k^{\ast}\vc{e} 
> \vc{0}.
\]
\end{rem}

\subsection{Computation of $\vc{U}_k^{\ast}$}\label{appen-comp-U_k^*}

In this subsection, we discuss the computation of
$\vc{U}_k^{\ast}=(-\vc{T}_k^{\ast})^{-1}$. We begin with the
following lemma.
\begin{lem}\label{lem-T_k^*}
For $k \in \bbZ_+$, the matrix $\vc{T}_k^{\ast}$ is a $Q$-matrix,
i.e., all the nondiagonal elements of $\vc{T}_k^{\ast}$ are
nonnegative and $\vc{T}_k^{\ast} \vc{e} \le \vc{0}$.
\end{lem}

\medskip

\proof From (\ref{eqn-(n_k)X_{k,k}}), (\ref{eqn-(n_k)X_{k,k}-02}) and
(\ref{defn-(-Q)^{-1}-MG1-type}), we have
\begin{eqnarray}
\vc{T}_k^{\ast} 
&=& (-\presub{(k)}\vc{X}_{k,k})^{-1}
\nonumber
\\
&=&  \vc{Q}_{k,k}
+ (\vc{O},\dots,\vc{O},\vc{Q}_{k,k-1})
(-\presub{(k-1)}\vc{Q})^{-1}
\left(
\begin{array}{c}
\vc{Q}_{0,k} 
\\
\vc{Q}_{1,k} 
\\
\vdots
\\
\vc{Q}_{k-1,k} 
\end{array}
\right),
\label{eqn-T_k^*}
\end{eqnarray}
where $\vc{Q}_{k,k}$ is a $Q$-matrix and the second term of
(\ref{eqn-T_k^*}) is nonnegative. Therefore, it suffices to show
$\vc{T}_k^{\ast} \vc{e} \le \vc{0}$. It follows from
(\ref{partition-(n_k)Q}) and $\presub{(k)}\vc{Q}\vc{e} \le \vc{0}$
that
\[
\presub{(k-1)}\vc{Q}\vc{e} + \left(
\begin{array}{c}
\vc{Q}_{0,k} 
\\
\vc{Q}_{1,k} 
\\
\vdots
\\
\vc{Q}_{k-1,k} 
\end{array}
\right) \vc{e}
\le \vc{0},
\]
which leads to
\[
(-\presub{(k-1)}\vc{Q})^{-1}
\left(
\begin{array}{c}
\vc{Q}_{0,k} 
\\
\vc{Q}_{1,k} 
\\
\vdots
\\
\vc{Q}_{k-1,k} 
\end{array}
\right) \vc{e}
\le \vc{e}.
\]
Using this inequality and (\ref{eqn-T_k^*}), we obtain
\begin{eqnarray*}
\vc{T}_k^{\ast}\vc{e}
&=&
\vc{Q}_{k,k} \vc{e}
+ 
(\vc{O},\dots,\vc{O},\vc{Q}_{k,k-1})(-\presub{(k-1)}\vc{Q})^{-1}
\left(
\begin{array}{c}
\vc{Q}_{0,k} 
\\
\vc{Q}_{1,k} 
\\
\vdots
\\
\vc{Q}_{k-1,k} 
\end{array}
\right) \vc{e}
\nonumber
\\
&\le& \vc{Q}_{k,k} \vc{e} +  \vc{Q}_{k,k-1}\vc{e} \le 
\sum_{\ell=0}^{\infty}\vc{Q}_{k,\ell} \vc{e} = \vc{0},
\end{eqnarray*}
where the last inequality follows from
$\sum_{\ell\in\bbZ_+}\vc{Q}_{k,\ell} \vc{e} = \vc{0}$ and
$\sum_{\ell\in\bbZ_+\setminus\{k\}}\vc{Q}_{k,\ell} \vc{e} \ge \vc{0}$.
The statement of the present lemma has been proved.  \qed

\medskip

We now define $\vc{P}_k^{\ast}$, $k\in\bbZ_+$ as
\begin{equation}
\vc{P}_k^{\ast}
= \vc{I} + \vc{T}_k^{\ast} / \theta_k,
\label{defn-P_k^*}
\end{equation}
where $\theta_k$ denotes the maximum of the absolute values of the
diagonal elements of $\vc{T}_k^{\ast}$.  It follows from
Lemma~\ref{lem-T_k^*} and the nonsingularity of $\vc{T}_k^{\ast}$ that
$\vc{P}_k^{\ast}$ is strictly substochastic, i.e., $\vc{P}_k^{\ast}
\ge \vc{O}$, $\vc{P}_k^{\ast}\vc{e} \le \vc{e}, \neq \vc{e}$ and ${\rm
  sp}(\vc{P}_k^{\ast}) < 1$, where ${\rm sp}(\vc{P}_k^{\ast})$ denotes
the spectral radius of $\vc{P}_k^{\ast}$. Thus, from
(\ref{defn-P_k^*}), we have
\begin{equation}
(-\vc{T}_k^{\ast})^{-1}
= \theta_k^{-1}(\vc{I} - \vc{P}_k^{\ast} )^{-1}
= \theta_k^{-1} \sum_{m=0}^{\infty} (\vc{P}_k^{\ast})^m \ge \vc{O}, \neq \vc{O}.
\label{eqn-(-T_k)^{-1}}
\end{equation}
According to (\ref{eqn-(-T_k)^{-1}}), we can obtain
$(-\vc{T}_k^{\ast})^{-1}$ approximately by computing
$\vc{P}_k^{\ast},(\vc{P}_k^{\ast})^2,\dots,\break (\vc{P}_k^{\ast})^M$
for sufficiently large $M \in \bbN$ and summing them up. However, Le
Boudec~\cite{Le-Boud91} proposes a more efficient algorithm for
computing $(-\vc{T}_k^{\ast})^{-1}$, which is based on the following
proposition.
\begin{prop}[{}{\cite[Proposition 1]{Le-Boud91}}]\label{prop-Le-Boud91}
Let $\{\vc{V}_n;n\in\bbZ_+\}$ and $\{\vc{W}_n;n\in\bbZ_+\}$ denote
sequences of matrices such that
\begin{eqnarray}
\vc{V}_n
&=&
\left\{
\begin{array}{ll}
\vc{P}_k^{\ast}, & \quad n = 0,
\\
(\vc{V}_{n-1})^2, & \quad n \in \bbN,
\end{array}
\right.
\label{defn-V_n}
\\
\vc{W}_n
&=&
\left\{
\begin{array}{ll}
\vc{I}, & \quad n = 0,
\\
(\vc{I} + \vc{V}_{n-1})\vc{W}_{n-1}, & \quad n \in \bbN.
\end{array}
\right.
\label{defn-W_n}
\end{eqnarray}
It then holds that
\begin{eqnarray*}
\lim_{n\to\infty}\vc{W}_n &=& (\vc{I} - \vc{P}_k^{\ast} )^{-1}.
\end{eqnarray*}
\end{prop}


It follows from (\ref{defn-V_n}) and (\ref{defn-W_n}) that $\vc{W}_n =
\sum_{m=0}^{2^n-1} (\vc{P}_k^{\ast})^m$. Therefore, Le Boudec's
algorithm~\cite{Le-Boud91} logarithmically reduces the number of
iterations for computing $\sum_{m=0}^M (\vc{P}_k^{\ast})^m$.

\section*{Acknowledgments}
The author thanks Mr.\ Masatoshi Kimura and Dr.\ Tetsuya Takine for
their invaluable comments on the convergence of the limit formula
(\ref{limit-formula-(n)F}).  The author also thanks Dr.\ Tetsuya
Takine for sharing an early version of \cite{Taki16}. In addition, the
author acknowledges stimulating discussions on
Algorithm~\ref{algo-lower} with Kazuya Fukuoka.

%
%
%
%
\bibliographystyle{plain} 
\bibliography{hm-180925}

\begin{thebibliography}{10}

\bibitem{Ande91}
W.~J. Anderson.
\newblock {\em Continuous-Time Markov Chains: An Applications-Oriented
  Approach}.
\newblock Springer, New York, 1991.

\bibitem{Baum12-Procedia}
H.~Baumann and W.~Sandmann.
\newblock Numerical solution of level dependent quasi-birth-and-death
  processes.
\newblock {\em Procedia Computer Science}, 1(1):1561--1569, 2012.

\bibitem{Baum12-COR}
H.~Baumann and W.~Sandmann.
\newblock Steady state analysis of level dependent quasi-birth-and-death
  processes with catastrophes.
\newblock {\em Computers \& Operations Research}, 39(2):413--423, 2012.

\bibitem{Brem99}
P.~Br\'{e}maud.
\newblock {\em Markov Chains: Gibbs Fields, Monte Carlo Simulation, and
  Queues}.
\newblock Springer, New York, 1999.

\bibitem{Brig95}
L.~Bright and P.~G. Taylor.
\newblock Calculating the equilibrium distribution in level dependent
  quasi-birth-and-death processes.
\newblock {\em Stochastic Models}, 11(3):497--525, 1995.

\bibitem{Gibs87-JAP}
D.~Gibson and E.~Seneta.
\newblock Augmented truncations of infinite stochastic matrices.
\newblock {\em Journal of Applied Probability}, 24(3):600--608, 1987.

\bibitem{Gras00-BC}
W.~K. Grassmann and D.~A. Stanford.
\newblock Matrix analytic methods.
\newblock In W.~K. Grassmann, editor, {\em Computational Probability},
  chapter~6, pages 153--203. Kluwer Academic Publisher, Boston, 2000.

\bibitem{Hart12}
A.~G. Hart and R.~L. Tweedie.
\newblock Convergence of invariant measures of truncation approximations to
  {Markov processes}.
\newblock {\em Applied Mathematics}, 3(12A):2205--2215, 2012.

\bibitem{Heid10}
B.~Heidergott, A.~Hordijk, and N.~Leder.
\newblock Series expansions for continuous-time {Markov} processes.
\newblock {\em Operations Research}, 58(3):756--767, 2010.

\bibitem{Horn90}
R.~A. Horn and C.~R. Johnson.
\newblock {\em Matrix Analysis}.
\newblock Cambridge University Press, Cambridge, {P}aperback edition, 1990.

\bibitem{Keme76}
J.~G. Kemeny, J.~L. Snell, and A.~W. Knapp.
\newblock {\em Denumerable Markov Chains}.
\newblock {Springer}, New York, 2nd edition, 1976.

\bibitem{M.Kimu18}
M.~Kimura and T.~Takine.
\newblock Computing the conditional stationary distribution in {Markov} chains
  of level-dependent {M/G/1}-type.
\newblock {\em Stochastic Models}, 34(2):207--238, 2018.

\bibitem{Klim06}
V.~Klimenok and A.~Dudin.
\newblock Multi-dimensional asymptotically quasi-{Toeplitz} {Markov} chains and
  their application in queueing theory.
\newblock {\em Queueing Systems}, 54(4):245--259, 2006.

\bibitem{Kont16}
I.~Kontoyiannis and S.~P. Meyn.
\newblock On the $f$-norm ergodicity of {Markov} processes in continuous time.
\newblock {\em Electronic Communications in Probability}, 21:paper no. 77,
  1--10, 2016.

\bibitem{Lato93}
G.~Latouche and V.~Ramaswami.
\newblock A logarithmic reduction algorithm for quasi-birth-death processes.
\newblock {\em Journal of Applied Probability}, 30(3):650--674, 1993.

\bibitem{Lato99}
G.~Latouche and V.~Ramaswami.
\newblock {\em Introduction to Matrix Analytic Methods in Stochastic Modeling}.
\newblock SIAM, Philadelphia, PA, 1999.

\bibitem{Le-Boud91}
J.-Y. {Le Boudec}.
\newblock An efficient solution method for {Markov} models of {ATM} links with
  loss priorities.
\newblock {\em IEEE Journal on Selected Areas in Communications},
  9(3):408--417, 1991.

\bibitem{LiQuan05-RG}
Q.-L. Li, Z.~Lian, and L.~Liu.
\newblock An {$RG$}-factorization approach for a {BMAP/M/1} generalized
  processor-sharing queue.
\newblock {\em Stochastic Models}, 21(2--3):507--530, 2005.

\bibitem{LiuYuan10}
Y.~Liu.
\newblock Augmented truncation approximations of discrete-time {Markov} chains.
\newblock {\em Operations Research Letters}, 38(3):218--222, 2010.

\bibitem{LiuYuan15}
Y.~Liu.
\newblock Perturbation analysis for continuous-time {Markov} chains.
\newblock {\em Science China Mathematics}, 58(12):2633--2642, 2015.

\bibitem{Masu15-ADV}
H.~Masuyama.
\newblock Error bounds for augmented truncations of discrete-time
  block-monotone {Markov} chains under geometric drift conditions.
\newblock {\em Advances in Applied Probability}, 47(1):83--105, 2015.

\bibitem{Masu16-SIAM}
H.~Masuyama.
\newblock Error bounds for augmented truncations of discrete-time
  block-monotone {Markov} chains under subgeometric drift conditions.
\newblock {\em SIAM Journal on Matrix Analysis and Applications},
  37(3):877--910, 2016.

\bibitem{Masu17-LAA}
H.~Masuyama.
\newblock Continuous-time block-monotone {Markov} chains and their
  block-augmented truncations.
\newblock {\em Linear Algebra and its Applications}, 514(1):105--150, 2017.

\bibitem{Masu17-JORSJ}
H.~Masuyama.
\newblock Error bounds for last-column-block-augmented truncations of
  block-structured {Markov} chains.
\newblock {\em Journal of the Operations Research Society of Japan},
  60(3):271--320, 2017.

\bibitem{Meyn93-III}
S.~P. Meyn and R.~L. Tweedie.
\newblock Stability of {Markovian} processes {III}: {Foster-Lyapunov} criteria
  for continuous-time processes.
\newblock {\em Advances in Applied Probability}, 25(3):518--548, 1993.

\bibitem{Meyn09}
S.~P. Meyn and R.~L. Tweedie.
\newblock {\em Markov Chains and Stochastic Stability}.
\newblock Cambridge University Press, Cambridge, 2nd edition, 2009.

\bibitem{Neut89}
M.~F. Neuts.
\newblock {\em Structured Stochastic Matrices of {M/G/1} Type and Their
  Applications}.
\newblock Marcel Dekker, New York, 1989.

\bibitem{Phun10-QTNA}
T.~Phung-Duc, H.~Masuyama, S.~Kasahara, and Y.~Takahashi.
\newblock A simple algorithm for the rate matrices of level-dependent {QBD}
  processes.
\newblock In {\em Proceedings of the 5th International Conference on Queueing
  Theory and Network Applications (QTNA2010)}, pages 46--52, New York, 2010.
  ACM.

\bibitem{Rama96}
V.~Ramaswami and P.~G. Taylor.
\newblock Some properties of the rate operators in level dependent
  quasi-birth-and-death processes with a countable number of phases.
\newblock {\em Stochastic Models}, 12(1):143--164, 1996.

\bibitem{Sene06}
E.~Seneta.
\newblock {\em Non-negative Matrices and {Markov} Chains}.
\newblock Springer, New York, {R}evised {P}rinting edition, 2006.

\bibitem{Shin09}
Y.~W. Shin.
\newblock Fundamental matrix of transient {QBD} generator with finite states
  and level dependent transitions.
\newblock {\em Asia-Pacific Journal of Operational Research}, 26(5):697--714,
  2009.

\bibitem{Shin98}
Y.~W. Shin and C.~E.~M. Pearce.
\newblock An algorithmic approach to the {Markov} chain with transition
  probability matrix of upper block-{Hessenberg} form.
\newblock {\em Korean Journal of Computational \& Applied Mathematics},
  5(2):361--384, 1998.

\bibitem{Taki16}
T.~Takine.
\newblock Analysis and computation of the stationary distribution in a special
  class of {Markov} chains of level-dependent {M/G/1}-type and its application
  to {BMAP/M/$\infty$} and {BMAP/M/$c+M$} queues.
\newblock {\em Queueing Systems}, 84(1--2):49--77, 2016.

\bibitem{Tijm03}
H.~C. Tijms.
\newblock {\em A First Course in Stochastic Models}.
\newblock John Wiley \& Sons, Chichester, UK, 2003.

\bibitem{Twee98}
R.~L. Tweedie.
\newblock Truncation approximations of invariant measures for {Markov} chains.
\newblock {\em Journal of Applied Probability}, 35(3):517--536, 1998.

\bibitem{Wolf80}
D.~Wolf.
\newblock Approximation of the invariant probability measure of an infinite
  stochastic matrix.
\newblock {\em Advances in Applied Probability}, 12(3):710--726, 1980.

\end{thebibliography}

\end{document}